\documentclass[a4paper]{amsart}
\usepackage[latin1]{inputenc}
\usepackage{amsmath}
\usepackage{amsfonts}
\usepackage{amssymb}
\usepackage{graphicx}
\usepackage{tikz}
\numberwithin{equation}{section}
\newtheorem{theorem}{Theorem}[section]
\newtheorem{lemma}[theorem]{Lemma}
\newtheorem{proposition}[theorem]{Proposition}
\newtheorem{corollary}[theorem]{Corollary}
\theoremstyle{definition}

\newtheorem{example}[theorem]{Example}

\theoremstyle{remark}
\newtheorem{remark}[theorem]{\bf{Remark}}

\newcommand{\rcross}{{\triangleright\!\!\!<}}   
\newcommand{\lcross}{{>\!\!\!\triangleleft}}

\newcommand{\lbiprod}{{>\!\!\!\triangleleft\kern-.33em\cdot}}
\newcommand{\rbiprod}{{\cdot\kern-.33em\triangleright\!\!\!<}}
\newcommand{\rcocross}{{\blacktriangleright\!\!<}}
\newcommand{\lcocross}{{>\!\!\!\blacktriangleleft}}

\newcommand{\dotop}{\cdot_{\mathrm{\underline{op}}}}
\newcommand{\delcop}{\underline{\Delta}_{\mathrm{\underline{cop}}}}

\newcommand{\cop}{{\mathrm{\underline{cop}}}}
\newcommand{\op}{{\mathrm{\underline{op}}}}

\makeatletter
\newcommand{\mathleft}{\@fleqntrue\@mathmargin0pt}
\newcommand{\mathcenter}{\@fleqnfalse}
\makeatother

\newcommand{\extd}{{\rm d}}

\renewcommand{\o}{{}_{\scriptscriptstyle(1)}} 
\renewcommand{\t}{{}_{\scriptscriptstyle(2)}}
\renewcommand{\th}{{}_{\scriptscriptstyle(3)}}
\newcommand{\fo}{{}_{\scriptscriptstyle(4)}}
\newcommand{\fiv}{{}_{\scriptscriptstyle(5)}}
\newcommand{\six}{{}_{\scriptscriptstyle(6)}}
\newcommand{\sev}{{}_{\scriptscriptstyle(7)}}
\newcommand{\ei}{{}_{\scriptscriptstyle(8)}}
\newcommand{\nine}{{}_{\scriptscriptstyle(9)}}

\newcommand{\infi}{{}_{\scriptscriptstyle(\infty)}}
\newcommand{\mo}{{\scriptscriptstyle -(1)}}
\newcommand{\mt}{{\scriptscriptstyle -(2)}}
\newcommand{\z}{{}_{\scriptscriptstyle(0)}}
\newcommand{\cg}{{\mathfrak g}}
\newcommand{\cc}{{\mathfrak c}}
\newcommand{\cu}{{\mathfrak u}}
\newcommand{\C}{{\Bbb C}}

\newcommand{\Z}{{\Bbb Z}}

\newcommand{\kk}{k}
\newcommand{\eps}{\epsilon}
\newcommand{\CR}{{\mathcal R}}
\newcommand{\CC}{{\mathcal C}}
\newcommand{\CF}{{\mathcal F}}
\newcommand{\CM}{{\mathcal M}}
\newcommand{\tens}{\otimes}
\newcommand{\id}{{\rm id}}

\newcommand{\atild}{\tilde{a}}
\newcommand{\btild}{\tilde{b}}
\newcommand{\ctild}{\tilde{c}}
\newcommand{\dtild}{\tilde{d}}

\renewcommand{\>}{{\rangle}}
\newcommand{\la}{{\triangleright}}
\newcommand{\ra}{{\triangleleft}}

\allowdisplaybreaks

\title{Co-double bosonisation and dual bases of $c_{q}[SL_2]$ and $c_q[SL_3]$}
\author{Ryan Aziz and Shahn Majid}
\address{School of Mathematical Sciences\\ Queen Mary University of London \\ Mile End Rd, London E1 4NS }
\email{r.aziz@qmul.ac.uk, s.majid@qmul.ac.uk}
\thanks{LPDP Indonesian Endowment Fund For Education}
\begin{document}
\maketitle

\begin{abstract} We find a dual version of a previous double-bosonisation theorem whereby each finite-dimensional  braided-Hopf algebra $B$ in the category of comodules of a coquasitriangular Hopf algebra $A$  has an associated coquasitriangular Hopf algebra $B^{\underline{\rm op}}\lbiprod A \rbiprod B^*$. As an application we find new generators for $c_q[SL_{2}]$ reduced at $q$ a primitive odd root of unity with the remarkable property that  their monomials are essentially a dual basis to the standard PBW basis of the reduced Drinfeld-Jimbo quantum enveloping algebra $u_q(sl_{2})$. Our methods apply in principle for general $c_{q}[G]$ as we demonstrate for $c_q[SL_3]$ at certain odd roots of unity.
\end{abstract}

\section{Introduction}
  
Quantum groups such as  $U_q(\cg)$ associated to complex semisimple Lie algebras \cite{Dri,Jimbo}, and their finite-dimensional quotients $u_q(\cg)$ at $q$ a primitive $n$-th root of unity, have been extensively studied since the 1980s and 1990s respectively. The latter are covered in several texts  such as \cite{Lus,Jan}, although precise definitions and the relation to Lusztig's celebrated divided-difference versions of $U_q(\cg)$ are quite subtle and depend on the precise root when $n$ is small. See \cite{Len}  for a recent work. There are also corresponding `coordinate algebra' quantum groups $C_q[G]$ and in principle reduced  finite-dimensional quotients $c_q[G]$ although again best understood for specific cases \cite{CL}. 

In spite of some extensive literature, one problem which we believe to be open till now even for the simplest case of $c_q[SL_2]$ is a description of its  dual basis in terms of the generators and relations. Here $u_q(sl_2)$ has generators $F,K,E$ with the relations of $U_q(sl_2)$ and additionally $E^n=F^n=0, K^n=1$, and PBW basis $\{F^i K^jE^k\}$ for $0\le i,j,k<n$. The dual Hopf algebra $c_q[SL_2]$ is a
quotient of $\C_q[SL_2]$ with its standard matrix entry generators $a,b,c,d$, and the additional relations $a^{n}=d^{n}=1,  b^{n}=c^{n}=0$ to give a Hopf algebra extension
\[ \C[SL_2]\hookrightarrow \C_q[SL_2]\twoheadrightarrow c_q[SL_2].\]
This $c_q[SL_2]$ has an obvious monomial basis $\{b^ia^jc^k\}$ but its Hopf algebra pairing with the PBW basis of $u_q(sl_2)$ is rather complicated (it can be related to the representation theory of the quantum group) and this does not constitute a dual basis even up to normalisation. Knowing a basis and dual basis is equivalent to knowing the canonical coevaluation element, which has many applications including Hopf algebra Fourier transform. Here we solve the dual basis problem for $c_q[SL_2]$  at $q$ a primitive odd root of unity in  Corollary~\ref{dual basis},  finding new generators $X,t,Y$ of $c_q[SL_2]$ such that normalised monomials $\{X^it^jY^k\}$ {\em are} essentially a dual basis  in the sense  of being dually paired by 
\[ \<X^it^jY^k,F^{i'}K^{j'}E^{k'}\>=\delta_{ii'}\delta_{kk'}q^{jj'}[i]_{q^{-1}}![k]_q!,\]
where $[i]_q$ etc. are $q$-integers. An actual dual basis immediately follows. Section~6 similarly computes the dual basis for $c_q[SL_3]$ at certain roots of unity including all $n$ that are prime and congruent to $\pm 1$ mod 12. 

This  result depends on a general `braided quantum double' or  {\em double-bosonisation}  construction \cite{db,db2,Primer} which associates to each finite-dimensional braided-Hopf algebra (`braided group') $B$ living in the category of modules over a quasitriangular Hopf algebra $H$,  a new quasitriangular Hopf algebra 
\[ B^{*\mathrm{\underline{cop}}} \lbiprod H \rbiprod B=:D_H(B),\]
where the second notation has also been used in the literature in line with the view of this in \cite{db2} as the closest one can come to the bosonisation of a  `braided double' of $B$ (the latter does not itself exist  in the strictly braided case).  Thus, for any $n$ we take $H=\C[K]/(K^n-1)$ with its natural quasitriangular structure $\CR_K$ so that its modules are the category of $\Z$-graded spaces with braiding given by a power of $q$ according to the degrees. We take $B=\C[E]/(E^n)$ and $B^{*\mathrm{\underline{cop}}}=\C[F]/(F^n)$ and obtain a version 
\[ B^{*\mathrm{\underline{cop}}} \lbiprod H \rbiprod B =\cu_q(sl_2)\cong\begin{cases} u_p(sl_2) & n=2m+1,\ p=q^{-m};\ p^2=q,\\ {\rm something\ else} & n\ {\rm even.}\end{cases}\]
To illustrate the even case, $\cu_{-1}(sl_2)$ in Example~\ref{exq2} is an interesting  8-dimensional strictly quasitriangular and self-dual Hopf algebra presumably known elsewhere. Our approach to the dual basis problem is to work out the dual version or {\em co-double bosonisation} and use this to construct the dual of $\cu_q(sl_2)$  in the dual form 
\[ B^{\mathrm{\underline{op}}}\lbiprod A \rbiprod B^{*}=coD_A(B),\]
where each tensor factor pairs with the corresponding factor on the $\cu_q(sl_2)$ side. This dual version of double bosonisation is in Section~3 and is conceptually given by reversing arrows in the original construction, but in practice takes a great deal of care to trace through all the layers of the construction. Moreover, we do not want to be limited to finite-dimensional $A$ and give a self-contained algebraic proof for $A$ any coquasitriangular Hopf algebra and $B$ a finite-dimensional braided group living in the category of its comodules. When $A$ is also finite dimensional, the resulting object has a Hopf algebra duality pairing  with the double bosonisation by, schematically,  $\langle B^{\mathrm{\underline{op}}}\lbiprod A \rbiprod B^{*}, B^{*\mathrm{\underline{cop}}}\lbiprod H \rbiprod B \rangle = \langle B^{\mathrm{\underline{op}}}, B^{*\mathrm{\underline{cop}}} \rangle \langle A, H \rangle \langle B^{*},B \rangle$. Unlike double bosonisation, the co-double bosonisation has coalgebra surjections to the constituents $B^{\underline{\rm op}},A,B^*$ making calculations for it  harder than the original version. 

In Section~\ref{Sec4} we take $B=\mathbb{C}[X]/(X^n)$ and its dual $B^{*}=\mathbb{C}[Y]/(Y^n)$, again braided-lines but this time viewed in the category of $A=\mathbb{C}[t]/(t^n-1)$-comodules with its standard coquasitriangular structure $\CR(t,t)=q$ (so that its comodules form the same braided category of $\Z$-graded vector spaces as before). Then the co-double bosonisation gives a coquasitriangular Hopf algebra  $\cc_q[SL_2]=c_p[SL_2]$ when $n$ is odd and some other  coquasitriangular Hopf algebra when $n$ is even. This is Theorem~\ref{cdbthm} with the dual basis result a corollary of the triangular decomposition.  As an application, Hopf algebra Fourier transform $\CF:\cc_q[SL_2]\to \cu_q[sl_2]$ is worked out in Section~\ref{Sec5} and shown to behave well with respect to the 3D-calculus of $c_p[SL_2]$. Another application of the canonical element for the pairing, which we do not discuss, is to provide the quasitriangular structure of the Drinfeld double $D(u_p(sl_2))$ of interest in 3D quantum gravity. 

Although the details rapidly become complicated, Sections~~\ref{upsl3} and~\ref{cpsl3} similarly study the next iteration, $\cc_q[SL_3]$ as dually paired to $\cu_q(sl_3)$, where $H=\widetilde{\cu_q(sl_2)}$ is a certain central extension and $B=\cc_q^2$ denotes the usual quantum-braided plane but reduced at the root of unity by making the generators nilpotent of order $n$. The central extension requires an integer $\beta$ such that $\beta^2=3$ mod $n$, where we assume that $n=2m+1$ is odd and set $p=q^{-m}$. The dual of $B$ has a similar form and double-bosonisation gives us 
\[ \cu_q(sl_3)\cong\begin{cases}u_p(sl_3) & m>1,\\  (u_p(sl_3)/\<K_1-K_2\>)\tens(\C[g]/(g^n-1)) & m=1,\end{cases}\]
where the $m=1$ case equates the two Cartan generators of the usual quantum group. This quotient is necessarily quasitriangular by our construction whereas we are not clear if this is the case for $u_p(sl_3)$ itself when $q^3=1$. We then construct the dual by $A=\widetilde{\cc_q[SL_2]}$ and similar quantum-braided planes now as Hopf algebras in its category comodules lead to a dual coquasitriangular Hopf  algebra $\cc_q[SL_3]$. For $m>1$ we show that this is isomorphic to  the usual $c_p[SL_3]$ while for $m=1$ we obtain  a central extension of a sub-Hopf algebra of $c_p[SL_3]$.  Clearly, one could go on to analyse other choices of $n$, as well as to look similarly at the next iteration for  $\cu_q(sl_4)$ and its duality with $\cc_q[SL_4]$, etc. Even at the second stage of $H=\widetilde{u_q(sl_2)}$, there are other potential choices for braided planes including some that give versions of $u_q(g_2)$ and $u_q(sp_2)$ (details will be given elsewhere) and others that have no classical picture at all. Existence was covered at the semiclassical or Lie bialgebra level in \cite{Ma:blie} as an inductive process that adds one to the rank of the Lie algebra at each iteration, and is also clear for generic $U_q(\cg)$ in a suitable setting \cite{db}. Section~\ref{fermionic} illustrates a non-classical choice where $A=\C_q[GL_2]$ is not finite dimensional, $q$ is generic and $B=\C_q^{0|2}$ is the `fermionic quantum-braided plane'  in the category of $A$-comodules. This leads to an exotic but still coquasitriangular version of  $\C_q[SL_3]$ with some matrix entries `fermionic'. We also note that the inductive approach, even after multiple iterations,  preserves a triangular decomposition in which the accumulated central generators form the `Cartan' factor, the accumulated braided groups $B$ form a `positive' braided group on one side and their duals form a `negative' braided group on the other side. For the classical families, this recovers versions of $u_q(n_\pm)$ but now as braided-Hopf algebras with dual bases.

\section{Preliminaries}\label{Sec2}

We recall  the notations and facts about Drinfeld's (co)quasitriangular Hopf algebras as can be found in several texts, for example  \cite{Foundation,Primer}, braided groups and bosonisation as introduced in \cite{AlgBr, Braid, Bos, Foundation} and double bosonisation \cite{db, db2, Ma:Rem, Primer}. We also establish lemmas needed for a clean presentation of the latter and its dualisation.  
 
\subsection{Quasitriangular Hopf algebras} Recall that a Hopf algebra is $(H, \Delta, \epsilon, S)$ where $H$ is a unital algebra,  $\Delta : H \to H \otimes H$ and $\eps:H\to \kk$ form a coalgebra and are algebra maps, and there is an antipode $S : H \to H$ obeying $(Sh\o)h\t=\eps(h)=h\o(Sh\t)$ for all $h\in H$. We use Sweedler's notation $\Delta h = h{\o}\otimes h{\t}$ (summation understood) and $\kk$ is the ground field. Modules/comodules of $H$ have a tensor product defined respectively by pull back/push out along the co/product.  Another Hopf algebra $A$ is `dually paired' if  $\langle\ ,\ \rangle : A\otimes H \to \kk$ makes the coalgebra and antipode on one side  adjoint to the algebra and antipode on the other (e.g., $\langle ab, h \rangle= \langle a,h{\o} \rangle \langle b,h{\t} \rangle$ for all $a,b\in A$ and $h\in H$). If $H$ is finite dimensional then we can take $A=H^{*}$.  

A Hopf algebra is {\em  quasitriangular} \cite{Dri} if equipped with invertible $\mathcal{R}=\mathcal{R}^{\o}\otimes \mathcal{R}^{\t}\in H\otimes H$ (summation understood) such that $(\Delta \otimes \mathrm{id})\mathcal{R}=\mathcal{R}_{13}\mathcal{R}_{23}$,  $(\mathrm{id}\otimes \Delta)\mathcal{R}=\mathcal{R}_{13}\mathcal{R}_{12}$ and $\mathrm{flip}\circ \Delta h =\mathcal{R}(\Delta h)\mathcal{R}^{-1}$. Here 
 $\mathcal{R}_{12}=\mathcal{R}\otimes 1$ and so forth. We denote by $\bar{H}$ the quasitriangular Hopf algebra which is  the same Hopf algebra as $H$  but with quasitriangular structure $\bar{\mathcal{R}}=\mathcal{R}^{-1}_{21}$. The dual notion, e.g. in  \cite{Braid}, of a {\em coquasitriangular} Hopf algebra $A$ is a Hopf algebra with  a convolution-invertible map $\mathcal{R}: A \otimes A \to k$ satisfying 
\begin{equation}\label{quabic} \mathcal{R}(ab,c)=\mathcal{R}(a,c{\o})\mathcal{R}(b,c{\t}),\quad \mathcal{R}(a,bc)=\mathcal{R}(a{\o},c)\mathcal{R}(a{\t},b),\end{equation}
\begin{equation}\label{quacom} a{\o}b{\o}\mathcal{R}(b{\t},a{\t})=\mathcal{R}(b{\o},a{\o})b{\t}a{\t}\end{equation}
 for all $a,b,c\in A$. We define $\bar A$ to  be the same Hopf algebra as $A$ but with coquasitriangular structure $\bar{\mathcal{R}}=\mathcal{R}^{-1}_{21}$. Equivalently, $\bar{\mathcal{R}}(a,b)=\mathcal{R}(Sb,a)$ for all $a,b \in \bar{A}$. It is shown in \cite{Braid,Foundation} that the antipode in this context is invertible.  
 
Let $H$ (resp. $A$)  be (co)quasitriangular. The monoidal categories of left or right (co)modules are braided in the sense of an isomorphism $\Psi_{V,W}:V\tens W\to W\tens V$  obeying axioms similar to the transposition map, but {\em not} $\Psi_{W,V}\Psi_{V,W}=\id$,  given by
\[\begin{aligned}
\Psi_{L}(v\otimes w)=&\mathcal{R}^{\t}\rhd w \otimes \mathcal{R}^{\o}\rhd v, & \Psi_{R}(v\otimes w)=& w\lhd \mathcal{R}^{\o}\otimes v \lhd \mathcal{R}^{\t},\\
\Psi^{L}(v\otimes w)=&\mathcal{R}(w^{\bar{\o}},v^{\bar{\o}})w^{\bar{\infi}}\otimes v^{\bar{\infi}}, & \Psi^{R}(v\otimes w)=& w^{\bar{\z}}\otimes v^{\bar{\z}}\mathcal{R}(v^{\bar{\o}},w^{\bar{\o}})
\end{aligned}\]
where $\Psi_{L}$ is the braiding for the left-modules category ${}_{H}\CM$ with action $\la$, $\Psi_R$ the same for right modules $\CM_H$ and action $\ra$,   $\Psi^{L}$  for the left-comodule category ${}^{A}\mathcal{M}$ with coaction $\Delta_Lv=v^{\bar{\o}}\tens v^{\bar{\infi}}$ and $\Psi^R$ for right-comodules $\CM^A$ with coaction denoted $\Delta_Rv=v^{\bar{\z}}\tens v^{\bar{\o}}$ (summations understood). 

\subsection{Bosonisation and cobosonisation} 

A left $H$-module algebra $B$ means a Hopf algebra $H$ acting from the left on an algebra $B$ such that  $h\la(bc)=(h\o\la b)(h\t \la c)$ and $h\la 1=\eps(h)$ for all $b,c\in B$ and $h\in H$. Equivalently, $B\in {}_H\CM$ as an algebra, i.e., an object
and the product and unit maps are morphisms. One has the familiar smash or cross product algebra which we denote  $B \lcross H$ built on $B\otimes H$ with $(b\otimes h)(c \otimes g)= b(h{\o}\rhd c)\otimes h{\t}g$ for all $b,c\in B$ and $h,g\in H$.  Similarly if $B\in \CM_H$ as an algebra there is a right cross product algebra $H \rcross B$ built on $H \otimes B$ with $(h\otimes b)(g \otimes c)= hg{\o}\otimes (b\lhd g{\t})c$. Similarly, given a coalgebra $B\in {}^H\CM$ (a left-$H$-comodule coalgebra or explicitly
\begin{align}
(\mathrm{id}\otimes \Delta)\Delta_{L}(b)=b{\o}^{\bar{\o}}b{\t}^{\bar{\o}}\otimes b{\o}^{\bar{\infi}}\otimes b{\t}^{\bar{\infi}},\quad (\id\tens\eps)\Delta_L=1\tens\eps \label{left-comod-coalg}
\end{align}
for all $b\in B$) one has a left cross coproduct coalgebra $B\lcocross H$ built on $B\otimes H$ with \begin{align*}
\Delta(b\otimes h)=& b{\o}\otimes b{\t}^{\bar{\o}}h{\o}\otimes b{\t}^{\bar{\infi}}\otimes h_{\t}. 
\end{align*}
Given a coalgebra $B\in \CM^H$ (so a right $H$-comodule coalgebra or 
\begin{align}
(\Delta \otimes \mathrm{id})\Delta_{R}(b)=b{\o}^{\bar{\z}}\otimes b{\t}^{\bar{\z}}\otimes b{\o}^{\bar{\o}}b{\t}^{\bar{\o}},\quad (\eps\tens\id)\Delta_R=\eps\tens 1 \label{right-comod-coalg}
\end{align}
for all $b\in B$)  there is a right cross coproduct coalgebra $H\rcocross B$ built on $H\otimes B$ with \begin{align*}
\Delta (h \otimes b)=& h{\o}\otimes b{\o}^{\bar{\z}}\otimes h{\t}b{\o}^{\bar{\o}}\otimes b{\t}. 
\end{align*}
We refer to \cite{Foundation} for details. When $H$ is quasitriangular there is a braided monoidal functor ${}_H\CM\hookrightarrow {}_H^H\CM$ in \cite{Ma:dou,Foundation} with a coaction induced  by the quasitriangular structure of $H$ so as to form a `crossed' or Yetter-Drinfeld module.  Similarly from the right and dually for  $A$ coquasitriangular via functors ${}^A\CM\hookrightarrow{}_A^A\CM$ and ${}^A\CM\hookrightarrow\CM{}_A^A$. The latter involve an action induced by the given coaction. 

We also need the notion of a `braided group' or Hopf algebra $B$ in a braided category $\CC$, the basic theory of which was worked out in \cite{Braid,Bos,AlgBr}. The unit element is viewed as a morphism $\eta:\underline 1\to B$ from the unit object which in our case will just be $k$. The product, counit, coproduct and antipode are morphisms and we underline the latter two for clarity. In our concrete setting we write $\underline\Delta b=b\underline{\o}\otimes b_{\underline{\t}}$ (summation understood) and recall that   $\underline\Delta$ is an algebra hom to the braided tensor product algebra, so that $\underline{\Delta} \cdot=(\cdot \otimes \cdot)(\mathrm{id}\otimes \Psi \otimes \mathrm{id})(\underline{\Delta}\otimes \underline{\Delta})$ with $\Psi$ the braiding on $B\otimes B$. We have \cite{Braid,AlgBr},
\begin{equation}\label{braidedS}
\underline{S} \circ \cdot = \cdot \circ \Psi \circ (\underline{S}\otimes \underline{S}),\quad 
\underline{\Delta}\circ \underline{S}=(\underline{S}\otimes \underline{S}) \circ \Psi \circ \underline{\Delta}. \end{equation}

\begin{lemma}\cite{Foundation,Primer,Bos} \label{boson}
Let $H$ be quasitriangular and $B\in\CM_H$ a braided group. Then $H\rcross B$ by the given action and $H\rcocross B$ by the induced coaction form a Hopf algebra $H \rbiprod B$ (the bosonisation of $B$). 
\end{lemma} 
The coproduct here is $\Delta(h\otimes b)=h{\o}\otimes b{\underline{\o}}\lhd \mathcal{R}^{\o}\otimes h{\t}\mathcal{R}^{\t}\otimes b{\underline{\t}}$. 
Similarly for $B\in {}_H\CM$ to give $B\lbiprod H$. If $A$ is coquasitriangular and  $B\in\CM^A$ then the {\em cobosonisation} is the ordinary Hopf algebra $A\rbiprod B$ with $A\rcocross B$ by the given coaction and $A\rcross B$ by the induced action from the above functor. Explicitly, the cross product is $(a \otimes b)(d \otimes c)=ad{\underline{\o}}\otimes b^{\bar{\z}}c\,  \mathcal{R}(b^{\bar{\o}},d{\underline{\t}})$. Both constructions are examples of a more general Radford-Majid biproduct theorem \cite{Rad:str, Ma:skl} (the latter gave the categorical picture) whereby for any Hopf algebra $H$ with invertible antipode, Hopf algebras with split projections to $H$ are of the form $B\lbiprod H$ for some $B\in {}_H^H\CM$.  

Finally, the  notion of a dually paired or categorical dual braided group $B^\star$ (when $B$ is a rigid object, e.g. finite-dimensional in our applications) in \cite{Braid,AlgBr} needs a little care to define the pairing $B^\star\tens B^\star\tens B\tens B$ by pairing $B^\star\tens B$ in the middle first. Pairing maps go to the trivial object. In our context, where objects are built on vector spaces, it is useful to match ordinary Hopf algebra conventions by defining $B^*$ with the adjoint algebra and coalgebra structures in the usual way rather than the above categorical way, which, however, canonically lands $B^*$ in a different category from $B$.  

\begin{lemma}\label{lemma dual boson}
Let $H$ be finite dimensional and quasitriangular with dual $A$, and $B$ be a finite-dimensional braided group in $\mathcal{M}_{H}$. Then $B^*\in \CM^A$ and $(H\rbiprod B)^{*}=A\rbiprod B^{*}$. Similarly, if $B\in {}_{H}\mathcal{M}$ then $B^*\in {}^A\CM$ and $(B\lbiprod H)^{*}=B^{*}\lbiprod A$.
\end{lemma}

 \subsection{Double bosonisation}
 
 Another basic fact about braided groups is that if $B\in \CC$ with invertible antipode then $B^{\mathrm{\underline{op}}}$ with the same coalgebra structure as $B$ but with braided-opposite product and antipode given by
 \begin{align}
 \cdot_{\mathrm{\underline{op}}} = \cdot \circ \Psi^{-1}_{B,B},\quad \bar{S}=\underline{S}^{-1}
 \end{align}
 is a braided group in $\bar \CC$, by which we mean $\CC$ with the reversed (inverse) braid crossing \cite{Braid}. The same remarks apply for $B^{\mathrm{\underline{cop}}}\in\bar \CC$ with $\delcop = \Psi^{-1}_{B,B} \circ \underline{\Delta}$ and inverted $\underline S$. 
 
 If $H$ is quasitriangular and $\CC=\CM_H$ then $\bar\CC=\CM_{\bar H}$. Let $B$ be a braided group in $\mathcal{M}_{H}$. By the theory of bosonisation, we have two Hopf algebras $H \rbiprod B$ and $B^{*\mathrm{\underline{cop}}}\lbiprod \bar{H}$. We can glue them together to get the following theorem.
 \begin{theorem}c.f. \cite[Theorem 3.2]{db}\label{Double Bosonisation}
 	Let $H$ be a quasitriangular Hopf algebra. Let $B$ be a  finite-dimensional braided group in $\mathcal{M}_{H}$. There is an ordinary Hopf algebra $B^{*\mathrm{\underline{cop}}}\lbiprod H \rbiprod B$, the double bosonisation, built on $B^{*\mathrm{\underline{cop}}}\otimes H \otimes B$ and containing $B^{*\mathrm{\underline{cop}}}\lbiprod \bar{H}$ and $H \rbiprod B$ as sub-Hopf algebras with cross relation
 	\begin{equation*}
 	\begin{aligned}
 	bc=&(\mathcal{R}_{1}^{\t}\rhd c{\overline{\t}})\mathcal{R}_{2}^{\t}\mathcal{R}_{1}^{\mo}(b{\underline{\t}}\lhd \mathcal{R}_{2}^{\mo})\langle \mathcal{R}_{1}^{\o} \rhd c{\overline{\o}}, b{\underline{\o}} \lhd \mathcal{R}_{2}^{\o} \rangle \langle \mathcal{R}_{1}^{\mt} \rhd \bar{S}c{\overline{\th}}, b{\underline{\th}} \lhd \mathcal{R}_{2}^{\mt} \rangle. 
 	\end{aligned}
 	\end{equation*}
 	Furthermore, $B^{*\mathrm{\underline{cop}}}\lbiprod H \rbiprod B$ has a quasitriangular structure  $\CR^{\rm new}=\overline{\mathrm{exp}}\cdot\mathcal{R}$, 
 	where $\overline{\mathrm{exp}}=\sum f^{a}\otimes \underline{S}e_{a}$, $\{e_{a}\}$ is a basis of $B$ and $\{f^{a}\}$ is a dual basis of $B^{*}$.
 \end{theorem}
 
 In fact, $B$ in \cite{db} is not required to be  finite dimensional but we have restricted to the finite-dimensional case for simplicity. Our goal is a dual version of this theorem with $A$ coquasitriangular and $B\in {}^A\CM$, in which case the category with reversed braiding is ${}^{\bar A}\CM$ and
 \begin{align}
 a \dotop b =& \mathcal{R}(Sa^{\bar{\o}},b^{\bar{\o}}) b^{\bar{\infi}}a^{\bar{\infi}}
 \end{align}
 for all $a,b \in B^{\mathrm{\underline{op}}}$. As in Lemma~\ref{lemma dual boson}, we think of ${}^A\CM$ as $\CM_H$ in the finite-dimensional Hopf algebra case by evaluating against a coaction of $A$ to get an action of $H$.
 
 \begin{lemma} \label{lemma about B op braid} If $H$ is finite dimensional and quasitriangular with dual $A$ and $B\in {}^A\CM$ is finite dimensional then 
 	$(B^{\underline{\mathrm{op}}})^{*}=B^{*\mathrm{\underline{cop}}}\in {}_{\bar{H}}\mathcal{M}$. 
 	\begin{proof}  Here $B^{\underline{\mathrm{op}}}\in {}^{\bar A}\CM$ or $\CM_{\bar H}$ and $(B^{\underline{\mathrm{op}}})^*\in {}_{\bar H}\CM$ where $B^{*\underline{\mathrm{cop}}}$ lives. It is clear that the coproduct of  $B^{\underline{\mathrm{op}}}$ corresponds to the product of $B^{*\underline{\mathrm{cop}}}$. For the other half, 
 		\begin{align*}
 		\langle x, b\dotop c \rangle&=\langle x, c^{\bar{\infi}}b^{\bar{\infi}} \rangle \mathcal{R}(Sb^{\bar{\o}},c^{\bar{\o}})=\langle x{\underline{\o}}, c^{\bar{\infi}} \rangle \langle x{\underline{\t}}, b^{\bar{\infi}} \rangle \langle b^{\bar{\o}}, \mathcal{R}^{\mt} \rangle \langle c^{\bar{\o}}, \mathcal{R}^{\mo} \rangle\\
 		&=\langle x{\underline{\o}}, c \lhd \mathcal{R}^{\mo} \rangle \langle x_{\underline{\t}}, b \lhd \mathcal{R}^{\mt} \rangle=\langle \mathcal{R}^{\mt} \rhd x_{\underline{\t}} \otimes \mathcal{R}^{\mo} \rhd x_{\underline{\o}},b\otimes c \rangle
 		\end{align*}
 which is $\langle \delcop x, b\otimes c \rangle$ as required. 
 	\end{proof}
 \end{lemma}
\section{Co-double Bosonisation} \label{Sec3}

The dual version of Theorem~\ref{Double Bosonisation} can in principle now be deduced using the lemmas in the preceding section, at least when $A$ is finite dimensional. However, we do not want to be limited to this case and give  a direct proof of the resulting formulae. 
 
 \begin{theorem}[Co-double bosonisation]\label{codbos}
Let $B$ be a finite-dimensional braided group in ${}^{A}\mathcal{M}$ with basis $\{e_{a}\}$. Denote its dual by $B^{*} \in \mathcal{M}^{A}$ with dual basis $\{f^{a}\}$. Then there is an ordinary Hopf algebra $B^{\mathrm{\underline{op}}}\lbiprod A \rbiprod B^{*}$, the {\em co-double bosonisation}, built on the vector space $B^{\mathrm{\underline{op}}} \otimes A \otimes B^{*}$  with \begin{align*}
(x\otimes & k \otimes y)(w \otimes \ell \otimes z)= x\cdot_{\mathrm{\underline{op}}}w^{\bar{\infi}} \otimes k\t \ell\o \otimes y^{\bar{\z}}z \ \mathcal{R}(y^{\bar{\o}}, \ell\t)\mathcal{R}(Sk\o, w^{\bar{\o}}),\\
\Delta&(x\otimes k \otimes y)\\
=& \sum\limits_{a}  x\underline{\o}\otimes x\underline{\t}^{\bar{\o}}\o k\o \otimes f^{a}\otimes e_{a}\underline{\o}^{\bar{\infi}} \dotop x{\underline{\t}}^{\bar{\infi}} \dotop \bar{S}e_{a}\underline{\th}^{\bar{\infi}} \otimes k{\fo}y{\underline{\o}}^{\bar{\o}}{\t}\otimes y{\underline{\t}}\\
&\mathcal{R}(e_{a}\underline{\o}^{\bar{\o}}, x{\underline{\t}}^{\bar{\o}}{\t}k{\t})\mathcal{R}(S(k{\th}y{\underline{\o}}^{\bar{\o}}{\o}),e_{a}\underline{\th}^{\bar{\o}}) \ \langle y{\underline{\o}}^{\bar{\z}},e_{a}{\underline{\t}} \rangle
\end{align*}
for all $x,w \in B^{\mathrm{\underline{op}}}$, $k,\ell \in A$, and $y,z \in B^{*}$.
\end{theorem}

Here $B^{\mathrm{\underline{op}}}$, $A$ and $B^{*}$ are subalgebras of $B^{\mathrm{\underline{op}}}\lbiprod A \rbiprod B^{*}$ and identifying $x=x\tens 1\tens 1$, $k=1\tens k\tens 1$ and $y=1\tens 1\tens y$ we have $xky \equiv x\otimes k \otimes y$. We also have algebra maps 
\[ B^{*} \hookrightarrow B^{\mathrm{\underline{op}}}\lbiprod A \rbiprod B^{*} \twoheadrightarrow B^{\mathrm{\underline{op}}}\lbiprod A,\quad B^{\mathrm{\underline{op}}} \hookrightarrow B^{\mathrm{\underline{op}}}\lbiprod A \rbiprod B^{*} \twoheadrightarrow A \rbiprod B^{*}\]
where the surjections are $\mathrm{id}\otimes \underline{\eps}$  and $\underline{\epsilon} \otimes \mathrm{id}$ respectively. It remains to prove Theorem~\ref{codbos}. 
\begin{lemma}
The product stated in Theorem \ref{codbos} is associative.
\end{lemma}
\begin{proof} We expand the definition of the product to find
\begin{align*}
\Big((x\otimes &k \otimes y)(w \otimes \ell \otimes z)\Big)(m \otimes j \otimes v)\\
 =& (x\dotop w^{\bar{\infi}} \otimes k{\t}\ell{\o} \otimes y^{\bar{\z}}z)(m \otimes j \otimes v) \ \mathcal{R}(y^{\bar{\o}}, \ell{\t})\mathcal{R}(Sk{\o}, w^{\bar{\o}})\\
=& x\dotop w^{\bar{\infi}}\dotop m^{\bar{\infi}}\otimes k{\th}\ell{\t}j{\o} \otimes y^{\bar{\z}\bar{\z}}z^{\bar{\z}}v \ \mathcal{R}(y^{\bar{\o}}, \ell{\th})\mathcal{R}(Sk{\o}, w^{\bar{\o}})\\
&\quad\mathcal{R}(y^{\bar{\z}\bar{\o}}z^{\bar{\o}},j{\t})\mathcal{R}(S(k{\t}\ell{\o}), m^{\bar{\o}})\\
=&x\dotop w^{\bar{\infi}}\dotop m^{\bar{\infi}}\otimes k{\th}\ell{\t}j{\o} \otimes y^{\bar{\z}}z^{\bar{\z}}v \ \mathcal{R}(y^{\bar{\o}}{\t}, \ell{\th})\mathcal{R}(Sk{\o}, w^{\bar{\o}})\\
&\quad\mathcal{R}(y^{\bar{\o}}{\o}, j{\t})\mathcal{R}(z^{\bar{\o}}, j{\th})\mathcal{R}(S\ell{\o}, m^{\bar{\o}}{\o})\mathcal{R}(Sk{\t}, m^{\bar{\o}}{\t}),
\end{align*}
where the last equality uses the right-coaction property on $y$. Similarly,
\begin{align*}
(x\otimes & k \otimes y)\Big((w \otimes \ell \otimes z)(m \otimes j \otimes v)\Big)\\
=&(x \otimes k \otimes y)(w\dotop m^{\bar{\infi}}\otimes \ell{\t}j{\o} \otimes z^{\bar{\z}}v) \ \mathcal{R}(z^{\bar{\o}}, j{\t})\mathcal{R}(S\ell{\o}, m^{\bar{\o}})\\
=&x\dotop w^{\bar{\infi}}\dotop m^{\bar{\infi}\bar{\infi}}\otimes k{\t}\ell{\t}j{\o} \otimes y^{\bar{\z}}z^{\bar{\z}}v \ \mathcal{R}(z^{\bar{\o}}, j{\th})\mathcal{R}(S\ell{\o}, m^{\bar{\o}})\\
&\quad\mathcal{R}(y^{\bar{\o}}, \ell{\th}j{\t})\mathcal{R}(Sk{\o}, w^{\o}m^{\bar{\infi}\bar{\o}}),
\end{align*}
which by the left-coaction property on $m$ agrees with our first calculation.  
\end{proof}
\begin{lemma}
The coproduct $\Delta$ stated in Theorem \ref{codbos} is an algebra map.
\end{lemma}
\begin{proof}
 Expanding the product and then the coproduct, we have 
 \begin{align*}
\Delta&\Big((x\otimes k \otimes y)(w \otimes \ell \otimes z)\Big)\\
=&x{\underline{\o}}\dotop w^{\bar{\infi}}{\underline{\o}}^{\bar{\infi}}\otimes x{\underline{\t}}^{\bar{\infi}\bar{\o}}{\o}w^{\bar{\infi}}{\underline{\t}}^{\bar{\o}}{{\o}}k{\t}\ell{\o}\otimes f^{a}\\
&\otimes e_{a}\underline{\o}^{\bar{\infi}}\dotop x{\underline{\t}}^{\bar{\infi}\bar{\infi}}\dotop w^{\bar{\infi}}{\underline{\t}}^{\bar{\infi}}\dotop \bar{S}e_{a}\underline{\th}^{\bar{\infi}}\otimes k{\fiv}\ell{\fo}y^{\bar{\z}}{\underline{\o}}^{\bar{\o}}{\t}z{\underline{\o}}^{\bar{\z}\bar{\o}}{\t}\\
&\otimes y^{\bar{\z}}{\underline{\t}}^{\bar{\z}}z{\underline{\t}} \mathcal{R}(e_{a}\underline{\o}^{\bar{\o}},x{\underline{\t}}^{\bar{\infi}\bar{\o}}{\t}w^{\bar{\infi}}{\underline{\t}}^{\bar{\infi}\bar\o}{}\t     k{\th}\ell{\t})\mathcal{R}(S(k{\fo}\ell{\th}y^{\bar{\z}}{\underline{\o}}^{\bar{\o}}{\o}z{\underline{\o}}^{\bar{\z}\bar{\o}}{\o}),e_{a}\underline{\th}^{\bar{\o}}) \\
&\mathcal{R}(Sk{\o},w^{\bar{\o}})\mathcal{R}(y^{\bar{\o}},\ell{\fiv})\mathcal{R}(Sx{\underline{\t}}^{\bar{\o}},w^{\bar{\infi}}{\underline{\o}}^{\bar{\o}})\mathcal{R}(y^{\bar{\z}}{\underline{\t}}^{\bar{\o}},z{\underline{\o}}^{\bar{\o}}) \ \langle y^{\bar{\z}}{\underline{\o}}^{\bar{\z}}z{\underline{\o}}^{\bar{\z}\bar{\z}},e_{a}\underline{\t} \rangle\\
=&x{\underline{\o}}\dotop w{\underline{\o}}^{\bar{\infi}}\otimes x{\underline{\t}}^{\bar{\o}}{\t}w{\underline{\t}}^{\bar{\o}}{\t}k{\t}\ell{\o}\otimes f^{a} \otimes e_{a}\underline{\o}^{\bar{\infi}}\dotop x{\underline{\t}}^{\bar{\infi}}\dotop w{\underline{\t}}^{\bar{\infi}}\dotop \bar{S}e_{a}\underline{\fo}^{\bar{\infi}}\\
&\otimes k{\fiv}\ell{\fo}y{\underline{\o}}^{\bar{\o}}{\t}z{\underline{\o}}^{\bar{\o}}{\t}\otimes y{\underline{\t}}^{\bar{\z}}z{\underline{\t}} \ \mathcal{R}(e_{a}\underline{\o}^{\bar{\o}}, x{\underline{\t}}^{\bar{\o}}{\th}w{\underline{\t}}^{\bar{\o}}{\th}k{\th}\ell{\t})\\
&\mathcal{R}(S(k{\fo}\ell{\th}y{\underline{\o}}^{\bar{\o}}{\o}z{\underline{\o}}^{\bar{\o}}{\o}),e_{a}\underline{\fo}^{\bar{\o}}) \mathcal{R}(Sk{\o}, w{\underline{\o}}^{\bar{\o}}{\o}w{\underline{\t}}^{\bar{\o}}{\o})\mathcal{R}(y{\underline{\o}}^{\bar{\o}}{\th}y{\underline{\t}}^{\bar{\o}}{\t},\ell{\fiv})\\
&~~~~\mathcal{R}(Sx{\underline{\t}}^{\bar{\o}}{\o},w{\underline{\o}}^{\bar{\o}}{\t})\mathcal{R}(y{\underline{\t}}^{\bar{\o}}{\o},z{\underline{\o}}^{\bar{\o}}) \ \langle y{\underline{\o}}^{\bar{\z}},e_{a}\underline{\t} \rangle \langle z{\underline{\o}}^{\bar{\z}},e_{a}\underline{\th} \rangle\\
=&x{\underline{\o}}\dotop w{\underline{\o}}^{\bar{\infi}}\otimes x{\underline{\t}}^{\bar{\o}}{\t}w{\underline{\t}}^{\bar{\o}}{\t}k{\th}\ell{\t}\tens f^a \otimes e_{a}\underline{\o}^{\bar{\infi}}\dotop x{\underline{\t}}^{\bar{\infi}}\dotop w{\underline{\t}}^{\bar{\infi}}\dotop \bar{S}e_{a}\underline{\fo}^{\bar{\infi}}\\
&\otimes k{\six}\ell{\fo}y{\underline{\o}}^{\bar{\o}}{\t}z{\underline{\o}}^{\bar{\o}}{\t}\otimes y{\underline{\t}}^{\bar{\z}}z{\underline{\t}} \ \mathcal{R}(e_{a}\underline{\o}^{\bar{\o}}, x{\underline{\t}}^{\bar{\o}}{\th}w{\underline{\t}}^{\bar{\o}}{\th}k{\fo}\ell{\t})\\
&\mathcal{R}(S(k{\fiv}\ell{\th}y{\underline{\o}}^{\bar{\o}}{\o}z{\underline{\o}}^{\bar{\o}}{\o}),e_{a}\underline{\fo}^{\bar{\o}}) \mathcal{R}(S(x{\underline{\t}}^{\bar{\o}}{\o}k{\o}),w{\underline{\o}}^{\bar{\o}}) \mathcal{R}(Sk{\t}, w{\underline{\t}}^{\bar{\o}}{\o})\\
&\mathcal{R}(y{\underline{\o}}^{\bar{\o}}{\th},\ell{\fiv}) \mathcal{R}(y{\underline{\t}}^{\bar{\o}},\ell{\six}z{\underline{\o}}^{\bar{\o}})\ \langle y{\underline{\o}}^{\bar{\z}}, e_{a}\underline{\t} \rangle \langle z{\underline{\o}}^{\bar{\z}},e_{a}\underline{\th} \rangle
\end{align*}
for all  $x,w \in B^{\mathrm{op}}$, $k,\ell \in A$,  $y,z \in B^{*}$. The second equality uses the comodule coalgebra property (\ref{left-comod-coalg}) on $w$ and coassociativity. The last expression uses coquasitriangularity (\ref{quabic}) to gather the parts of $w{\underline{\o}}^{\bar{\o}}$ and $y{\underline{\t}}^{\bar{\o}}$ inside $\mathcal{R}$. On the other side,
\begin{align*}
\Delta&(x\otimes k \otimes y)\Delta(w \otimes \ell \otimes z)\\
=&x{\underline{\o}}\dotop w{\underline{\o}}^{\bar{\infi}}\otimes x{\underline{\t}}^{\bar{\o}}{\o\t}k{\o\t}w{\underline{\t}}^{\bar{\o}}{\o\o}\ell{\o\o}\otimes f^{a\bar{\z}}f^{b}\\
&\otimes e_{a}\underline{\o}^{\bar{\infi}}\dotop x{\underline{\t}}^{\bar{\infi}}\dotop \bar{S}e_{a}\underline{\th}^{\bar{\infi}}\dotop e_{b}\underline{\o}^{\bar{\infi}\bar{\infi}}\dotop w{\underline{\t}}^{\bar{\infi}\bar{\infi}}\dotop \bar{S}e_{b}\underline{\th}^{\bar{\infi}\bar{\infi}}\\
&\otimes k{\fo\t}y{\underline{\o}}^{\bar{\o}}{\t\t}\ell{\fo\o}z{\underline{\o}}^{\bar{\o}}{\t\o}\otimes y{\underline{\t}}^{\bar{\z}}z{\underline{\t}} \ \mathcal{R}(e_{a}\underline{\o}^{\bar{\o}},x{\underline{\t}}^{\bar{\o}}{\t}k{\t})\\
&\mathcal{R}(S(k{\th}y{\underline{\o}}^{\bar{\o}}{\o}),e_{a}\underline{\th}^{\o})\mathcal{R}(e_{b}\underline{\o}^{\bar{\o}},w{\underline{\t}}^{\bar{\o}}{\t}\ell{\t}) \mathcal{R}(S(\ell{\th}z{\underline{\o}}^{\bar{\o}}{\o}),e_{b}\underline{\th}^{\bar{\o}}) \\
&\mathcal{R}(S(x{\underline{\t}}^{\bar{\o}}{\o\o}k{\o\o}),w{\underline{\o}}^{\bar{\o}})\mathcal{R}(f^{a\bar{\o}},w{\underline{\t}}^{\bar{\o}}{\o\t}\ell{\o\t})\\
&\mathcal{R}(S(k{\fo\o}z{\underline{\o}}^{\bar{\o}}{\t\o}),e_{b}\underline{\o}^{\bar{\infi}\bar{\o}}w{\underline{\t}}^{\bar{\infi}\bar{\o}}e_{b}\underline{\th}^{\bar{\infi}\bar{\o}})\mathcal{R}(y{\underline{\t}}^{\bar{\o}},\ell{\fo\t}z{\underline{\o}}^{\bar{\o}}{\t\t})\\
&\langle y{\underline{\o}}^{\bar{\z}},e_{a}\underline{\t} \rangle \langle z{\underline{\o}}^{\bar{\z}},e_{b}\underline{\t} \rangle\\
=&x{\underline{\o}}\dotop w{\underline{\o}}^{\bar{\infi}} \otimes x{\underline{\t}}^{\bar{\o}}{\t}k{\t}w{\underline{\t}}^{\bar{\o}}{\o}\ell{\o}\otimes f^{a}f^{b}\\
&\otimes e_{a}\underline{\o}^{\bar{\infi}\bar{\infi}}\dotop x{\underline{\t}}^{\bar{\infi}}\dotop \bar{S}e_{a}\underline{\th}^{\bar{\infi}\bar{\infi}}\dotop e_{a}\underline{\fo}^{\bar{\infi}\bar{\infi}}\dotop w{\underline{\t}}^{\bar{\infi}\bar{\infi}}\dotop \bar{S}e_{a}\underline{\six}^{\bar{\infi}\bar{\infi}}\\
&\otimes k{\six}y{\underline{\o}}^{\bar{\o}}{\th}\ell{\fiv}z{\underline{\o}}^{\bar{\o}}{\t}\otimes y{\underline{\t}}^{\bar{\z}}z{\underline{\t}} \ \mathcal{R}(e_{a}\underline{\o}^{\bar{\infi}\bar{\o}}, x{\underline{\t}}^{\bar{\o}}{\th}k{\th})\\
&\mathcal{R}(S(k{\fo}y{\underline{\o}}^{\bar{\o}}{\o}), e_{a}\underline{\th}^{\bar{\infi}\bar{\o}})\mathcal{R}(e_{a}\underline{\fo}^{\bar{\o}},w{\underline{\t}}^{\bar{\o}}{\th}\ell{\th})\mathcal{R}(S(\ell{\fo}z{\underline{\o}}^{\bar{\o}}{\o}),e_{a}\underline{\six}^{\bar{\o}})\\
&\mathcal{R}(S(x{\underline{\t}}^{\bar{\o}}{\o}k{\o}),w{\underline{\o}}^{\bar{\o}})\mathcal{R}(e_{a}\underline{\o}^{\bar{\o}}e_{a}\underline{\t}^{\bar{\o}}e_{a}\underline{\th}^{\bar{\o}},w{\underline{\t}}^{\bar{\o}}{\t}\ell{\t})\\
&\mathcal{R}(S(k{\fiv}y{\underline{\o}}^{\bar{\o}}{\t}),e_{a}\underline{\fo}^{\bar{\infi}\bar{\o}}w{\underline{\t}}^{\bar{\infi}\bar{\o}}e_{a}{\underline{\six}}^{\bar{\infi}\bar{\o}})\mathcal{R}(y{\underline{\t}}^{\bar{\o}}, \ell{\six}z{\underline{\o}}^{\bar{\o}}{\th})\\
&\langle y{\underline{\o}}^{\bar{\z}}, e_{a}\underline{\t}^{\bar{\infi}} \rangle \langle z{\underline{\o}}^{\bar{\z}}, e_{a}\underline{\fiv} \rangle\\
=&x{\underline{\o}}\dotop w{\underline{\o}}^{\bar{\infi}}\otimes x{\underline{\t}}^{\bar{\o}}{\t}k{\t}w{\underline{\t}}^{\bar{\o}}{\o}\ell{\o}\otimes f^{a} \otimes e_{a}\underline{\o}^{\bar{\infi}}\dotop x{\underline{\t}}^{\bar{\infi}}\dotop w{\underline{\t}}^{\bar{\infi}}\dotop \bar{S}e_{a}\underline{\fo}^{\bar{\infi}}\\
&\otimes k{\fiv}y{\underline{\o}}^{\bar{\o}}{\th}\ell{\fo}z{\underline{\o}}^{\bar{\o}}{\t}\otimes y{\underline{\t}}^{\bar{\z}}z{\underline{\t}} \ \mathcal{R}(e_{a}\underline{\o}^{\bar{\o}}{\t}, x{\underline{\t}}^{\bar{\o}}{\th}k{\th})\mathcal{R}(S(\ell{\th}z{\underline{\o}}^{\bar{\o}}{\o}),e_{a}\underline{\fo}^{\bar{\o}}{\o})\\
&\mathcal{R}(S(x{\underline{\t}}^{\bar{\o}}{\o}k{\o}),w{\underline{\o}}^{\bar{\o}}) \mathcal{R}(e_{a}\underline{\o}^{\bar{\o}}{\o}y{\underline{\o}}^{\bar{\o}}{\o}, w{\underline{\t}}^{\bar{\o}}{\t}\ell{\t})\\
&\mathcal{R}(S(k{\fo}y{\underline{\o}}^{\bar{\o}}{\t}), w{\underline{\t}}^{\bar{\o}}{\th}e_{a}\underline{\fo}^{\bar{\o}}{\t})\mathcal{R}(y{\underline{\t}}^{\bar{\o}}, \ell{\fiv}z{\underline{\o}}^{\bar{\o}}{\th}) \langle y{\underline{\o}}^{\bar{\z}}, e_{a}\underline{\t} \rangle \langle z{\underline{\o}}^{\bar{\z}}, e_{a}\underline{\th} \rangle\\
=&x{\underline{\o}}\dotop w{\underline{\o}}^{\bar{\infi}} \otimes x{\underline{\t}}^{\bar{\o}}{\t}k{\t}w{\underline{\t}}^{\bar{\o}}{\o}\ell{\o}\otimes f^{a}\otimes e_{a}\underline{\o}^{\bar{\infi}}\dotop x{\underline{\t}}^{\bar{\infi}}\dotop w{\underline{\t}}^{\bar{\infi}}\dotop \bar{S}e_{a}\underline{\fo}^{\bar{\infi}}\\
&\otimes k{\six}y{\underline{\o}}^{\bar{\o}}{\th}\ell{\fiv}z{\underline{\o}}^{\bar{\o}}{\t} \otimes y{\underline{\t}}^{\bar{\z}}z{\t}\mathcal{R}(e_{a}\underline{\o}^{\bar{\o}}, x{\underline{\t}}^{\bar{\o}}{\th}k{\th}w{\underline{\t}}^{\bar{\o}}{\t}\ell{\t}) \\
&\mathcal{R}(S(k{\fiv}y{\underline{\o}}^{\bar{\o}}{\t}\ell{\fo}z{\underline{\o}}^{\bar{\o}}{\o}), e_{a}\underline{\fo}^{\bar{\o}})\mathcal{R}(S(x{\underline{\t}}^{\bar{\o}}{\o}k{\o}),w{\underline{\o}}^{\bar{\o}})\mathcal{R}(y{\underline{\t}}^{\bar{\o}}, \ell{\six}z{\underline{\o}}^{\bar{\o}}{\th}) \\
&\mathcal{R}(y{\underline{\o}}^{\bar{\o}}{\o}, \ell{\th})\mathcal{R}(Sk{\fo}, w{\underline{\t}}^{\bar{\o}}{\th})\langle y{\underline{\o}}^{\bar{\z}}, e_{a}\underline{\t} \rangle \langle z{\underline{\o}}^{\bar{\z}},e_{a}\underline{\th} \rangle, 
\end{align*}
where the second equality uses duality $\langle f^{a\bar{\z}}, e_{a} \rangle f^{a\bar{\o}}=\langle f^{a}, e_{a}^{\bar{\infi}} \rangle e_{a}^{\bar{\o}}$ followed by the comodule coalgebra property (\ref{left-comod-coalg}) on $e_{a}$. The third equality cancels $(\bar{S}e_{a}\underline{\th}\dotop e_{a}\underline{\fo})^{\bar{\infi}\bar{\infi}}$ making all subsequent  coactions trivial. The fourth equality uses (\ref{quabic}) to gather the parts of $e_{a}\underline{\o}^{\bar{\o}}$ and $e_{a}\underline{\fo}^{\bar{\o}}$ inside $\CR$, and cancels some $\CR$s. In the final expression, one can use quasicommutativity (\ref{quacom}) to reorder the second tensor factor so as to coincide with the result of the first calculation. 
\end{proof}

\begin{lemma}
The coproduct $\Delta$ stated in Theorem \ref{codbos} is coassociative.
\end{lemma}
\begin{proof} We expand the definition of the coproduct to find
\begin{align*}
(\mathrm{id}&\otimes \Delta)\Delta(x\otimes k \otimes y)\\
=&x{\underline{\o}}\otimes x{\underline{\t}}^{\bar{\o}}{\o}x{\underline{\th}}^{\bar{\o}}{\o}k{\o}\otimes f^{a} \otimes e_{a}\underline{\o}^{\bar{\infi}}\dotop x{\underline{\t}}^{\bar{\infi}}\dotop \bar{S}e_{a}\underline{\fiv}^{\bar{\infi}}\\
&\otimes e_{a}\underline{\t}^{\bar{\o}}{\fo}x{\underline{\th}}^{\bar{\o}}{\fo}e_{a}\underline{\fo}^{\bar{\o}}{\th}k{\fo}y{\underline{\o}}^{\bar{\o}}{\t}\otimes f^{b}\\
&\otimes e_{b}\underline{\o}^{\bar{\infi}}\dotop e_{a}\underline{\t}^{\bar{\infi}}\dotop x{\underline{\th}}^{\bar{\infi}}\dotop \bar{S}e_{a}\underline{\fo}^{\bar{\infi}}\dotop \bar{S}e_{b}\underline{\th}^{\bar{\infi}}\otimes k{\sev}y{\underline{\o}}^{\bar{\o}}{\fiv}y{\underline{\t}}^{\bar{\o}}{\t} \otimes y{\underline{\th}}\\
&\langle y{\underline{\o}}^{\bar{\z}}, e_{a}\underline{\th}\rangle \langle y{\underline{\t}}^{\bar{\z}},e_{b}\underline{\t} \rangle \ \mathcal{R}(e_{b}\underline{\o}^{\bar{\o}},e_{a}\underline{\t}^{\bar{\o}}{\fiv}x{\underline{\th}}^{\bar{\o}}{\fiv}e_{a}\underline{\fo}^{\bar{\o}}{\fo}k{\fiv}y{\underline{\o}}^{\bar{\o}}{\th}) \\
&\mathcal{R}(S(k{\six}y{\underline{\o}}^{\bar{\o}}{\fo}y{\underline{\t}}^{\bar{\o}}{\o}),e_{b}\underline{\th}^{\bar{\o}})\mathcal{R}(e_{a}\underline{\o}^{\bar{\o}}e_{a}\underline{\t}^{\bar{\o}}{\o}, x{\underline{\t}}^{\bar{\o}}{\t}x{\underline{\th}}^{\bar{\o}}{\t}k{\t})\\
&\mathcal{R}(S(k{\th}y{\underline{\o}}^{\bar{\o}}{\o}),e_{a}\underline{\fo}^{\bar{\o}}{\o}e_{a}\underline{\fiv}^{\bar{\o}}{\o})\mathcal{R}(S(e_{a}\underline{\t}^{\bar{\o}}{\th}x{\underline{\th}}^{\bar{\o}}{\th}),e_{a}\underline{\fiv}^{\bar{\o}}{\th})\\
&\mathcal{R}(Se_{a}\underline{\fo}^{\bar{\o}}{\t},e_{a}\underline{\fiv}^{\bar{\o}}{\t})\mathcal{R}(Se_{a}\underline{\t}^{\bar{\o}}{\t},x{\underline{\t}}^{\bar{\o}}{\th})\\
=&x{\underline{\o}}\otimes x{\underline{\t}}^{\bar{\o}}{\o}x{\underline{\th}}^{\bar{\o}}{\o}k{\o}\otimes f^{a} \otimes e_{a}\underline{\o}^{\bar{\infi}}\dotop x{\underline{\t}}^{\bar{\infi}}\dotop \bar{S}e_{a}\underline{\fiv}^{\bar{\infi}}\\
&\otimes e_{a}\underline{\t}^{\bar{\o}}{\th}x{\underline{\th}}^{\bar{\o}}{\fiv}e_{a}\underline{\fo}^{\bar{\o}}{\th}k{\six}y{\underline{\o}}^{\bar{\o}}{\th}\otimes f^{b}\\
&\otimes e_{b}\underline{\o}^{\bar{\infi}}\dotop e_{a}\underline{\t}^{\bar{\infi}}\dotop x{\underline{\th}}^{\bar{\infi}}\dotop \bar{S}e_{a}\underline{\fo}^{\bar{\infi}}\dotop \bar{S}e_{b}\underline{\th}^{\bar{\infi}}\otimes k{\nine}y{\underline{\o}}^{\bar{\o}}{\six}y{\underline{\t}}^{\bar{\o}}{\t}\otimes y{\underline{\th}}\\
&\langle y{\underline{\o}}^{\bar{\z}},e_{a}\underline{\th} \rangle \langle y{\underline{\t}}^{\bar{\z}},e_{b}\underline{\t}  \rangle \ \mathcal{R}(e_{b}\underline{\o}^{\bar{\o}}, e_{a}\underline{\t}^{\bar{\o}}{\fo}x{\underline{\th}}^{\bar{\o}}{\six}e_{a}\underline{\fo}^{\bar{\o}}{\fo}k{\sev}y{\underline{\o}}^{\bar{\o}}{\fo})\\
&\mathcal{R}(S(k{\ei}y{\underline{\o}}^{\bar{\o}}{\fiv}y{\underline{\t}}^{\bar{\o}}{\o}),e_{b}\underline{\th}^{\bar{\o}})\mathcal{R}(e_{a}\underline{\o}^{\bar{\o}}, x{\underline{\t}}^{\bar{\o}}{\t}x{\underline{\th}}^{\bar{\o}}{\t}k{\t})\\
&\mathcal{R}(e_{a}\underline{\t}^{\bar{\o}}{\o}, x{\underline{\th}}^{\bar{\o}}{\th}k{\th})\mathcal{R}(S(e_{a}\underline{\t}^{\bar{\o}}{\t}x{\underline{\th}}^{\bar{\o}}{\fo}e_{a}\underline{\fo}^{\bar{\o}}{\t}k{\fiv}y{\underline{\o}}^{\bar{\o}}{\t}),e_{a}\underline{\fiv}^{\bar{\o}})\\
&\mathcal{R}(S(k{\fo}y{\underline{\o}}^{\bar{\o}}{\o}),e_{a}\underline{\fo}^{\bar{\o}}{\o})\\
=&x{\underline{\o}}\otimes x{\underline{\t}}^{\bar{\o}}{\o}x{\underline{\th}}^{\bar{\o}}{\o}k{\o}\otimes f^{a} \otimes e_{a}\underline{\o}^{\bar{\infi}}\dotop x{\underline{\t}}^{\bar{\infi}}\dotop \bar{S}e_{a}\underline{\fiv}^{\bar{\infi}}\\
&\otimes x{\underline{\th}}^{\bar{\o}}{\fo}e_{a}\underline{\t}^{\bar{\o}}{\th}k{\fiv}e_{a}\underline{\fo}^{\bar{\o}}{\th}y{\underline{\o}}^{\bar{\o}}{\th} \otimes f^{b}\\
&\otimes e_{b}\underline{\o}^{\bar{\infi}}\dotop e_{a}\underline{\t}^{\bar{\infi}}\dotop x{\underline{\th}}^{\bar{\infi}}\dotop \bar{S}e_{a}\underline{\fo}^{\bar{\infi}}\dotop \bar{S}e_{b}\underline{\th}^{\bar{\infi}}\otimes k{\nine}y{\underline{\o}}^{\bar{\o}}{\six}y{\underline{\t}}^{\bar{\o}}{\t}\otimes y{\underline{\th}}\\
&\langle y{\underline{\o}}^{\bar{\z}}, e_{a}\underline{\th} \rangle \langle y{\underline{\t}}^{\bar{\z}}, e_{b}\underline{\t} \rangle \ \mathcal{R}(e_{b}\underline{\o}^{\bar{\o}}, x{\underline{\th}}^{\bar{\o}}{\fiv}e_{a}\underline{\t}^{\bar{\o}}{\fo}k{\six}e_{a}\underline{\fo}^{\bar{\o}}{\fo}y{\underline{\o}}^{\bar{\o}}{\fo})\\
&\mathcal{R}(S(k{\ei}y{\underline{\o}}^{\bar{\o}}{\fiv}y{\underline{\t}}^{\bar{\o}}{\o}),e_{b}\underline{\th}^{\bar{\o}}) \mathcal{R}(e_{a}\underline{\o}^{\bar{\o}},x{\underline{\t}}^{\bar{\o}}{\t}x{\underline{\th}}^{\bar{\o}}{\t}k{\t})\\
&\mathcal{R}(S(x{\underline{\th}}^{\bar{\o}}{\th}e_{a}\underline{\t}^{\bar{\o}}{\t}k{\fo}e_{a}\underline{\fo}^{\bar{\o}}{\t}y{\underline{\o}}^{\bar{\o}}{\t}),e_{a}\underline{\fiv}^{\bar{\o}})\mathcal{R}(e_{a}\underline{\t}^{\bar{\o}}{\o},k{\th})\\
&\mathcal{R}(e_{a}\underline{\t}^{\bar{\o}}{\fiv}, x{\underline{\th}}^{\bar{\o}}{\six})\mathcal{R}(Sy{\underline{\o}}^{\bar{\o}}{\o},e_{a}\underline{\fo}^{\bar{\o}}{\o})\mathcal{R}(Sk{\sev},e_{a}\underline{\fo}^{\bar{\o}}{\fiv}),
\end{align*}
where the second equality  uses (\ref{quabic}) to gather the parts of $e_{a}\underline{\o}^{\bar{\o}}$ and $e_{a}\underline{\fiv}^{\bar{\o}}$,  cancelling some of the $\CR$s. We lastly  use (\ref{quacom}) to change the order in the fifth tensor factor and in a similar term inside $\mathcal{R}$, again cancelling some of the  $\CR$s. On the other side,
\begin{align*}
(\Delta &\otimes \mathrm{id})\Delta(x\otimes k \otimes y)\\
=&x{\underline{\o}}\otimes x{\underline{\t}}^{\bar{\o}}{\o}x{\underline{\th}}^{\bar{\o}}{\o}k{\o}\otimes f^{b}\otimes e_{b}\underline{\o}^{\bar{\infi}}\dotop x{\underline{\t}}^{\bar{\infi}}\dotop \bar{S}e_{b}\underline{\th}^{\bar{\infi}}\otimes x{\underline{\th}}^{\bar{\o}}{\fo}k{\fo}f^{a}{\underline{\o}}^{\bar{\o}}{\t}\\
&\otimes f^{a}{\underline{\t}}\otimes e_{a}\underline{\o}^{\bar{\infi}}\dotop x{\underline{\th}}^{\bar{\infi}}\dotop \bar{S}e_{a}\underline{\th}^{\bar{\infi}}\otimes k{\sev}y{\underline{\o}}^{\bar{\o}}{\t}\otimes y{\underline{\t}} \ \langle y{\underline{\o}}^{\bar{\z}},e_{a}\underline{\t} \rangle \langle f^{a}{\underline{\o}}^{\bar{\z}}, e_{b}\underline{\t} \rangle\\
&\mathcal{R}(e_{b}\underline{\o}^{\bar{\o}}, x{\underline{\t}}^{\bar{\o}}{\t}x{\underline{\th}}^{\bar{\o}}{\t}k{\t})\mathcal{R}(S(x{\underline{\th}}^{\bar{\o}}{\th}k{\th}f^{a}{\underline{\o}}^{\bar{\o}}{\o}),e_{b}\underline{\th}^{\bar{\o}})\\
&\mathcal{R}(e_{a}\underline{\o}^{\bar{\o}}, x{\underline{\th}}^{\bar{\o}}{\fiv}k{\fiv})\mathcal{R}(S(k{\six}y{\underline{\o}}^{\bar{\o}}{}_{\o}),e_{a}\underline{\th}^{\bar{\o}})\\
=&x{\underline{\o}}\otimes x{\underline{\t}}^{\bar{\o}}{\o}x{\underline{\th}}^{\bar{\o}}{\o}k{\o}\otimes f^{b}\otimes e_{b}\underline{\o}^{\bar{\infi}}\dotop x{\underline{\t}}^{\bar{\infi}}\dotop \bar{S}e_{b}\underline{\th}^{\bar{\infi}}\otimes x{\underline{\th}}^{\bar{\o}}{\fo}k{\fo}e_{b}\underline{\t}^{\bar{\o}}{\t}\\
&\otimes f^{c} \otimes (e_{c}^{\bar{\infi}}\dotop e_{a}^{\bar{\infi}}){\underline{\o}}^{\bar{\infi}}\dotop x{\underline{\th}}^{\bar{\infi}}\bar{S}(e_{c}^{\bar{\infi}}\dotop e_{a}^{\bar{\infi}}){\underline{\th}}^{\bar{\infi}}\otimes k{\sev}y{\underline{\o}}^{\bar{\o}}{\t}\otimes y{\underline{\t}}\\
&\langle f^{a},e_{b}\underline{\t}^{\bar{\infi}} \rangle \langle y{\underline{\o}}^{\bar{\z}}, (e_{c}^{\bar{\infi}}\dotop e_{a}^{\bar{\infi}}){\underline{\t}} \rangle\ \mathcal{R}(e_{b}\underline{\o}^{\bar{\o}}, x{\underline{\t}}^{\bar{\o}}{\t}x{\underline{\th}}^{\bar{\o}}{\t}k{\t})\\
&\mathcal{R}(S(x{\underline{\th}}^{\bar{\o}}{\th}k{\th}e_{b}\underline{\t}^{\bar{\o}}{\o}),e_{b}\underline{\th}^{\bar{\o}})\mathcal{R}((e_{c}^{\bar{\infi}}\dotop e_{a}^{\bar{\infi}}){\underline{\o}}^{\bar{\o}}, x{\underline{\th}}^{\bar{\o}}{\fiv}k{\fiv})\\
&\mathcal{R}(S(k{\six}y{\underline{\o}}^{\bar{\o}}{\o}),(e_{c}^{\bar{\infi}}\dotop e_{a}^{\bar{\infi}}){\underline{\th}}^{\bar{\o}})\CR(e_c ^{\bar{\o}}, e_a ^{\bar{\o}})\\
=&x{\underline{\o}}\otimes x{\underline{\t}}^{\bar{\o}}{\o}x{\underline{\th}}^{\bar{\o}}{\o}k{\o}\otimes f^{b}\otimes e_{b}\underline{\o}^{\bar{\infi}}\dotop x{\underline{\t}}^{\bar{\infi}}\dotop \bar{S}e_{b}\underline{\fiv}^{\bar{\infi}}\\
&\otimes x{\underline{\th}}^{\bar{\o}}{\fo}k{\fo}e_{b}\underline{\t}^{\bar{\o}}{\t}e_{b}\underline{\th}^{\bar{\o}}{\t}e_{b}\underline{\fo}^{\bar{\o}}{\t}\otimes f^{c}\\
&\otimes e_{c}\underline{\o}^{\bar{\infi}}\dotop e_{b}\underline{\t}^{\bar{\infi}}\dotop x{\underline{\th}}^{\bar{\infi}}\dotop \bar{S}(e_{c}\underline{\th}^{\bar{\infi}}\dotop e_{b}\underline{\fo}^{\bar{\infi}})\otimes k{\sev}y{\underline{\o}}^{\bar{\o}}{\t}\otimes y{\t}\\
&\langle y{\underline{\o}}^{\bar{\z}}, e_{c}\underline{\t}^{\bar{\infi}}\dotop e_{b}\underline{\th}^{\bar{\infi}} \rangle \ \mathcal{R}(e_{b}\underline{\o}^{\bar{\o}}{\o}, x{\underline{\t}}^{\bar{\o}}{\t}x{\underline{\th}}^{\bar{\o}}{\t}k{\t})\\
&\mathcal{R}(S(x{\underline{\th}}^{\bar{\o}}{\th}k{\th}e_{b}\underline{\t}^{\bar{\o}}{\o}e_{b}\underline{\th}^{\bar{\o}}{\o}e_{b}\underline{\fo}^{\bar{\o}}{\o}), e_{b}\underline{\fiv}^{\bar{\o}})\mathcal{R}(e_{c}\underline{\o}^{\bar{\o}}{\t}e_{b}\underline{\t}^{\bar{\o}}{\six}, x{\underline{\th}}^{\bar{\o}}{\fiv}k{\fiv})\\
&\mathcal{R}(S(k{\six}y{\underline{\o}}^{\bar{\o}}{\o}),e_{c}\underline{\th}^{\bar{\o}}{\th}e_{b}\underline{\fo}^{\bar{\o}}{\fo})\mathcal{R}(e_{c}\underline{\o}^{\bar{\o}}{\o}e_{c}\underline{\t}^{\bar{\o}}{\o}e_{c}\underline{\th}^{\bar{\o}}{\o},e_{b}\underline{\t}^{\bar{\o}}{\th}e_{b}\underline{\th}^{\bar{\o}}{\th}e_{b}\underline{\fo}^{\bar{\o}}{\th})\\
&\mathcal{R}(Se_{c}\underline{\th}^{\bar{\o}}{\t} e_{b}\underline{\t}^{\bar{\o}}{\fo}e_{b}\underline{\th}^{\bar{\o}}{\fo})\mathcal{R}(Se_{c}\underline{\t}^{\bar{\o}}{\t},e_{b}\underline{\t}^{\bar{\o}}{\fiv}). 
\end{align*}
For the second equality we use  duality $\langle f^{a}{\underline{\o}}^{\bar{\z}},e_{b}\underline{\t} \rangle f^{a}{\underline{\o}}^{\bar{\o}} = \langle f^{a}{\underline{\o}}, e_{b}\underline{\t}^{\bar{\infi}} \rangle e_{b}\underline{\t}^{\bar{\o}}$ to replace $f^{a}{\underline{\o}}^{\bar{\o}}$ by $e_{b}\underline{\t}^{\bar{\o}}$, followed by 
\[ e_{a}\otimes f^{a}{\underline{\o}}\otimes f^{a}{\underline{\t}} = e_{c}^{\bar{\infi}}\dotop e_{a}^{\bar{\infi}}\otimes f^{a}\otimes f^{c} \ \CR(e_c ^{\bar{\o}}, e_a ^{\bar{\o}})\]
to replace $f^{a}{\underline{\o}}\tens f^{a}{\underline{\t}}$ by $f^{a}\tens f^{c}$. For the third equality, we use $\langle f^{a},e_{b}\underline{\t}^{\bar{\infi}} \rangle$ to replace $e_{a}$ by $e_{b}\underline{\t}^{\bar{\infi}}$, after which we expand $(e_{c}^{\bar{\infi}}\dotop e_{b}\underline{\t}^{\bar{\infi}\bar{\infi}})\underline{\o}$ etc. using $\underline\Delta$ a braided-homomorphism.  In the last expression, we expand $\bar{S}$ of a $\dotop$ product and  use
\begin{align*}
\langle y{\underline{\o}}^{\bar{\z}}, e_{c}\underline{\t}^{\bar{\infi}}\dotop e_{b}\underline{\th}^{\bar{\infi}} \rangle=\langle y{\underline{\o}}^{\bar{\z}}{\underline{\o}}, e_{b}\underline{\th}^{\bar{\infi}} \rangle \langle y{\underline{\o}}^{\bar{\z}}{\underline{\t}}, e_{c}\underline{\t}^{\bar{\infi}} \rangle \CR(Se_{c}\underline{\t}^{\bar{\infi}\bar{\o}}, e_{b}\underline{\th}^{\bar{\infi}\bar{\o}}).
\end{align*}
By the comodule coalgebra property (\ref{right-comod-coalg}), the first pairing on the right becomes $\langle y{\underline{\o}}^{\bar{\z}}, e_{b}\underline{\th}^{\bar{\infi}} \rangle\langle y{\underline{\t}}^{\bar{\z}}, e_{c}\underline{\t}^{\bar{\infi}} \rangle$ and duality $\langle y{\underline{\o}}^{\bar{\z}}, e_{b}\underline{\th}^{\bar{\infi}} \rangle e_{b}\underline{\th}^{\bar{\o}}=\langle y{\underline{\o}}^{\bar{\z}\bar{\z}}, e_{b}\underline{\th} \rangle y{\underline{\o}}^{\bar{\z}\bar{\o}}$ replaces $e_{b}\underline{\th}^{\bar{\o}}$ by $y{\underline{\o}}^{\bar{\z}\bar{\o}}$. The other pairing similarly  replaces $e_{c}\underline{\t}^{\bar{\o}}$  by $y{\underline{\t}}^{\bar{\z}\bar{\o}}$, so
\begin{align*}
(\Delta &\otimes \mathrm{id})\Delta(x\otimes k \otimes y)\\
=&x{\underline{\o}}\otimes x{\underline{\t}}^{\bar{\o}}{\o}x{\underline{\th}}^{\bar{\o}}{\o}k{\o}\otimes f^{b}\otimes e_{b}\underline{\o}^{\bar{\infi}}\dotop x{\underline{\t}}^{\bar{\infi}}\dotop \bar{S}e_{b}\underline{\fiv}^{\bar{\infi}}\\
&\otimes x{\underline{\th}}^{\bar{\o}}{\fo}k{\fo}e_{b}\underline{\t}^{\bar{\o}}{\t}y{\underline{\o}}^{\bar{\o}}{\t}e_{b}\underline{\fo}^{\bar{\o}}{\t}\otimes f^{c}\\
&\otimes e_{c}\underline{\o}^{\bar{\infi}}\dotop e_{b}\underline{\t}^{\bar{\infi}}\dotop x{\underline{\th}}^{\bar{\infi}}\dotop \bar{S}e_{b}\underline{\fo}^{\bar{\infi}}\dotop\bar{S}e_{c}\underline{\th}^{\bar{\infi}} \otimes k{\sev}y{\underline{\o}}^{\bar{\o}}{\sev}y{\underline{\t}}^{\bar{\o}}{\fiv}\otimes y{\th}\\
&\langle y{\underline{\o}}^{\bar{\z}}, e_{b}\underline{\th} \rangle \langle y{\underline{\t}}^{\bar{\z}}, e_{c}\underline{\t} \rangle \ \mathcal{R}(e_{b}\underline{\o}^{\bar{\o}}, x{\underline{\t}}^{\bar{\o}}{\t}x{\underline{\th}}^{\bar{\o}}{\t}k{\t})\\
&\mathcal{R}(S(x{\underline{\th}}^{\bar{\o}}{\th}k{\th}e_{b}\underline{\t}^{\bar{\o}}{\o}y{\underline{\o}}^{\bar{\o}}{\o}e_{b}\underline{\fo}^{\bar{\o}}{\o}),e_{b}\underline{\fiv}^{\bar{\o}})\mathcal{R}(e_{c}\underline{\o}^{\bar{\o}}{\t}e_{b}\underline{\t}^{\bar{\o}}{\six}, x{\underline{\th}}^{\bar{\o}}{\fiv}k{\fiv})\\
&\mathcal{R}(S(k{\six}y{\underline{\o}}^{\bar{\o}}{\six}y{\underline{\t}}^{\bar{\o}}{\fo}),e_{c}\underline{\th}^{\bar{\o}}{\th}e_{b}\underline{\fo}^{\bar{\o}}{\fo})\mathcal{R}(e_{c}\underline{\o}^{\bar{\o}}{\o}y{\underline{\t}}^{\bar{\o}}{\o}e_{c}\underline{\t}^{\bar{\o}}{\o}, e_{b}\underline{\t}^{\bar{\o}}{\th}y{\underline{\o}}^{\bar{\o}}{\th}e_{b}\underline{\fo}^{\bar{\o}}{\o})\\
&\mathcal{R}(Se_{c}\underline{\th}^{\bar{\o}}{\t},e_{b}\underline{\t}^{\bar{\o}}{\fo}y{\underline{\o}}^{\bar{\o}}{\fo})\mathcal{R}(Sy{\underline{\t}}^{\bar{\o}}{\t},e_{b}\underline{\t}^{\bar{\o}}{}{\fiv})\mathcal{R}(Sy{\underline{\t}}^{\bar{\o}}{\th},y{\underline{\o}}^{\bar{\o}}{\fiv})\mathcal{R}(Se_{c}\underline{\th}^{\bar{\o}}{\fo},e_{b}\underline{\fo}^{\bar{\o}}{\fiv})\\
=&x{\underline{\o}}\otimes x{\underline{\t}}^{\bar{\o}}{\o}x{\underline{\th}}^{\bar{\o}}{\o}k{\o}\otimes f^{b}\otimes e_{b}\underline{\o}^{\bar{\infi}}\dotop x{\underline{\t}}^{\bar{\infi}}\dotop \bar{S}e_{b}\underline{\fiv}^{\bar{\infi}}\\
&\otimes x{\underline{\th}}^{\bar{\o}}{\fo}k{\fo}e_{b}\underline{\t}^{\bar{\o}}{\t}y{\underline{\o}}^{\bar{\o}}{\t}e_{b}\underline{\fo}^{\bar{\o}}{\t}\otimes f^{c}\\
&\otimes e_{c}\underline{\o}^{\bar{\infi}}\dotop e_{b}\underline{\t}^{\bar{\infi}}\dotop x{\underline{\th}}^{\bar{\infi}}\dotop \bar{S}e_{b}\underline{\fo}^{\bar{\infi}}\dotop\bar{S}e_{c}\underline{\th}^{\bar{\infi}} \otimes k{\nine}y{\underline{\o}}^{\bar{\o}}{\six}y{\underline{\t}}^{\bar{\o}}{\t}\otimes y{\th}\\
&\langle y{\underline{\o}}^{\bar{\z}}, e_{b}\underline{\th} \rangle \langle y{\underline{\t}}^{\bar{\z}}, e_{c}\underline{\t} \rangle \ \mathcal{R}(e_{b}\underline{\o}^{\bar{\o}}, x{\underline{\t}}^{\bar{\o}}{\t}x{\underline{\th}}^{\bar{\o}}{\t}k{\t})\\
&\mathcal{R}(S(x{\underline{\th}}^{\bar{\o}}{\th}k{\th}e_{b}\underline{\t}^{\bar{\o}}{\o}y{\underline{\o}}^{\bar{\o}}{\o}e_{b}\underline{\fo}^{\bar{\o}}{\o}),e_{b}\underline{\fiv}^{\bar{\o}})\mathcal{R}(S(k{\ei}y{\underline{\o}}^{\bar{\o}}{\fiv}y{\underline{\t}}^{\bar{\o}}{\o}),e_{c}\underline{\th}^{\bar{\o}})\\
&\mathcal{R}(S(k{\sev}y{\underline{\o}}^{\bar{\o}}{\fo}),e_{b}\underline{\fo}^{\bar{\o}}{\fo})\mathcal{R}(e_{c}\underline{\o}^{\bar{\o}}, x{\underline{\th}}^{\bar{\o}}{\fiv}k{\fiv}e_{b}\underline{\t}^{\bar{\o}}{\th}y{\underline{\o}}^{\bar{\o}}{\th}e_{b}\underline{\fo}^{\bar{\o}}{\th})\\
&\mathcal{R}(e_{b}\underline{\t}^{\bar{\o}}{\fo}, x{\underline{\th}}^{\bar{\o}}{\six}k{\six}), 
\end{align*}
using (\ref{quabic}) to gather $e_{c}\underline{\o}^{\bar{\o}}$ and $e_{c}\underline{\th}^{\bar{\o}}$, and cancelling some $\mathcal{R}$s. In the final expression one can use (\ref{quacom})  to change the order in the fifth tensor factor  as well as inside $\mathcal{R}$, to recover our calculation of $(\mathrm{id}\otimes \Delta)\Delta(x\otimes k \otimes y)$ up to a change of basis labels. \end{proof}

\begin{lemma}\label{antipode}
The antipode of $B^{\mathrm{op}}\lbiprod A \rbiprod B^{*}$ in Theorem \ref{codbos} is given by
\begin{align*}
S(x\otimes k \otimes y)=&\bar{S}(e_{a}\underline{\o}^{\bar{\infi}}\dotop x^{\bar{\infi}}\dotop \bar{S}e_{a}\underline{\th}^{\bar{\infi}}) \otimes S(x^{\bar{\o}}{\t}k{\t}f^{a\bar{\o}}{\th})\otimes \underline{S}f^{a\bar{\z}}\\
&\mathcal{R}(f^{a\bar{\o}}{\o},S(x^{\bar{\o}}{\o}k{\o})) \mathcal{R}(S^{2}(k{\fiv}y^{\bar{\o}}{\t}),e_{a}\underline{\o}^{\bar{\o}}{\th}x^{\bar{\o}}{\fiv}e_{a}\underline{\th}^{\bar{\o}}{\th})\\
&\mathcal{R}(e_{a}\underline{\o}^{\bar{\o}}{\o}, x^{\bar{\o}}{\th}k{\th})\mathcal{R}(S(k{\fo}y^{\bar{\o}}{\o}), e_{a}\underline{\th}^{\bar{\o}}{\o}) \ \langle y^{\bar{\z}},e_{a}\underline{\t} \rangle\\
&v(f^{a\bar{\o}}{\t})\bar{u}(e_{a}\underline{\o}^{\bar{\o}}{\t}x^{\bar{\o}}{\fo}e_{a}\underline{\th}^{\bar{\o}}{\t}), 
\end{align*}
where $v(k)=\mathcal{R}(k{\o},Sk{\t})$ and $\bar{u}(k)=\mathcal{R}(S^{2}k{\o},k{\t})$.
\end{lemma}
\begin{proof}
We first compute $(S(x\otimes k \otimes y){\o})(x\otimes k \otimes y){\t}$, which on expanding out the product has in the first tensor factor 
\begin{align*}
\bar{S}(e_{b}&\underline{\o}^{\bar{\infi}}\dotop x{\underline{\o}}^{\bar{\infi}}\dotop \bar{S}e_{b}\underline{\th}^{\bar{\infi}})\dotop e_{a}\underline{\o}^{\bar{\infi}}\dotop x{\underline{\t}}^{\bar{\infi}}\dotop \bar{S}e_{a}\underline{\th}^{\bar{\infi}}\\
=&(\underline{\epsilon}(e_{a}\underline{\o})\underline{\epsilon}(e_{a}\underline{\th})\underline{\epsilon}(x)\underline{\epsilon}(e_{b}\underline{\o})\underline{\epsilon}(e_{b}\underline{\th}))^{\bar{\infi}}, 
\end{align*} 
which further collapses the full expression to give
\begin{align*}
(S(x\otimes & k \otimes y){\o})(x\otimes k \otimes y){\t}
=\underline{\epsilon}(x)\otimes S(k{\t}f^{b\bar{\o}}{\fo})k{\th}y{\underline{\o}}^{\bar{\o}}{\o}\otimes (\underline{S}f^{b\bar{\z}})y{\underline{\t}}\\
&\quad\quad\langle f^{a},e_{b} \rangle \langle y{\underline{\o}}^{\bar{\z}},e_{a} \rangle \ \mathcal{R}(f^{b\bar{\o}}{\t},Sk{\o})\mathcal{R}(f^{b\bar{\o}}{\o},k{\fo}y{\underline{\o}}^{\bar{\o}}{\t})v(f^{b\bar{\o}}{\th})\\
=&\underline{\epsilon}(x) \otimes (Sy{\underline{\o}}^{\bar{\z}\bar{\o}}{\fiv})(Sk{\t})k{\th}y{\underline{\o}}^{\bar{\o}}{\o}\otimes (\underline{S}y{\underline{\o}}^{\bar{\z}\bar{\z}})y{\underline{\t}}\\
&\mathcal{R}(y{\underline{\o}}^{\bar{\z}\bar{\o}}{\t}, Sk{\o})\mathcal{R}(y{\underline{\o}}^{\bar{\z}\bar{\o}}{\o}, y{\underline{\o}}^{\bar{\o}}{\t})\mathcal{R}(y{\underline{\o}}^{\bar{\z}\bar{\o}}{\t},k{\fo})v(y{\underline{\o}}^{\bar{\z}\bar{\o}}{\fo})\\
=&\underline{\epsilon}(x) \otimes \epsilon(k)(Sy{\underline{\o}}^{\bar{\o}}{\th})y{\underline{\o}}^{\bar{\o}}{\fo}\otimes (\underline{S}y{\underline{\o}}^{\bar{\z}})y{\underline{\t}} \ \mathcal{R}(y{\underline{\o}}^{\bar{\o}}{\o}, y{\underline{\o}}^{\bar{\o}}{\fiv})v(y{\underline{\o}}^{\bar{\o}}{\t})\\
=&\underline{\epsilon}(x)\otimes \epsilon(k) \otimes (\underline{S}y{\underline{\o}}^{\bar{\z}})y{\underline{\t}} \ \mathcal{R}(y{\underline{\o}}^{\bar{\o}}{\o}, y{\underline{\o}}^{\bar{\o}}{\th})v(y{\underline{\o}}^{\bar{\o}}{\t})\\ 
=&\underline{\epsilon}(x)\otimes \epsilon(k)\otimes \underline{\epsilon}(y)=\epsilon(x\otimes k \otimes y).
\end{align*}
Similarly, on computing $(x\otimes k \otimes y){\o}(S(x\otimes k \otimes y){\t})$ we have $f^{a\bar{\z}}\underline{S}f^{b\bar{\z}}=(\underline{\epsilon}f^{a}\underline{\epsilon}f^{b})^{\bar{\z}}$ in the third tensor factor which collapses the expressions to give \begin{align*}
(x\otimes &k \otimes  y){\o}(S(x\otimes k \otimes y){\t})
=x{\underline{\o}}\dotop \bar{S}x{\underline{\t}}^{\bar{\infi}\bar{\infi}}\otimes x{\underline{\t}}^{\bar{\o}}{\t}k{\t}S(x{\underline{\t}}^{\bar{\o}}{\th}k{\th}) \otimes \underline{\epsilon}y\\
&\quad\quad\quad\quad\quad\quad\quad\quad\quad\quad\quad\mathcal{R}(S^{2}k{\fo}, x{\underline{\t}}^{\bar{\o}}{\fiv})\mathcal{R}(S(x{\underline{\t}}^{\bar{\o}}{\o}k{\o}), x{\underline{\t}}^{\bar{\infi}\bar{\o}})\bar{u}(x{\underline{\t}}^{\bar{\o}}{\fo})\\
=&x{\underline{\o}}\dotop \bar{S}x{\underline{\t}}^{\bar{\infi}} \otimes x{\underline{\t}}^{\bar{\o}}{\t}k{\t}Sk{\th}Sx{\underline{\t}}^{\bar{\o}}{\th}\otimes \underline{\epsilon}y \ \mathcal{R}(S^{2}k_{\t}, x{\underline{\t}}^{\bar{\o}}{\th})\\
&\mathcal{R}(Sk{\o}Sx{\underline{\t}}^{\bar{\o}}{\o}, x{\underline{\t}}^{\bar{\o}}{\six}) \bar{u}(x{\underline{\t}}^{\bar{\o}}{\fo})\\
=&x{\underline{\o}}\dotop \bar{S}x{\underline{\t}}^{\bar{\infi}}\otimes 1 \otimes \underline{\epsilon}y \ \mathcal{R}(S^{2}k{\t}, x{\underline{\t}}^{\bar{\o}}{\th})\\
&\mathcal{R}(Sk{\o},x{\underline{\t}}^{\bar{\o}}{\fo})\mathcal{R}(Sx{\underline{\t}}^{\bar{\o}}{\o}, x{\underline{\t}}^{\bar{\o}}{\fiv})\bar{u}(x{\underline{\t}}^{\bar{\o}}{\t})\\
=&x{\underline{\o}}\dotop \bar{S}x{\underline{\t}}^{\bar{\infi}}\otimes \epsilon k \otimes \underline{\epsilon}y \ \mathcal{R}(Sx{\underline{\t}}^{\bar{\o}}{\o}, x{\underline{\t}}^{\bar{\o}}{\th})\bar{u}(x{\underline{\t}}^{\bar{\o}}{\t})
=\epsilon(x\otimes k \otimes y).
\end{align*}
\end{proof}

Finally, we show that the co-double bosonisation is coquasitriangular so as to have an inductive construction of such Hopf algebras. 

\begin{proposition}\label{codouble coquasitriangularity}
The co-double bosonisation $B^{\underline{\mathrm{op}}} \lbiprod A \rbiprod B^{*}$ is coquasitriangular with
\begin{align*}
\mathcal{R}(x\otimes k \otimes y, w \otimes \ell \otimes z)=& \langle \underline{S}z^{\bar{\z}}, x \rangle \mathcal{R}(k, \ell z^{\bar{\o}})\underline{\epsilon}(y)\underline{\epsilon}(w)
\end{align*}
for all $x,w \otimes B^{\underline{\mathrm{op}}}$, $k,\ell \in A$, and $y,z \in B^{*}$.
\end{proposition}

\begin{proof} (i) Expanding the definitions of the product and  the coquasitriangular structure,
\begin{align*}
\mathcal{R}\Big((m\otimes j &\otimes v), (x\otimes k \otimes y)(w \otimes \ell \otimes z)\Big)\\
=&\langle \underline{S}(y^{\bar{\z}\bar{\z}}z^{\bar{\z}}), m \rangle \mathcal{R}(j, k\ell{\o}y^{\bar{\z}\bar{\o}}z^{\bar{\o}})\mathcal{R}(y^{\bar{\o}}, \ell{\t})\\
=&\langle \underline{S}(y^{\bar{\z}}z^{\bar{\z}}),m\rangle \mathcal{R}(j,k\ell{\o}y^{\bar{\o}}{\o}z^{\bar{\z}}) \mathcal{R}(y^{\bar{\o}}{\t},\ell_{\t})\\
=&\langle \underline{S}z^{\bar{\z}\bar{\z}}\underline{S}y^{\bar{\z}\bar{\z}}, m \rangle \mathcal{R}(j,k\ell{\o}y^{\bar{\o}}{\o}z^{\bar{\z}}) \mathcal{R}(y^{\bar{\o}}{\t},\ell{\t}) \mathcal{R}(y^{\bar{\z}\bar{\o}}, z^{\bar{\z}\bar{\o}})\\
=&\langle \underline{S}z^{\bar{\z}}, m{\underline{\o}} \rangle \langle \underline{S}y^{\bar{\z}}, m{\underline{\t}} \rangle \mathcal{R}(j,k\ell{\o}y^{\bar{\o}}{\t}z^{\bar{\o}}{\t})\mathcal{R}(y^{\bar{\o}}{\th},\ell{\t})\mathcal{R}(y^{\bar{\o}}{\o}, z^{\bar{\o}}{\o})\\
=&\langle \underline{S}z^{\bar{\z}}, m{\underline{\o}} \rangle \langle \underline{S}y^{\bar{\z}}, m{\underline{\t}} \rangle \mathcal{R}(j, ky^{\bar{\o}}{\th}\ell{\t}z^{\bar{\o}}{\t})\mathcal{R}(y^{\bar{\o}}, z^{\bar{\o}}{\o})\mathcal{R}(y^{\bar{\o}}{\t}, \ell{\o}). 
\end{align*}
The second equality uses the right coaction on $y$. The third equality expands the braided-antipode $\underline{S}(y^{\bar{\z}}z^{\bar{\z}})$. The fourth equality uses the right-coaction on $y$ and $z$, and  evaluation. The last equality is quasicommutativity to change the order of product inside the first $\CR$. On the other side, 
\begin{align*}
\mathcal{R}\Big((m\otimes & j \otimes v){\o}, w \otimes \ell \otimes z\Big)\mathcal{R}\Big((m\otimes j \otimes v){\t}, x \otimes k \otimes y\Big)\\
=&\mathcal{R}(m{\underline{\o}}\otimes m{\underline{\t}}^{\bar{\o}}{\o}j{\o} \otimes f^{a}, w \otimes \ell \otimes z)\\
&\mathcal{R}(e_{a}\underline{\o}^{\bar{\infi}}\dotop m{\underline{\t}}^{\bar{\infi}}\dotop \bar{S}e_{a}\underline{\th}^{\bar{\infi}} \otimes j{\fo}v{\underline{\o}}^{\bar{\o}}{\t}\otimes v{\t}, x \otimes k \otimes y)\\
& \mathcal{R}(e_{a}\underline{\o}^{\bar{\o}}, m{\underline{\t}}^{\bar{\o}}{\t}j{\t})\mathcal{R}(S(j{\th}v{\underline{\o}}^{\bar{\o}}{\o}),e_{a}\underline{\th}^{\bar{\o}}) \langle v{\underline{\o}}^{\bar{\z}}, e_{a}\underline{\t} \rangle\\
=&\langle \underline{S}z^{\bar{\z}}, m{\underline{\o}} \rangle \langle \underline{S}y^{\bar{\z}}, m{\underline{\t}}^{\bar{\infi}} \rangle \mathcal{R}(m{\underline{\t}}^{\bar{\o}}j{\o}, \ell z^{\bar{\o}})\mathcal{R}(j{\t}, k y^{\bar{\o}})\\
=&\langle \underline{S}z^{\bar{\z}}, m{\underline{\o}} \rangle \langle \underline{S}y^{\bar{\z}\bar{\z}}, m{\underline{\t}} \rangle \mathcal{R}(y^{\bar{\z}\bar{\o}}j{\o}, \ell z^{\bar{\o}})\mathcal{R}(j{\t}, k y^{\bar{\o}})\\
=&\langle \underline{S}z^{\bar{\z}}, m{\underline{\o}} \rangle \langle \underline{S}y^{\bar{\z}}, m{\underline{\t}} \rangle\mathcal{R}(y^{\bar{\o}}{\o}j{\o}, \ell z^{\bar{\o}})\mathcal{R}(j{\t}, k y^{\bar{\o}}{\t}).
\end{align*}
The third equality uses $
\langle y^{\bar{\z}}, m{\underline{\t}}^{\bar{\infi}}\rangle m{\underline{\t}}^{\bar{\o}} = \langle y^{\bar{\z}\bar{\z}}, m{\underline{\t}} \rangle y^{\bar{\z}\bar{\o}}$ and the fourth uses the right coaction on $y$.  We can then use (\ref{quacom})  to gather the parts of $j$ and obtain the same expression as on the first side. (ii) Similarly expanding the definitions,
\begin{align*}
\mathcal{R}\Big((x\otimes & k \otimes y)(w\otimes \ell \otimes z), (m \otimes j \otimes v)\Big)\\
=&\langle \underline{S}v^{\bar{\z}}, x\dotop w^{\bar{\infi}} \rangle \mathcal{R}(k{\t}\ell, jv^{\bar{\o}})\mathcal{R}(Sk{\o}, w^{\bar{\o}})\\
=&\langle \underline{S}v^{\bar{\z}}, w^{\bar{\infi}\bar{\infi}}x^{\bar{\infi}} \rangle \mathcal{R}(k{\t}\ell,jv^{\bar{\o}})\mathcal{R}(Sk{\o},w^{\bar{\o}})\mathcal{R}(Sx^{\bar{\o}},w^{\bar{\infi}\bar{\o}})\\
=&\langle \underline{S}v^{\bar{\z}}{\underline{\o}}^{\bar{\z}}, x^{\bar{\infi}} \rangle \langle \underline{S}v^{\bar{\z}}{\underline{\t}}^{\bar{\z}}, w^{\bar{\infi}} \rangle \mathcal{R}(k{\t}\ell, jv^{\bar{\o}}) \mathcal{R}(Sk{\o}, w^{\bar{\o}}{\o})\\
&\mathcal{R}(Sx^{\bar{\o}}, w^{\bar{\o}}{\t}) \mathcal{R}(v^{\bar{\z}}{\underline{\o}}^{\bar{\o}},v^{\bar{\z}}{\underline{\t}}^{\bar{\o}})\\
=&\langle \underline{S}v{\underline{\o}}^{\bar{\z}}, x^{\bar{\infi}} \rangle \langle \underline{S}v{\underline{\t}}^{\bar{\z}}, w^{\bar{\infi}} \rangle \mathcal{R}(k{\t}\ell, jv{\underline{\o}}^{\bar{\o}}{\t}v{\underline{\t}}^{\bar{\o}}{\t})\mathcal{R}(Sk{\o}, w^{\bar{\o}}{\o})\\
&\mathcal{R}(Sx^{\bar{\o}}, w^{\bar{\o}}{\t})\mathcal{R}(v{\underline{\o}}^{\bar{\o}}{\o},v{\underline{\t}}^{\bar{\o}}{\o})\\
=&\langle \underline{S}v{\underline{\o}}^{\bar{\z}\bar{\z}}, x \rangle \langle \underline{S}v{\underline{\t}}^{\bar{\z}\bar{\z}}, w \rangle \mathcal{R}(k{\t}\ell, jv{\underline{\o}}^{\bar{\o}}{\t}v{\underline{\t}}^{\bar{\o}}{\t})\mathcal{R}(Sk{\o}, v{\underline{\t}}^{\bar{\z}\bar{\o}}{\o}) \\
&\mathcal{R}(Sv{\underline{\o}}^{\bar{\z}\bar{\o}}, v{\underline{\t}}^{\bar{\z}\bar{\o}}{\t})\mathcal{R}(v{\underline{\o}}^{\bar{\o}}{\o},v{\underline{\t}}^{\bar{\o}}{\o})\\
=&\langle \underline{S}v{\underline{\o}}^{\bar{\z}}, x \rangle \langle \underline{S}v{\underline{\t}}^{\bar{\z}}, w \rangle\mathcal{R}(k{\t}\ell, jv{\underline{\o}}^{\bar{\o}}{\th}v{\underline{\t}}^{\bar{\o}}{\fo}) \mathcal{R}(Sk{\o}, v{\underline{\t}}^{\bar{\o}}{\o})\\
&\mathcal{R}(Sv{\underline{\o}}^{\bar{\o}}{\o}, v{\underline{\t}}^{\bar{\o}}{\t})\mathcal{R}(v{\underline{\o}}^{\bar{\o}}{\t}, v{\underline{\t}}^{\bar{\o}}{\th})\\
=&\langle \underline{S}v{\underline{\o}}^{\bar{\z}},x \rangle \langle \underline{S}v{\underline{\t}}^{\bar{\z}}, w \rangle \mathcal{R}(k, j{\o}v{\underline{\o}}^{\bar{\o}}{\o})\mathcal{R}(\ell, j{\t}v{\underline{\o}}^{\bar{\o}}{\t}v{\underline{\t}}^{\bar{\o}}). 
\end{align*}
The second equality expands the braided-product $\dotop$. The third equality uses the left-coaction on $w$, followed by the duality pairing and taking $\underline S$ to the left in $\underline{\Delta}(\underline{S}v^{\bar{\z}})$. The fourth equality uses the  comodule coalgebra property (\ref{right-comod-coalg})  on $v$ and the right coaction axioms. The fifth equality  moves the coactions  onto $x,w$ by duality. The sixth equality is similar to the fourth. For the last equality we cancel the last two $\mathcal{R}$s and use (\ref{quabic}) to gather $k$ inside $\CR$ and cancel further. On the other side,
\begin{align*}
\mathcal{R}\Big((x&\otimes k \otimes y), (m\otimes j \otimes v){\o}\Big)\mathcal{R}\Big((w \otimes \ell \otimes z),(m\otimes j \otimes v){\t}  \Big)\\
=& \langle \underline{S}f^{a\bar{\z}}, x \rangle \langle \underline{S}v{\underline{\t}}^{\bar{\z}}, w \rangle \langle v{\underline{\o}}^{\bar{\z}}, e_{a} \rangle \mathcal{R}(k, j{\underline{\o}}f^{a\underline{\o}}) \mathcal{R}(\ell, j{\t}v{\underline{\o}}^{\bar{\o}}v{\underline{\t}}^{\bar{\o}})\\
=&\langle \underline{S}v{\underline{\o}}^{\bar{\z}\bar{\z}}, x \rangle \langle \underline{S}v{\underline{\t}}^{\bar{\z}}, w \rangle \mathcal{R}(k,j{\o}v{\underline{\o}}^{\bar{\z}\bar{\o}}) \mathcal{R}(\ell, j{\t}v{\underline{\o}}^{\bar{\o}}v{\underline{\t}}^{\bar{\o}})
\end{align*}
on substituting $f^{a}=v{\underline{\o}}^{\bar{\z}}$. We can then use the right coaction property on $v{\underline{\o}}$ to recover the result of our first calculation. (iii) We expand the definitions to compute
\begin{align*}
(x\otimes& k \otimes y){\t}(w \otimes \ell \otimes z){\t} \mathcal{R}\Big((x \otimes k \otimes y){\o},(w \otimes \ell \otimes z){\t}\Big)\\
=& x{\underline{\t}}^{\bar{\infi}}\dotop e_{b}\underline{\o}^{\bar{\infi}}\dotop w^{\bar{\infi}}\dotop \bar{S}e_{b}\underline{\th}^{\bar{\infi}} \otimes k{\th}\ell{\fo}z{\underline{\o}}^{\bar{\o}}{\t} \otimes y^{\bar{\z}} z{\underline{\t}}\\
&\mathcal{R}(Sk{\t}, e_{b}\underline{\o}^{\bar{\o}}{\t}w^{\bar{\o}}{\th}e_{b}\underline{\th}^{\bar{\o}}{\t}) \mathcal{R}(y^{\bar{\o}}, \ell{\fiv}z{\underline{\o}}^{\bar{\o}}{\th})\mathcal{R}(e_{b}\underline{\o}^{\bar{\o}}{\o}, w^{\bar{\o}}{\t}\ell{\t})\\
&\mathcal{R}(S(\ell{\th}z{\underline{\o}}^{\bar{\o}}{\o}), e_{b}\underline{\th}^{\bar{\o}}{\o}) \mathcal{R}(x{\underline{\t}}^{\bar{\o}}k{\o}, w^{\bar{\o}}{\o}\ell{\o}f^{b\bar{\o}}) \langle \underline{S}f^{b\bar{\z}}, x{\underline{\o}} \rangle \langle z{\underline{\o}}^{\bar{\z}}, e_{b}\underline{\t} \rangle\\
=&x{\underline{\fo}}^{\bar{\infi}}\dotop \bar{S}^{-1}x{\underline{\th}}^{\bar{\infi}}\dotop w^{\bar{\infi}} \dotop \bar{S}\bar{S}^{-1}x{\underline{\o}}^{\bar{\infi}}\otimes k{\underline{\th}}\ell{\fo}z{\underline{\o}}^{\bar{\o}}{\t}\otimes y^{\bar{\z}}z{\underline{\t}}\\
&\mathcal{R}(Sk{\t}, x{\th}^{\bar{\o}}{\fiv}w^{\bar{\o}}{\th}x{\underline{\o}}^{\bar{\o}}{\fo})\mathcal{R}(y^{\bar{\o}}, \ell{\fiv}z{\underline{\o}}^{\bar{\o}}{\th})\mathcal{R}(x{\underline{\th}}^{\bar{\o}}{\fo}, w^{\bar{\o}}{\t}\ell{\t})\\
&\mathcal{R}(S(\ell{\th}z{\underline{\o}}^{\bar{\o}}{\o}), x{\underline{\o}}^{\bar{\o}}{\th})\mathcal{R}(x{\underline{\fo}}^{\bar{\o}}k{\o}, w^{\bar{\o}}{\o}\ell{\o}x{\underline{\o}}^{\bar{\o}}{\o}x{\underline{\t}}^{\bar{\o}}{\o}x{\underline{\th}}^{\bar{\o}}{\o})\\
&\mathcal{R}(x{\underline{\t}}^{\bar{\o}}{\t}x{\underline{\th}}^{\bar{\o}}{\t}, x{\underline{\o}}^{\bar{\o}}{\t})\mathcal{R}(x{\underline{\th}}^{\bar{\o}}{\th}, x{\underline{\t}}^{\bar{\o}}{\th}) \langle z{\underline{\o}}^{\bar{\z}}, \bar{S}^{-1} x{\underline{\t}}^{\bar{\infi}} \rangle\\
=&w^{\bar{\infi}}\dotop x{\underline{\o}}^{\bar{\infi}}\otimes k{\th}\ell{\th}z{\underline{\o}}^{\bar{\o}}{\t}\otimes y^{\bar{\z}}z{\underline{\t}} \\
&\mathcal{R}(Sk{\t}, w^{\bar{\o}}{\t}x{\underline{\o}}^{\bar{\o}}{\fo})\mathcal{R}(y^{\bar{\o}}, \ell{\fo}z{\underline{\o}}^{\bar{\o}}{\th})\mathcal{R}(S(\ell{\t}z{\underline{\o}}^{\bar{\o}}{\o}), x{\underline{\o}}^{\bar{\o}}{\th})\\
&\mathcal{R}(k{\o},w^{\bar{\o}}{\o}\ell{\o}x{\underline{\o}}^{\bar{\o}}{\o}x{\underline{\t}}^{\bar{\o}}{\o})\mathcal{R}(x{\underline{\t}}^{\bar{\o}}{\t}, x{\underline{\o}}^{\bar{\o}}{\t}) \langle z{\underline{\o}}^{\bar{\z}}, \bar{S}^{-1}x{\underline{\t}}^{\bar{\infi}} \rangle\\
=&w^{\bar{\infi}}\dotop x{\underline{\o}}^{\bar{\infi}}\otimes k{\th}\ell{\th}z{\underline{\o}}^{\bar{\o}}{\t}\otimes y^{\bar{\z}}z{\underline{\t}}\\
&\mathcal{R}(Sk{\t}, w^{\bar{\o}}{\t}x{\underline{\o}}^{\bar{\o}}{\fo})\mathcal{R}(y^{\bar{\o}}, \ell{\fo}z{\underline{\o}}^{\bar{\o}}{\th})\mathcal{R}(S(\ell{\t}z{\underline{\o}}^{\bar{\o}}{\o})x{\underline{\o}}^{\bar{\o}}{\th})\\
&\mathcal{R}(k{\o}, w^{\bar{\o}}{\o}\ell{\o}x{\underline{\o}}^{\bar{\o}}{\o}z{\underline{\o}}^{\bar{\z}\bar{\o}}{\o})\mathcal{R}(z{\underline{\o}}^{\bar{\z}\bar{\o}}{\t}, x{\underline{\o}}^{\bar{\o}}{\t})\langle \underline{S}z{\underline{\o}}^{\bar{\z}}, x{\underline{\t}} \rangle\\
=&w\dotop x{\underline{\o}}^{\bar{\infi}}\otimes k{\t}\ell{\th}z{\underline{\o}}^{\bar{\o}}{\t}\otimes y^{\bar{\z}}z{\underline{\t}} \langle \underline{S}z{\underline{\o}}^{\bar{\z}}, x{\underline{\t}} \rangle\\
&\mathcal{R}(Sk{\t}, x{\underline{\o}}^{\bar{\o}}{\th})\mathcal{R}(y^{\bar{\o}},\ell{\fo}z{\underline{\o}}^{\bar{\o}}{\th})\mathcal{R}(S\ell{\t}, x{\underline{\o}}^{\bar{\o}}{\t}) \mathcal{R}(k{\o}, \ell{\o}x{\underline{\o}}^{\bar{\o}}{\o}z{\underline{\o}}^{\bar{\o}}{\o})\\
=&w\dotop x{\underline{\o}}^{\bar{\infi}}\otimes k{\t}\ell{\th}z{\underline{\o}}^{\bar{\o}}{\t}\otimes y^{\bar{\z}}z{\underline{\t}}\ \langle \underline{S}z{\underline{\o}}^{\bar{\z}}, x{\underline{\t}} \rangle\\
&\mathcal{R}(y^{\bar{\o}}, \ell{\fo}z{\underline{\o}}^{\bar{\o}}{\th})\mathcal{R}(S\ell{\o}, x{\underline{\o}}^{\bar{\o}})\mathcal{R}(k{\o}, \ell{\t}z{\underline{\o}}^{\bar{\o}}{\o})\\
=&w\dotop x{\underline{\o}}^{\bar{\infi}}\otimes \ell{\t}z{\underline{\o}}^{\bar{\o}}{\o}k{\o} \otimes y^{\bar{\z}}z \ \langle \underline{S}z{\underline{\o}}^{\bar{\z}}, x{\underline{\t}} \rangle\\
&\mathcal{R}(y^{\bar{\o}}, \ell{\fo}z{\underline{\o}}^{\bar{\o}}{\th})\mathcal{R}(S\ell{\o},x{\underline{\o}}^{\bar{\o}})\mathcal{R}(k{\t}, \ell{\th}z{\underline{\o}}^{\bar{\o}}{\t}). 
\end{align*}
The second equality uses the duality $\langle \underline{S}f^{b\bar{\z}}, x_{\underline{\o}} \rangle f^{b\bar{\o}} = \langle f^{b}, \bar{S}^{-1}x_{\underline{\o}}^{~~\bar{\infi}} \rangle x_{\underline{\o}}^{~~\bar{\o}}$ to substitute $e_{b}=\bar{S}^{-1}x_{\underline{\o}}^{~~\bar{\infi}}$ in all the places where it occurs, and the comodule algebra property (\ref{left-comod-coalg}). The third equality cancels $(x_{\underline{\fo}}\dotop \bar{S}^{-1}x_{\underline{\th}})^{\bar{\infi}}$ resulting in trivial coactions.  We use the duality $
\langle z{\underline{\o}}^{\bar{\z}}, \bar{S}^{-1}x{\underline{\t}}^{\bar{\infi}} \rangle x{\underline{\t}}^{\bar{\o}} = \langle \underline{S}z{\underline{\o}}^{\bar{\z}}, x{\underline{\t}} \rangle z{\underline{\o}}^{\bar{\z}\bar{\o}}$ for the fourth equality and gather  $w^{\bar{\o}}$ inside $\mathcal{R}$ to cancel it for the fifth. The sixth equality uses (\ref{quacom}) in   \[ \ell{\o}x{\underline{\o}}^{\bar{\o}}{\o}\bar{\mathcal{R}}(x{\underline{\o}}^{\bar{\o}}{\t},\ell{\t})= x{\underline{\o}}^{\bar{\o}}{\t}\ell{\t}\bar{\mathcal{R}}(x{\underline{\o}}^{\bar{\o}}{\o}, \ell{\t})\] 
and then gathers the parts of $x{\underline{\o}}^{\bar{\o}}$, and cancels some $\CR$s. We finally use (\ref{quacom}) to change the order of products in the third tensor factor. On the other side, 
\begin{align*}
(w&\otimes \ell \otimes z){\o}(x\otimes k \otimes y){\o}\mathcal{R}\Big((x\otimes k \otimes y){\t},(w\otimes \ell \otimes z){\t}\Big)\\
=&w\dotop x{\underline{\o}}^{\bar{\infi}}\otimes \ell{\t}x{\underline{\t}}^{\bar{\o}}{\o}k{\o}\otimes f^{b\bar{\z}}f^{a} \ \langle \underline{S}z{\underline{\t}}^{\bar{\z}}, e_{a}\underline{\o}^{\bar{\infi}}\dotop x{\underline{\t}}^{\bar{\infi}}\dotop \bar{S}e_{a}\underline{\th}^{\bar{\infi}} \rangle\\
&\mathcal{R}(k{\fiv}y^{\bar{\o}}{\t}, \ell{\th}z{\underline{\o}}^{\bar{\o}}z{\underline{\t}}^{\bar{\o}})\mathcal{R}(S\ell{\o}, x{\underline{\o}}^{\bar{\o}}) \mathcal{R}(f^{b\bar{\o}}, x{\underline{\t}}^{\bar{\o}}{\t}k{\t})\mathcal{R}(e_{a}\underline{\o}^{\bar{\o}}, x{\underline{\t}}^{\bar{\o}}{\th}k{\th})\\
&\mathcal{R}(S(k{\fo}y^{\bar{\o}}{\bar{\o}}),e_{a}\underline{\th}^{\bar{\o}}) \ \langle y^{\bar{\z}},e_{a}\underline{\t} \rangle \langle z{\underline{\o}}^{\bar{\z}}, e_{b} \rangle\\
=&w\dotop x{\underline{\o}}^{\bar{\infi}} \otimes \ell{\t}x{\underline{\t}}^{\bar{\o}}{\o}k{\o}\otimes z{\underline{\o}}^{\bar{\z}}f^{a} \ \langle \underline{S}z{\underline{\t}}^{\bar{\z}}, (\underline{S}^{-1}e_{a}\underline{\th}^{\bar{\infi}})x{\underline{\t}}^{\bar{\infi}}e_{a}\underline{\o}^{\bar{\infi}} \rangle\\
&\mathcal{R}(k{\fiv}y^{\bar{\o}}{\t}, \ell{\th}z{\underline{\o}}^{\bar{\o}}{\t}z{\underline{\t}}^{\bar{\o}})\mathcal{R}(S\ell{\o},x{\underline{\o}}^{\bar{\o}})\mathcal{R}(z{\underline{\o}}^{\bar{\o}}{\o}, x{\underline{\t}}^{\bar{\o}}{\t}k{\t})\\
&\mathcal{R}(e_{a}\underline{\o}^{\bar{\o}}{\o}, x{\underline{\t}}^{\bar{\o}}{\th}k{\th})\mathcal{R}(S(k{\fo} y^{\bar{\o}}{\o}),e_{a}\underline{\th}^{\bar{\o}}{\o})\mathcal{R}(Se_{a}\underline{\o}^{\bar{\o}}{\t},x{\underline{\t}}^{\bar{\o}}{\fo})\\
&\mathcal{R}(S(x{\underline{\t}}^{\bar{\o}}{\fiv}e_{a}\underline{\o}^{\bar{\o}}{\th})e_{a}\underline{\th}^{\bar{\o}}{\t}) \ \langle y^{\bar{\z}}, e_{a}\underline{\t} \rangle\\
=&w\dotop x{\underline{\o}}^{\bar{\infi}}\otimes \ell{\t}x{\underline{\t}}^{\bar{\o}}{\o}k{\o}\otimes z{\underline{\o}}^{\bar{\z}}f^{a} \ \mathcal{R}(S\ell{\o}, x{\underline{\o}}^{\bar{\o}}) \mathcal{R}(z{\underline{\o}}^{\bar{\o}}{\o}, x{\underline{\t}}^{\bar{\o}}{\t}k{\t}) \\
&\mathcal{R}(z{\underline{\t}}^{\bar{\o}}{\o}, x{\underline{\t}}^{\bar{\o}}{\th}k{\th}) \mathcal{R}(S(k{\fo} y^{\bar{\o}}{\o}), z{\underline{\o}}^{\bar{\o}}{\o})\mathcal{R}(z{\underline{\o\t}}^{\bar{\o}}{\t},x{\underline{\t}}^{\bar{\o}}{\fo})\\
&\mathcal{R}(S(x{\underline{\t}}^{\bar{\o}}{\fiv}z{\underline{\t}}^{\bar{\o}}{\th}),z{\underline{\fo}}^{\bar{\o}}{\t}) \mathcal{R}(k{\fiv}y^{\bar{\o}}{\t}, \ell{\th}z{\underline{\o}}^{\bar{\o}}{\t}z{\underline{\t}}^{\bar{\o}}{\fiv}z{\underline{\th}}^{\bar{\o}}{\th}z{\underline{\fo}}^{\bar{\o}}{\fiv})\\
&\mathcal{R}(z{\underline{\th}}^{\bar{\o}}{\o}, z{\underline{\fo}}^{\bar{\o}}{\th}) \mathcal{R}(z{\underline{\t}}^{\bar{\o}}{\fo}z{\underline{\fo}}^{\bar{\o}}{\fo}) \langle \underline{S}z{\underline{\t}}^{\bar{\z}}, e_{a}\underline{\o}^{\bar{\infi}} \rangle \langle y^{\bar{\z}},e_{a}\underline{\t} \rangle \langle \underline{S}z{\underline{\fo}}^{\bar{\z}},\underline{S}^{-1}e_{a}\underline{\th}^{\bar{\infi}} \rangle \\
&\langle \underline{S}z{\underline{\th}}^{\bar{\z}}, x{\underline{\t}}^{\bar{\infi}} \rangle\\
=&w\dotop x{\underline{\o}}^{\bar{\infi}}\otimes \ell{\t}x{\underline{\t}}^{\bar{\o}}{\o}k_{\o} \otimes y^{\bar{\z}}z{\underline{\t}}^{\bar{\z}} \ \mathcal{R}(S\ell{\o}, x{\underline{\o}}^{\bar{\o}}) \mathcal{R}(S(k{\underline{\t}}y^{\bar{\o}}{\o}),z{\underline{\t}}^{\bar{\o}}{\o})\\
&\mathcal{R}(Sx{\underline{\t}}^{\bar{\o}}{\t}, z{\underline{\t}}^{\bar{\o}}{\t}) \mathcal{R}(k{\th}y^{\bar{\o}}{\t}z{\underline{\o}}^{\bar{\o}}{\t}z{\underline{\t}}^{\bar{\o}}{\fo})\mathcal{R}(z{\underline{\o}}^{\bar{\o}}{\o},z{\underline{\t}}^{\bar{\o}}{\th}) \langle \underline{S}z{\underline{\o}}^{\bar{\z}}, x{\underline{\t}}^{\bar{\infi}} \rangle\\
=&w \dotop x{\underline{\o}}^{\bar{\infi}}\otimes \ell{\t}z{\underline{\o}}^{\bar{\o}}{\o}k{\o} \otimes y^{\bar{\z}}z{\underline{\t}}^{\bar{\z}} \ \mathcal{R}(S\ell{\o},x{\underline{\o}}^{\bar{\o}}{\o})\mathcal{R}(S(k{\t}y^{\bar{\o}}{\o}),z{\underline{\t}}^{\bar{\o}}{\o})\\
&\mathcal{R}(k{\th}y^{\bar{\o}}{\t}, \ell{\th}z{\underline{\o}}^{\bar{\o}}{\fo}z{\underline{\t}}^{\bar{\o}}{\fo}) \mathcal{R}(Sz{\underline{\o}}^{\bar{\o}}{\t}, z{\underline{\t}}^{\bar{\o}}{\t})\mathcal{R}(z{\underline{\o}}^{\bar{\o}}{\th}, z{\underline{\t}}^{\bar{\o}}{\th}) \langle \underline{S}z{\underline{\o}}^{\bar{\z}}, x{\underline{\o}}^{\bar{\o}} \rangle.
\end{align*}
The second equality uses $\langle z{\underline{\o}}^{\bar{\z}}, e_{b} \rangle$ to substitute $f^{b}=z{\underline{\o}}^{\bar{\z}}$ and we then expand $\dotop$ inside the pairing. For the third equality, we use 
\begin{align*}
\langle \underline{S}z{\underline{\t}}^{\bar{\z}}&, (\underline{S}^{-1}e_{a}\underline{\th}^{\bar{\infi}})x{\underline{\t}}^{\bar{\infi}}e_{a}\underline{\o}^{\bar{\infi}} \rangle\\
=&\langle (\underline{S}z{\underline{\t}}^{\bar{\z}}){\underline{\o}}, \underline{S}^{-1}e_{a}\underline{\th}^{\bar{\infi}} \rangle \langle (\underline{S}z{\underline{\t}}^{\bar{\z}}){\underline{\t}}, x{\underline{\t}}^{\bar{\infi}} \rangle  \langle (\underline{S}z{\underline{\t}}^{\bar{\z}})_{\underline{\th}}, e_{a}\underline{\o}^{\bar{\infi}} \rangle
\end{align*}
and move $\underline S$ to the left in $\underline{\Delta}^{2}(\underline{S}z{\underline{\t}}^{\bar{\z}})$. For the fourth equality we gather the coproducts of $e_{a}$ to give $\langle (\underline{S}z{\underline{\t}}^{\bar{\z}})y^{\bar{\z}}z{\underline{\fo}}^{\bar{\z}},e_{a} \rangle$ so that we can set $f^{a}= (\underline{S}z{\underline{\t}}^{\bar{\z}})y^{\bar{\z}}z{\underline{\fo}}^{\bar{\z}}$, allowing us to cancel $(z{\underline{\o}}\underline{S}z{\underline{\t}})^{\bar{\z}}$ and drop out following coactions. For the fifth equality, we use the duality pairing $
\langle \underline{S}z{\underline{\o}}^{\bar{\z}}, x{\underline{\t}}^{\bar{\infi}} \rangle x{\underline{\t}}^{\bar{\o}} = \langle \underline{S}z{\underline{\o}}^{\bar{\z}\bar{\z}}, x{\underline{\t}} \rangle z{\underline{\o}}^{\bar{\z}\bar{\o}}$ and then gather  $z{\underline{\t}}^{\bar{\o}}$ inside $\mathcal{R}$ so as to cancel it and recover the result of our first calculation.
\end{proof}

\section{Construction of $\cc_{q}[SL_{2}]$ by co-double bosonisation} \label{Sec4}

The coquasitriangular Hopf algebra $\mathbb{C}_{q}[SL_{2}]$ in some standard conventions is generated by $a,b,c,d$ with the relations, 
\[ba=qab, \quad ca=qac, \quad db=qbd, \quad dc=qcd, \quad cb=bc,\]
\[ad-q^{-1}bc=1, \quad da-ad=(q-q^{-1})bc,\]
a `matrix' form of coproduct (so $\Delta a=a\tens a+b\tens c$ etc.), $\eps(a)=\eps(d)=1$, $\eps(b)=\eps(c)=0$ and antipode $Sa=d, Sd=a, Sb=-qb, Sc=-q^{-1}$. The reduced version $c_{q}[SL_{2}]$ has 
\[ a^{n}=1=d^{n},\quad  b^{n}=0=c^{n}\] as additional relations when $q$ is a primitive $n$-th root of unity. We will show how some version of this
is obtained by co-double bosonisation.  Let $A=\mathbb{C}_{q}[t]/(t^{n}-1)$ be a coquasitriangular Hopf algebra with $t$ grouplike and $\mathcal{R}(t^{r},t^{s})= q^{rs}$. Also let $B=\mathbb{C}[X]/(X^{n})$ be a braided group in ${}^{A}\mathcal{M}$ with 
\[\Delta_{L} X= t \otimes X, \quad \underline{\Delta}X=1\otimes X + X \otimes 1, \quad \underline{\epsilon}X=0, \quad \underline{S}X=-X,\quad \Psi(X^{r}\tens X^{s})=q^{rs}X^{s}\otimes X^{r}.\]
The dual $B^{*}=\mathbb{C}[Y]/(Y^{n})$ lives in $\mathcal{M}^{A}$ with the same form of coproduct, etc., as for $B$, but with right-action $\Delta_R Y= Y\otimes t $. We choose pairing $\langle X, Y \rangle =1$ and take a basis of $B$ and a dual basis of $B^*$ respectively as
\[ \{e_{a}\}=\{X^a\}_{0\leq a < n},\quad \{f^{a}\}=\Big\{\dfrac{Y^a}{[a]_q!}\Big\}_{0\leq a < n},\]
where $[a]_q$ is a $q$-integers defined by $[a]_{q}=(1-q^{a})/(1-q)$ and $[a]_{q}!=[a]_{q}[a-1]_{q}\cdots [1]_{q}!$ with $[0]_{q}!=1$. We also write $\left[ a \atop r\right]_{q}=\frac{[a]_{q}!}{[r]_{q}![a-r]_{q}!}$. We write $X^{a(\underline{\mathrm{op}})} = X\dotop X \dotop \cdots \dotop X$ with $a$-many $X$, and find inductively that
\begin{align}\label{Xop}
X^{a}=&q^{\frac{a(a-1)}{2}}X^{a(\underline{\mathrm{op}})}, \quad \bar{S}(X^{a})=(-1)^{a}X^{a(\underline{\mathrm{op}})}.
\end{align}

\begin{theorem}\label{cdbthm} Let $q$ be a primitive $n$-th root of unity and $A,B,B^{*}$ be as above. 
\begin{enumerate}
\item  The co-double bosonisation of $B$, denoted $\cc_q[SL_2]$,  has generators $X,t,Y$ and 
\begin{align*}
&X^{n}=Y^{n}=0, \quad t^{n}=1, \quad YX=XY,\quad Xt=qtX,\quad Yt=qtY,\\
\Delta t&= q\sum\limits_{a=0}^{n-2}(q-1)^{a-1}(1-q^{-a-1}) tY^{a}\otimes X^{a}t\\
&={t\over q-1}\left(q{1-((q-1)Y\otimes X)^{n-1}\over 1-(q-1)Y\otimes X}- {1-((1-q^{-1})Y\otimes X)^{n-1}\over 1-(1-q^{-1})Y\otimes X} \right)t, \\
\Delta X&=X\otimes 1 + \sum\limits_{a=0}^{n-2} (q-1)^a tY^{a}\otimes X^{a+1}=X\otimes 1+t\left({1-((q-1)Y\otimes X)^{n-1}\over1-(q-1)Y\otimes X}\right)X,\\
\Delta Y&=1\otimes Y +\sum\limits_{a=0}^{n-2}(1-q^{-1})^a Y^{a+1}\otimes X^{a}t=1\otimes Y+Y\left({1-((1-q^{-1})Y\otimes X)^{n-1}\over1-(1-q^{-1})Y\otimes X}\right) t. \end{align*}
\item If $n=2m+1$, there is an isomorphism $\phi : \cc_q[SL_2] \to c_{q^{-m}}[SL_2]$ defined by
\[ \phi(X)= bd^{-1},\quad \phi(t)=d^{-2}, \quad \phi(Y)={d^{-1}c\over q^m-q^{-m}}.\]
\end{enumerate}
\end{theorem}
\begin{proof}
(1) First we determine the products
\begin{align*}
(1\otimes 1 \otimes Y)(X \otimes 1 \otimes 1)=& X^{\bar{\infi}} \otimes 1 \otimes Y^{\bar{\z}} \ \mathcal{R}(Y^{\bar{\o}},1)\mathcal{R}(S1,X^{\bar{\o}})=X\otimes 1 \otimes Y,\\
(1\otimes t \otimes 1)(X \otimes 1 \otimes 1)=& X^{\bar{\infi}} \otimes t \otimes 1 \ \mathcal{R}(St, X^{\bar{\o}})=q^{-1} X\otimes t \otimes 1,\\
(1 \otimes 1 \otimes Y)(1 \otimes t \otimes 1)=& 1 \otimes t \otimes Y^{\bar{\z}} \ \mathcal{R}(Y^{\bar{\o}},t)=q 1 \otimes t \otimes Y
\end{align*}
as stated. The algebra generated by $X,Y,t$ with these relations is $n^3$ dimensional, hence these are all the relations we need. Before go  further, we note the $q$-identities
\begin{align}\label{q-identity1}
\sum\limits_{r=0}^{a} (-1)^{r}{q^{\frac{r(r+1)}{2}}\over [r]_q! [a-r]_q!}=(1-q)^a, \  \sum\limits_{r=0}^{a} q^r(-1)^{r}{q^{\frac{r(r+1)}{2}}\over [r]_q! [a-r]_q!}=(1-q)^a[a+1]_q.
\end{align}
Then, using (\ref{Xop}), we compute
\begin{align*}
\Delta&(1\otimes t \otimes 1)=\sum\limits_{a=0}^{n-1}1 \otimes t \otimes \frac{Y^{a}}{[a]_{q}!}\otimes \big(\sum\limits_{r=0}^{a} \left[a\atop r \right]_{q} X^{r}\dotop \bar{S}X^{a-r}\otimes t \otimes 1 \ \mathcal{R}(t^{r},t)\mathcal{R}(St,t^{a-r})\big)\\
=&\sum\limits_{a=0}^{n-1}1 \otimes t \otimes \frac{Y^{a}}{[a]_{q}!}\otimes \big(\sum\limits_{r=0}^{a} \left[a\atop r \right]_{q}(-1)^{a-r} q^{\frac{r(r-1)}{2}+2r-a}X^{r(\underline{\mathrm{op}})}\dotop X^{(a-r)(\underline{\mathrm{op}})}\otimes t \otimes 1\big)\\
=& \sum\limits_{a=0}^{n-2}\sum\limits_{r=0}^{a} \frac{(-1)^{a-r}\left[ a \atop r \right]_{q} q^{\frac{r(r+1)}{2}}q^{-a+r}}{[a]_{q}!} tY^{a}\otimes X^{a}t \label{Delta t}
\end{align*}
since there is no contribution when $a=n-1$. We then use (\ref{q-identity1}). Similarly,
\begin{align*}
\Delta(X \otimes &1 \otimes 1)=X \otimes 1 \otimes 1 \otimes 1 \otimes 1 \otimes 1 \\
&+ \sum\limits_{a=0}^{n-1} 1 \otimes t \otimes \frac{Y^{a}}{[a]_{q}!}\otimes\big(\sum\limits_{r=0}^{a}\left[a \atop r\right]_{q} X^{r}\dotop X \dotop \bar{S}X^{a-r}\otimes 1 \otimes 1 \ \mathcal{R}(t^{r}, t)\big)\\
&=X \otimes 1 \otimes 1 \otimes 1 \otimes 1 \otimes 1\\
&+ \sum\limits_{a=0}^{n-1} 1 \otimes t \otimes \frac{Y^{a}}{[a]_{q}!}\otimes \big(\sum\limits_{r=0}^{a}\left[a \atop r\right]_{q} q^{\frac{r(r+1)}{2}} X^{r(\underline{\mathrm{op}})}\dotop X \dotop (-1)^{a-r}X^{(a-r)(\underline{\mathrm{op}})}\otimes 1 \otimes 1\big)\\
&= X\otimes 1 + \sum\limits_{a=0}^{n-2}\sum\limits_{r=0}^{a}\frac{(-1)^{a-r}\left[a \atop r\right]_{q} q^{\frac{r(r+1)}{2}}}{[a]_{q}!}tY^{a}\otimes X^{a+1},
\end{align*}
where for $a=n-1$, we will have the term $tY^{n-1}\otimes X^{n}=0$.
We again use (\ref{q-identity1}). Finally, we use $\underline{\Delta}^{2}(e_{a})=\underline{\Delta}^{2}(X^{a})=\sum\limits_{r=0}^{a}\sum\limits_{s=0}^{r}\left[ a \atop r\right]_{q}\left[r \atop s \right]_{q} X^{s}\otimes X^{r-s}\otimes X^{a-r}$ to find
\begin{align*}
&\Delta(1\otimes 1 \otimes Y)= 1 \otimes 1 \otimes 1 \otimes 1 \otimes 1 \otimes Y \\
&\ + \sum\limits_{a=0}^{n-1}1\otimes 1 \otimes f^{a}\otimes e_{a}\underline{\o}\dotop\bar{S}e_{a}\underline{\th}^{\bar{\infi}}\otimes t \otimes 1 \ \mathcal{R}(St, e_{a}\underline{\th}^{\bar{\o}})\langle Y, e_{a}\underline{\t} \rangle.\\
&= 1 \otimes 1 \otimes 1 \otimes 1 \otimes 1 \otimes Y \\
&\ +\sum\limits_{a=0}^{n-1}1\otimes 1 \otimes \frac{Y^{a}}{[a]_{q}!}\otimes \big(\sum\limits_{r=0}^{a}\sum\limits_{s=0}^{r}\left[ a \atop r\right]_{q}\left[r \atop s \right]_{q} X^{s}\dotop\bar{S}X^{a-r}\otimes t \otimes 1 \ \mathcal{R}(St, t^{a-r})\langle Y, X^{r-s} \rangle\big)\\
&=1 \otimes 1 \otimes 1 \otimes 1 \otimes 1 \otimes Y \\
&\ + \sum\limits_{a=0}^{n-1}1\otimes 1 \otimes \frac{Y^{a}}{[a]_{q}!}\otimes \big(\sum\limits_{r=0}^{a}\sum\limits_{s=0}^{r}\left[ a \atop r\right]_{q}\left[r \atop s \right]_{q} \delta_{1, r-s}q^{\frac{s(s-1)}{2}}q^{-a+r}X^{s(\underline{\mathrm{op}})}\dotop(-1)^{a-r}X^{(a-r)(\underline{\mathrm{op}})}\otimes t \otimes 1\big)\\
&=1\otimes Y + \sum\limits_{a=0}^{n-1}\sum\limits_{r=0}^{a}\sum\limits_{s=0}^{r} \frac{(-1)^{a-r}\delta_{1, r-s}\left[ a \atop r\right]_{q}\left[r \atop s \right]_{q}q^{\frac{s(s-1)}{2}}q^{-a+r}}{[a]_{q}!}Y^{a}\otimes X^{a-r+s}t\\
&=1\otimes Y +\sum\limits_{a=0}^{n-2}\sum\limits_{r=0}^{a}\frac{(-1)^{a-r}\left[a \atop r  \right]_{q}q^{\frac{r(r+1)}{2}}q^{-a}}{[a]_{q}!} Y^{a+1}\otimes X^{a}t. 
\end{align*}
There was no contribution from $a=0$ and for $a>0$ we needed  $s=r-1$ for a contribution. We then use (\ref{q-identity1}). The general theory in Section~\ref{Sec3} ensures that the Hopf algebra is coquasitriangular. 

(2) If $n=2m+1$ then $\varphi : c_{q^{-m}}[SL_2] \to \cc_q[SL_2]$ defined by
\[
\varphi(a)=t^{m+1}+ (q^{m}-q^{-m}) Xt^{m}Y,\quad 
\varphi(b)=Xt^{m},\quad
\varphi(c)=(q^{m}-q^{-m})t^{m}Y,\quad  
\varphi(d)=t^{m}. \]
is an algebra map and inverse to $\phi$. Tedious but straightforward calculation gives
\begin{align*}
\Delta(\varphi(d))=\Delta t^{m}=&t^{m}\otimes t^{m} + (q^{2m}-1)t^{m}Y \otimes t^{m}X=t^m\otimes t^m+(q^{m}-q^{-m})t^{m}Y \otimes Xt^{m},
\end{align*} 
to prove that $\Delta(\varphi(d))=(\varphi\tens \varphi)\Delta d$. The coalgebra map property on the other generators then follows using  this formula for $\Delta t^m$.  Furthermore, the coquasitriangular structure from Lemma \ref{codouble coquasitriangularity} computed on $\varphi(a),\varphi(b),\varphi(c),\varphi(d)$ as a matrix $\varphi(t^{i}_{~j})$ is  
\begin{align}\label{RSL2}
R^{I}{}_J=q^{m(m+1)}\begin{pmatrix}
q^{-m} & 0 & 0 & 0\\
0 & 1 & q^{-m}-q^{m} & 0\\
0 & 0 & 1 & 0\\
0 & 0 & 0 & q^{-m}
\end{pmatrix}.
\end{align}
for the values of $\CR(\varphi(t^{i}_{~j}), \varphi(t^{k}_{~l}))$ where $I=(i,k)$ is (1,1),(1,2),(2,1), or (2,2) and similarly for $J=(j,l)$. If we set $p=q^{-m}$ then any power of $p$ is also a $2m+1$-th root of unity and $q=q^{-2m}=p^{2}$ so that our Hopf algebra is $c_{p}[SL_{2}]$ with its standard coquasitriangular structure with the correct factor $q^{m(m+1)}=p^{-m-1}=p^m=p^{-{1\over 2}}$. 
\end{proof}

We now recall explicitly that for $q$ a primitive $n$-th root of unity and $q^2\ne 1$, $u_{q}(sl_{2})$ is generated by $E,F,K$, with relations, coproducts and coquasitriangular structure
\[ E^n=F^n=0, \quad K^n=1, \quad  KEK^{-1}=q^{-2}E, \quad KFK^{-1}=q^{2}F, \quad [E,F]=K-K^{-1},\]
\[\Delta K = K\otimes K, \quad \Delta F= F\otimes 1 + K^{-1}\otimes F , \quad \Delta E= E\otimes K + 1\otimes E,\]
\[\CR=\dfrac{1}{n}\sum\limits_{r,a,b=0}^{n-1}\dfrac{(-1)^r q^{-2ab}}{[r]_{q^{-2}}}F^rK^a \otimes E^r K^b,  \]
where in our conventions we do not divide by the usual $q-q^{-1}$ in the $[E,F]$-relation (and where we use $q^{-2}$ rather than $q^2$ in the remaining relations compared with \cite{Primer}). One can consider this as an unconventional normalisation of $E$ which is cleaner when we are not interested in a classical limit. It gives a commutative Hopf algebra $u_{-1}(sl_2)$ when $q=-1$.  We first show that double bosonisation gives us some version of such reduced quantum groups, agreeing for primitive odd roots. This was outlined in \cite[Example 17.6]{Primer} in the odd root case but we give a short derivation for all roots.
\begin{lemma}\label{dboslem}\cite{Primer} Let $q$ be a primitive $n$-th root of unity and let $H=\C_{q}\Z_{n}=\C_q[K]/(K^n-1)$ be a quasitriangular Hopf algebra by $\CR_K = \frac{1}{n}\sum\limits_{a,b=0}^{n-1}q^{-ab}K^{a}\otimes K^{b}$ as in \cite{Foundation}. Let $B= \C[E]/(E^{n})$ be a braided group in $\CM_{H}$ and dual $B^{*}=\C[F]/(F^{n})$ in ${}_{H}\CM$ with actions $E\lhd K =qE$ and $K \rhd F =qF$. 
\begin{enumerate}\item The double bosonisation $B^{*\cop}\lbiprod H \rbiprod B$ is a quasitriangular Hopf algebra, which we denote $\cu_q(sl_2)$, with the same coalgebra structure as above but with 
\[E^n=F^n=0, \quad K^n=1, \quad KEK^{-1}=q^{-1}E, \quad KFK^{-1}=qF, \quad [E,F]=K-K^{-1},\]
\[\CR_{\cu_q(sl_2)}=\dfrac{1}{n}\sum\limits_{r,a,b=0}^{n-1}\dfrac{(-1)^r q^{-ab}}{[r]_{q^{-1}}!}F^r K^a\otimes E^r K^b.\]
\item If $n=2m+1$  then $\cu_q(sl_2)$ is  isomorphic to $u_{q^{-m}}(sl_{2})$ with its standard quasitriangular structure.
\end{enumerate}
\end{lemma}
\begin{proof}
Here $EK\equiv (1\otimes E)(K \otimes 1)= K \otimes E \lhd K=K \otimes qE \equiv qKE$ and $KF\equiv (1\otimes K)(F\otimes 1)=K\rhd F \otimes K=qF\otimes K\equiv qFK$. From the cross relations stated in Theorem \ref{Double Bosonisation}, we also have
\begin{align*}
EF=&FE+\dfrac{1}{n}\sum\limits_{a,b=0}^{n-1}q^{-ab}K^{b}\langle F, E\lhd K^{a} \rangle - \dfrac{1}{n}\sum\limits_{a,b=0}^{n-1}q^{ab}K^{a}\langle K^{b} \rhd F,E \rangle\\
=&FE+\dfrac{1}{n}\sum\limits_{b=0}^{n-1}\left(\frac{1-q^{-n(b-1)}}{1-q^{-(b-1)}}\right)K^{b}\langle F,E \rangle-\dfrac{1}{n}\sum\limits_{a=0}^{n-1}\left(\frac{1-q^{n(a+1)}}{1-q^{a+1}}\right)K^{a}\langle F,E \rangle\\
=&FE+K-K^{-1}, 
\end{align*}
where we choose $\langle F,E \rangle = 1$. This is the same choice of normalisation for the braided line duality as in the calculation in Theorem \ref{cdbthm}. For the coproduct, clearly $\Delta K = K \otimes K$ while  $\Delta E \equiv  \Delta(1\otimes E)=1 \otimes 1 \otimes 1 \otimes E + 1 \otimes E \lhd \CR_K^{\o} \otimes \CR_K^{\t}\otimes 1 = 1\otimes 1 \otimes 1 \otimes E + 1 \otimes E \otimes K \otimes 1 \equiv 1 \otimes E + E \otimes K$ and $\Delta F \equiv \Delta(F \otimes 1)= F \otimes 1 \otimes 1 \otimes 1 + 1 \otimes \CR_K^{\mo}\otimes \CR_K^{\mt}\rhd F \otimes 1 = F \otimes 1 \otimes 1 \otimes 1 + 1 \otimes K^{-1}\otimes F \otimes 1 \equiv F\otimes 1 + K^{-1}\otimes F$. Hence we have the  relations and coalgebra as stated. Also from Theorem \ref{Double Bosonisation}, 
\[\CR_{\cu_q(sl_2)}=\sum\limits_{r=0}^{n-1}\left(\dfrac{F^r}{[r]_{q}!}\otimes\underline{S}E^r\right)\CR_K=\dfrac{1}{n}\sum\limits_{r,a,b=0}^{n-1}\dfrac{(-1)^r q^{\frac{r(r-1)}{2}}q^{-ab}}{[r]_q!}F^r K^a\otimes E^r K^b, 
\]
which we write as stated. When $n=2m+1$, it is easy to see that the relations and quasitriangular structure become those of  $u_{p}(sl_2)$ with $p=q^{-m}$, which are the same as in \cite{Primer} after allowing for the normalisation of  the generators.  Note that if $q$ is an even root of unity then $\CR_{\cu_q(sl_2)}$ need not be  factorisable, see Example \ref{exq2}. In fact, $\CR_{\cu_q(sl_2)}$ is factorisable iff $n$ is odd, which can be proven in a similar way to the proof in \cite{Lyu}. \end{proof}

We see that the double bosonisation $\cu_q(sl_2)$ recovers $u_{p}(sl_{2})$ in the odd root of unity case with $p=q^{1\over 2}$, in line with the generic $q$ case in \cite{db}. Clearly $\cu_q(sl_2)$ has a PBW-type basis $\{F^{i}K^{j}E^{k}\}_{0 \leq i,j,k \leq n-1}$ as familiar in the odd case for $u_p(sl_{2})$. 

\begin{corollary}\label{dual basis}
 The basis $\{X^{i}t^{j}Y^{k}\}_{0\leq i,j,k \leq n-1}$ of $\cc_q[SL_2]$ is, up to normalisation, dual to the PBW basis of $\cu_q(sl_{2})$ in the sense
\begin{align*}
\langle X^{i}t^{j}Y^{k}, F^{i'}K^{j'}E^{k'}  \rangle = \delta_{i,i'}\delta_{k,k'} q^{jj'}[i]_{q^{-1}}![k]_{q}!.
\end{align*}
More precisely, $\Big\{\dfrac{X^i \delta_j(t) Y^k}{[i]_{q^{-1}}![k]_{q}!}\Big\}_{0\leq i,j,k<n}$ is a dual basis to $\{F^i K^j E^k\}_{0\leq i,j,k<n}$, where $\delta_j(t)={1\over n}\sum\limits_{l=0}^{n-1}q^{-jl}t^l$.
\end{corollary} 
\begin{proof}
The duality pairing between the double and co-double bosonisations is 
\begin{align*}
\langle X^{i}t^{j}Y^{k}, F^{i'}K^{j'}E^{n'}  \rangle = \langle X^{i(\underline{\mathrm{op}})}, F^{i'} \rangle \langle t^{j}, K^{j'} \rangle \langle Y^{k}, E^{k'} \rangle, 
\end{align*}
where the pairing between $(\mathbb{C}[X]/(X^{n}))^{\mathrm{\underline{op}}}$ and $(\mathbb{C}[F]/(F^{n}))^{\mathrm{\underline{cop}}}$ implied by $\langle X, F \rangle =1$ is $\langle X^{i(\underline{\mathrm{op}})}, F^{i'} \rangle =\delta_{i,i'} [i]_{q^{-1}}!$ while $\langle t^{j}, K^{j'} \rangle =q^{jj'}$ is implied by $\langle t,K \rangle =q$. The latter is the duality pairing in the Pontryagin sense in which $\Z_n$ is self-dual, and can be written as a usual dual pairing with the $\delta_j$. Equally well, $\Big\{ \dfrac{F^i \delta_j(K)E^k}{[i]_{q^{-1}}![k]_{q}!} \Big\}_{0\leq i,j,k < n}$ is a dual basis to $\{X^i t^j Y^k\}_{0\leq i,j,k<n}$. 
\end{proof}

This applies even when  $q=-1$, in that case as a self-duality pairing.

\begin{example}\label{exq2}
If $q=-1$ then the double bosonisation $\mathbf{\cu}_{-1}(sl_2)=B^{*\mathrm{\underline{cop}}}\lbiprod H \rbiprod B$ from Lemma~\ref{dboslem} has relations and coalgebra 
\begin{align*}
&E^2=F^2=0,\quad K^2=1,\quad EF=FE, \quad KE=-EK, \quad KF=-FK,\\
&\Delta K=K\otimes K, \quad \Delta F= F\otimes 1 + K\otimes F,\quad \Delta E= E \otimes K+1\otimes E
\end{align*}
and is self-dual and strictly quasitriangular with
\[ \CR=(1\otimes 1 - F\otimes E)\CR_{K},\quad \CR_K={1\over 2}(1\tens1+1\tens K+K\tens 1-K\tens K).\]
 It is easy to check that this is not triangular, i.e $Q:=\CR_{21}\CR=1 \otimes 1-E\otimes F-KF\otimes EK-EKF\otimes FKE \ne 1\otimes 1$, and also not factorisable in the sense that the map $\mathbf{\cu}_{-1}(sl_2)^*\to \mathbf{\cu}_{-1}(sl_2)$ which sends $\phi \mapsto (\phi\otimes \mathrm{id})Q$ is not surjective (the element $FK \in \mathbf{\cu}_{-1}(sl_2)$ is not in the image). On the other hand, Theorem~\ref{cdbthm} (1) gives us an isomorphic Hopf algebra by  $X\mapsto F, Y\mapsto E$ and $t\mapsto K$, so our Hopf algebra is self-dual, i.e., $\cu_{-1}(sl_2)\cong \cc_{-1}[SL_2].$ Note that $\cu_{-1}(sl_2)$ has the same dimension and coalgebra as $u_{-1}(sl_2)$ but cannot be isomorphic, being noncommutative. One can also check that $\cc_{-1}[SL_2]$ is not isomorphic as a Hopf algebra to $c_{-1}[SL_{2}]$ and the latter, being noncocommutative, cannot be dual to $u_{-1}(sl_2)$. \end{example}

\section{Application to Hopf Algebra Fourier Transform} \label{Sec5}

As a corollary of the above results, we briefly consider Hopf algebra Fourier transform between our double and co-double bosonisations. Recall from standard Hopf algebra theory, e.g. \cite{Foundation}, that for a finite-dimensional Hopf algebra $H$ there is, up to scale, a unique right integral structure $\int : H \to \kk$ satisfying
\begin{align*}
\left(\int \otimes \mathrm{id}\right)\Delta h = \left(\int h\right)1
\end{align*}
for all $h\in H$. Such a right integral is the main ingredient for Fourier transform $\mathcal{F}: H \to H^{*}$.  The following preliminary lemma is essentially well-known (see \cite[Proposition, 1.7.7]{Foundation}), but for completeness  we give the easier part that we need.

\begin{lemma} \label{Fourier transform}
Let $\int, \int^{*}$ be right integrals on finite-dimensional Hopf algebras $H, H^{*}$ respectively and $\mu = {\int(\int^{*})}$. The Fourier transform $\mathcal{F}: H\to H^{*}$ and adjunct $\CF^*$ obey
\[\mathcal{F}(h)=\sum\limits_{a}\left(\int e_{a} h\right)f^{a}, \quad \mathcal{F}^{*}(\phi)=\sum\limits_{a} e_{a}\Big(\int^{*}\phi f^{a}\Big), \quad \CF^{*}\circ \CF = \mu S, 
\]
where $\{e_{a}\}$ is basis of $H$, $\{f^{a}\}$ is the dual basis of $H^{*}$. Hence $\CF$ is invertible if $\mu\ne 0$. 
\end{lemma}
\begin{proof}
We write $\int^{*}=\Lambda^{*}$ when  regarded as element in $H$. Then
\begin{align*}
\CF^{*}\circ \CF(h)=& e_{a}\Big(\int^{*}(\int e_{b} h)f^{b} f^{a}\Big) = (\int e_{a_{\o}} h) e_{a_{\t}}(\int ^{*}f^{a}) \\
=&(\int \Lambda^{*}{\o}h_{\o})\Lambda^{*}\t h_{\t}Sh_{\th} = (\int \Lambda^{*} h_{\o})Sh_{\t}=(\int \Lambda^{*}) Sh = \mu Sh.
\end{align*}
If $\mu\ne 0$ then this implies that $\CF$ is injective and hence in our case invertible (with a bit more work \cite{Foundation} one can show that the inverse is $\mu^{-1}S^{-1}\CF^*$). 
\end{proof}

\begin{proposition} \label{Fourier c_p[SL2]}
Let $q$ be a primitive $n$-th root of unity. The Fourier transform $\mathcal{F} : \cc_{q}[SL_{2}]\to \cu_{q}(sl_{2})$ is  invertible and given by
\[ \mathcal{F}(X^{\alpha}t^{\beta}Y^{\gamma}) = \sum\limits_{l=0}^{n-1}\dfrac{q^{-(l+\alpha)(1-\beta)+\beta(n-1-\gamma)}}{n[n-1-\alpha]_{q^{-1}}![n-1-\gamma]_{q}!}F^{n-1-\alpha}K^{l}E^{n-1-\gamma}.\]

\end{proposition}
\begin{proof}
The right integral for $\cc_{q}[SL_{2}]$ is given by.
\[\int X^\alpha t^\beta Y^\gamma = \begin{cases}
1, \text{ if $\alpha=\gamma=n-1, \beta=1$}\\
0, \text{ otherwise.}
\end{cases}\]
This integral is equivalent in usual generators to $\int b^{n-1}c^{n-1}=1$ and zero otherwise. We use Corollary \ref{dual basis} to give us the basis $\{ e_{a}\}=\{X^{i}t^{j}Y^{k}\}_{0\leq i,j,k\leq n-1}$ of $\cc_{q}[SL_2]$ and the dual basis $\{f^{a}\}=\{\frac{F^{i}\delta_{j}(K)E^{k}}{[i]_{q^{-1}}!q^{j^{2}}[k]_{q}!}\}_{0\leq i,j,k\leq n-1}$ of $\cu_{q}(sl_{2})$. Then 
\begin{align*}
\CF(X^\alpha t^\beta Y^\gamma)=&\sum\limits_{i,j,k=0}^{n-1}(\int X^i t^j Y^k X^\alpha t^\beta Y^\gamma)\dfrac{F^i \delta_{j}(K)E^k}{[i]_{q^{-1}}![k]_{q}!}\\
=&\sum\limits_{i,j,k=0}^{n-1}q^{-\alpha j+\beta k}(\int X^{i+\alpha}t^{j+\beta}Y^{k+\gamma})\dfrac{F^i \delta_{j}(K)E^k}{[i]_{q^{-1}}![k]_{q}!}\\
=&q^{-\alpha(1-\beta)+\beta(n-1-\gamma)}\dfrac{F^{2-\alpha} \delta_{1-\beta}(K)E^{n-1-\gamma}}{[n-1-\alpha]_{q^{-1}}![n-1-\gamma]_{q}!}\\
=&\sum\limits_{l=0}^{n-1}\dfrac{q^{-(l+\alpha)(1-\beta)+\beta(n-1-\gamma)}}{(n-1+1)[n-1-\alpha]_{q^{-1}}![n-1-\gamma]_{q}!}F^{n-1-\alpha}K^{l}E^{n-1-\gamma}.
\end{align*}
The similar right integral of $\cu_{q}(sl_{2})$ and resulting $\mu$ are
\[\int^* F^\alpha K^\beta E^\gamma = \begin{cases}
1 & \text{ if $\alpha=\gamma=n-1, \beta=1$}\\
0 & \text{ otherwise},
\end{cases} \quad \mu =\dfrac{q^{-1}}{n[n-1]_{q^{-1}}![n-1]_{q}!},\]
which is nonzero. 
\end{proof}

It appears to be a hard computational problem to give the general formula of the inverse Fourier transform, but one can compute it for specific cases.

\begin{example}
Let $q$ be a primitive cube root of unity. First, observe that for $\alpha,\beta=0,1,2$, we have
\begin{align*}[E^{\alpha},F^{\beta}]=F^{\beta-1}&([\alpha]_q [\beta]_q K - [\alpha]_{q^{-1}}[\beta]_{q^{-1}}K^{-1})E^{\alpha-1}\\ & + F^{\beta-2}([2]_q K - [2]_{q^{-1}}K^{-1})(K-K^{-1})E^{\alpha-2}\end{align*}
in $\cu_q(sl_2)$. Using this commutation relation, we obtain 
\begin{align*}
\CF^{*}(F^\alpha K^\beta E^\gamma)=&\sum\limits_{l=0}^{2}\dfrac{q^{\beta(2-\alpha)+(\gamma
-l)(1-\beta)}}{3[2-\alpha]_{q^{-1}}![2-\gamma]_{q}!}X^{2-\alpha}t^l Y^{2-\gamma}\\
&~+\sum\limits_{l=0}^{2}\dfrac{q^{\beta(l-\alpha-\gamma)}([\gamma]_q [3-\alpha]_q - q^{2(\gamma-l)+1}[\gamma]_{q^{-1}}[3-\alpha]_{q^{-1}})}{3[3-\alpha]_{q^{-1}}![3-\gamma]_{q}!}F^{3-\alpha}t^l Y^{3-\gamma}\\
&~-\dfrac{q^{2\beta+1}}{3[4-\alpha]_{q^{-1}}![4-\gamma]_{q}!}X^{4-\alpha}t^2 Y^{4-\gamma}. 
\end{align*}
One can check that $\CF^{*}\CF(X^\alpha t^\beta Y^\gamma)=\mu S(X^\alpha t^\beta Y^\gamma)$, where $\mu = \frac{q^{-1}}{3 [2]_{q^{-1}}![2]_{q}!}={q^2 \over 3}$ and
\begin{align*}
S(X^\alpha t^\beta Y^\gamma) =& \dfrac{q^{\alpha \beta-\beta\gamma} }{[2-\alpha]_{q^{-1}}![\alpha]_{q^{-1}}![2-\gamma]_{q}![\gamma]_{q}!}X^\alpha t^{-\delta} Y^\gamma\\
&+\dfrac{q^{\alpha \beta-\beta \gamma}(q^{2\beta-2}-q^{\delta-2})}{[2-\alpha]_{q^{-1}}![\alpha]_{q^{-1}}![1-\gamma]_q ![1+\gamma]_q!}X^{\alpha+1}t^{-\delta-1}Y^{\gamma+1}\\
&-\dfrac{1+q^{\beta+1}+q^{2\beta+2}}{[2-\alpha]_{q^{-1}}![2+\alpha]_{q^{-1}}![2-\gamma]_{q}![2+\gamma]_{q}!} X^{2+\alpha}t^2Y^{2+\gamma},
\end{align*}
where $\delta=\alpha+\beta+\gamma$.
\end{example}

\begin{example}
At $q=-1$, the Fourier transform in Proposition \ref{Fourier c_p[SL2]} combined with the self-duality in Example~\ref{exq2} becomes a Fourier transform operator $\cc_{-1}[SL_2]\to \cc_{-1}[SL_2]$. This has eigenvalues $\pm \frac{\imath}{\sqrt{2}}$ with multiplicity 2, $\pm \frac{(-1)^{1/4}}{\sqrt{2}}$ and $\pm \frac{(-1)^{3/4}}{\sqrt{2}}$ with multiplicity 1, and characteristic polynomial $f(x)=\frac{1}{16} +\frac{x^2}{4}+\frac{x^4}{2} + x^6 + x^8$. We also have
\[\CF^*(F^a K^b E^c)=\dfrac{1}{2}\sum\limits_{l=0}^{1} (-1)^{(1-b)(c-l)+b(1-a)}X^{1-a}t^l Y^{1-c} \]
and one can check that $\CF^{-1}=\mu^{-1}S^{-1}\CF^*$ as in Lemma \ref{Fourier transform}.
\end{example}

It is known that Fourier transform behaves well with respect to the coregular representation. This implies that it behaves well with respect to any covariant calculus. Thus, let $(\Omega^1,\extd)$ be a left-covariant calculus on $H$. By definition, a differential calculus means an $H$-$H$-bimodule $\Omega^1$ together with a derivation $\extd:H\to \Omega^1$ such that the map $H\tens H\to \Omega^1$ sending $h\tens g \mapsto h\extd g$ is surjective. This is left covariant if the map
\[ \Delta_L(h\extd g)=h\o g\o \tens h\t \extd g\t,\quad \Delta_L:\Omega^1\to H\tens\Omega^1\]
is well-defined. In this case it is a left coaction and $\extd$ is a comodule map with respect to the left coproduct on $H$. By the Hopf-module lemma, such $\Omega^1$ are free modules over their space $\Lambda^1$ of invariant 1-forms while  
$\extd h = h\o \varpi\pi_\eps h\t$
for all $h\in H$, where $\pi_\eps=\id-1\eps:H\to H^+$ and $\varpi:H^+\twoheadrightarrow \Lambda^1$ is the Maurer-Cartan form  $\varpi(h)=Sh\o\extd h\t$ for all $h\in H^+$. We refer to \cite{Woronowicz89,Primer} for details. The following is known, see e.g. \cite{Ma:Hod}, but we include a short derivation in our conventions. In our case $H$ is finite-dimensional. 

\begin{lemma} Let $\{e_a\}$ be a basis of $\Lambda^1$,  $\{f^a\}$ a dual basis and define partial derivatives $\partial^a:H\to H$ by $\extd h=\sum_a (\partial^a h)e_a$ and $\chi_a\in H^*{}$ by $\chi_a(h)=\<f^a,\varpi\pi_\eps S^{-1} h\>$ for all $h\in H$. Then $\CF(\partial^a h)=(\CF h)\chi_a$ for all $h\in H$. 
\end{lemma}\begin{proof}Using the right-integral property, we have
\begin{align*}\CF&(\partial^a h)=\CF(h\o) \<f^a,\varpi\pi_\eps h\t\>=\sum_b( \int e_b\o  h\o)f^b \<f^a,\varpi\pi_\eps ((S^{-1}e_b\th)e_b\t h\t\>\\
=& \sum_b (\int e_b\o  h)f^b \<f^a,\varpi\pi_\eps (S^{-1}e_b\t\>= \sum_{b,c}( \int e_b  h)f^b f^c \<f^a,\varpi\pi_\eps (S^{-1}e_c\>=(\CF h)\chi_a.
\end{align*} \end{proof}

\begin{example}
The 3D calculus c.f. \cite{Woronowicz89} has left-invariant basic 1-forms $e_\pm, e_0$ with $e_\pm h =p^{|h|}h e_\pm$ and $e_0 h=p^{2|h|}h e_0$ where $p=q^{-m}$ and $|\ |$ denotes a grading with $a,c$ grade 1 and $b,d$ grade -1 as a $\Z_n$-grading of $c_p[SL_2]$. Correspondingly for $\cc_q[SL_2]$, we have a calculus with  $|X|=0, |t|=|Y|=2$ 
and one  can compute
\[
\extd X =q^{-m} te_-, \quad \extd t= (1+q)\big(q(q^{-m}-q^m)tYe_- + te_0 \big),\]
\[\extd Y =(q^{-1}-1)^{-1}e_+ + (1+q)Ye_0 +q(q^{-m}-q^m)Y^{2}e_-, 
\]
which implies on a general monomial basis element that 
\begin{align*}
\extd (X^i t^j Y^k)=&(q^{-1}-1)^{-1}[k]_{q}X^{i}t^jY^{k-1}e_+ + (1+q)[j+k]_{q^2}X^i t^j Y^k e_0\\
&~+\Big(q(q^{-m}-q^m)[2j+k]_{q}X^i t^j Y^{k+1} + [i]_{q^{-1}}q^{-m+j+k}X^{i-1}t^{j+1}Y^k\Big)e_-.
\end{align*}
We determine $\chi_a\in u_{p}(sl_2)$ from $\langle X^i t^j Y^k, \chi_a \rangle=\epsilon(\partial^a (X^i t^j Y^k))$ with the result
\[\chi_+=\sum\limits_{j=0}^{n-1}\delta_{j}(K)E = \sum\limits_{i,j=0}^{n-1}\dfrac{(q^{-1}-1)^{-1} q^{-ij}}{n}K^i E = \dfrac{E}{(q^{-1}-1)},\]
\[\chi_0=\sum\limits_{j=0}^{n-1}(1+q)[j]_{q^{2}}\delta_{j}(K)=\sum\limits_{i,j=0}^{n-1}\dfrac{q^{-ij}(1-q^{2j})}{n(1-q)}K^{i}=\dfrac{1-K^2}{1-q}, \]
\[\chi_-= \sum\limits_{j=0}^{n-1}q^{-m+j}F \delta_{j}(K)=\sum\limits_{i,j=0}^{n-1}\frac{q^{-m}q^{j-ij}}{n} FK^i =q^{-m}FK.\]
These are versions of similar elements found for $\C_q[SU_2]$ with real $q$ in \cite{Woronowicz89}.
\end{example}

\section{Construction of $\cu_q(sl_3)$ and $\cc_q[SL_3]$ by (co)double bosonisation}\label{SecSL3} 
As mentioned in the introduction, double bosonisation can in principle be used iteratively to construct all the $u_q(\mathfrak{g})$  \cite{db2,Primer} and hence now, by making the corresponding co-double bosonisation at each step, an appropriate dual $c_q[G]$. The quantum-braided planes and their duals adjoined at each step generally have a more straightforward duality pairing given by braided factorial operators, see \cite{Ma:Hod}. Here we find
\[ \cu_q(sl_3)=\cc_q^{2}\lbiprod \widetilde{\cu_q(sl_{2})}\rbiprod \cc_q^{2}=\cc_q^{2}\lbiprod (\cc_q^{1}\lbiprod \C\Z^2_n \rbiprod \cc_q^{1})\rbiprod \cc_q^{2}\]
for certain $n$-th roots of unity, and a parallel result for $\cc_q[SL_3]$. The former was explained for generic $q$ in \cite{db} but at roots of unity we need to be much more careful.

\subsection{Construction of $\cu_q(sl_3)$ from $\cu_q(sl_2)$}\label{upsl3}
The quantum group $u_q(sl_3)$ in more or less standard conventions is generated by $E_i,F_i,K_i$ for $i=1,2$, with, c.f.  \cite{Jan}, 
\[E_i^n=F_i^n=0, \quad K_i^n=1,\]
\[K_iK_j=K_j K_i, \quad E_i K_j = q^{a_{ij}}K_i E_j, \quad K_i F_j = q^{a_{ij}}F_j K_i, \quad [E_i, F_j]=\delta_{ij}(K_i-K_i^{-1}),\]
\[\Delta K_i=K_i\otimes K_i, \quad \Delta E_i=E_i\otimes K_i + 1\otimes E_i, \quad \Delta F_i = F_i\otimes 1 + K_i^{-1}\otimes F_i,\]
where $a_{11}=a_{22}=2$ and $a_{12}=a_{21}=-1$. As before, we absorbed a factor $q-q^{-1}$ in the cross relation as a normalisation of $E_i$. We also require the $q$-Serre relations
\[ E_i^2E_j-(q+q^{-1})E_iE_jE_i+E_jE_i^2=0, \quad F_i^2F_j-(q+q^{-1})F_iF_jF_i+F_jF_i^2=0\]
for $i\ne j$. Note that  $u_q(sl_2)$ appears as a sub-Hopf algebra generated by $E_1,F_1,K_1$. 

Let $q$ be a primitive $n$-th root of unity with $n=2m+1$ and $p=q^{-m}=q^{1\over 2}$. Let $B=\cc_q^2$ be the algebra generated by $e_1,e_2$ with relation $e_2e_1=q^{-m}e_1e_2$ in the category of right $\cu_q(sl_2)$-modules. The canonical left-action of $\cu_q(sl_2)$ on $B$ is given by
\begin{align}\label{u_q(sl_2) right-action on c_q^2}
\begin{split}
\left( e_1 \atop e_2 \right) \lhd K &= \langle K, \begin{pmatrix}
a & b \\ c & d
\end{pmatrix} \rangle\left( e_1 \atop e_2 \right) = \begin{pmatrix}
q^{-m} & 0\\ 0 & q^{m}
\end{pmatrix}\left( e_1 \atop e_2 \right),\\
\left( e_1 \atop e_2 \right) \lhd E &= \langle E, \begin{pmatrix}
a & b \\ c & d
\end{pmatrix} \rangle\left( e_1 \atop e_2 \right) = \begin{pmatrix}
0 & 0\\ \lambda & 0
\end{pmatrix}\left( e_1 \atop e_2 \right), \\
\left( e_1 \atop e_2 \right) \lhd F &= \langle F, \begin{pmatrix}
a & b \\ c & d
\end{pmatrix} \rangle\left( e_1 \atop e_2 \right) = \begin{pmatrix}
0 & 1\\ 0 & 0
\end{pmatrix}\left( e_1 \atop e_2 \right),
\end{split}
\end{align}
where $\lambda = q^{m}-q^{-m}$. The duality between $\cu_q(sl_2)$ and $c_{q^{-m}}[SL_2]$ is the standard one when the former is identified with $u_{q^{-m}}(sl_2)$, or can be obtained from  Corollary~\ref{dual basis}. 

\begin{lemma}\label{braided-plane c_q^2}
Let $q$ be a primitive $n$-th root of 1 with $n=2m+1$ such that $\beta^2=3$ has a solution mod $n$. Let $H=\widetilde{\cu_{q}(sl_2)}=\cu_q(sl_2)\otimes \C_q[g]/(g^n-1)$, and let $g$ act on $e_i$ by
\[e_i\lhd g=q^{m\beta}e_i.  \]
Then  $\cc_q^2$ is a braided-Hopf algebra in the braided category of right $H$-modules with
\[e_1^n=e_2^n=0,\quad e_2e_1=q^{-m}e_1e_2, \quad \underline{\Delta}(e_i)=e_i\otimes 1 + 1 \otimes e_i, \quad \underline{\epsilon}(e_i)=0, \quad \underline{S}(e_i)=-e_i,\]
\[\Psi(e_i\otimes e_i)=q e_i \otimes e_i, \quad\Psi(e_1 \otimes e_2)=q^{-m}e_2\otimes e_1, \quad \Psi(e_2\otimes e_1)=q^{-m}e_1\otimes e_2 + (q-1) e_2 \otimes e_1.\]
 \end{lemma}
\begin{proof} The quasitriangular structure of $\widetilde{\cu_q(sl_2)}$ is given by $\CR_{\cu_q(sl_2)}\CR_g$, where $\CR_g=\frac{1}{n}\sum\limits_{s,t=0}^{n-1}q^{-st}g^s\otimes g^t$ and $\CR_{\cu_q(sl_2)}$ is given in Lemma \ref{dboslem}. Thus, we can compute that
	\[\Psi(e_i\otimes e_j)=q^{m^2\beta^2}(e_i\otimes e_j)\lhd \CR_{\cu_q(sl_2)}.\]
This braiding is equal to the correctly normalised braiding in the statement (as needed for $\underline\Delta$ to extend as a homomorphism to the braided tensor product algebra) iff $m^2\beta^2=m(m-1)$ mod $n$, or $m\beta^2=m-1$ since any $m>0$ is invertible mod $n$ (this is true for $m=1$ and if $m>1$ then $m$ and $2m+1$ are coprime).  Thus the condition for $\cc_q^2$ to form a braided-Hopf algebra in the category of $\widetilde{\cu_q(sl_2)}$-modules by an action of the stated form is $m(\beta^2-1)=-1=2m \mod n$, or $\beta^2= 3 \mod n$. Some version of this lemma was largely in \cite{EJB : bar},  working directly with $p=q^{-m}$. 
\end{proof}
 
Here $\beta=0$ is only possible for $m=1$, i.e.,  $n=3$.  In this case $\cc_q^2$ is already a braided-Hopf algebra in the category of $\cu_q(sl_2)$-modules without a central extension being needed. Otherwise, the least $n$ satisfying the condition is $n=11$ with $\beta=5$. For $n$ prime,  $\beta$ exists if and only if $n=\pm 1 \mod 12$, see \cite{NumberTheory}.

The dual $B^*=(\cc_q^2)^* \in {}_{H}M$ is generated by $f_1,f_2$ satisfying the same relations $f_2f_1=q^{-m}f_1f_2$ and additive braided coproduct as $B$ but with the  left action
\begin{align}\label{u_q(sl_2) left-action on c_q^2}
\begin{split}
K\rhd \begin{pmatrix}
f_1 & f_2
\end{pmatrix} &= \begin{pmatrix}
f_1 & f_2
\end{pmatrix} \begin{pmatrix}
q^{-m} & 0\\ 0& q^m
\end{pmatrix}, \quad g \rhd f_i = q^{m\beta} f_i,\\
E\rhd \begin{pmatrix}
f_1 & f_2
\end{pmatrix} &= \begin{pmatrix}
f_1 & f_2
\end{pmatrix}\begin{pmatrix}
0 & 0 \\ \lambda & 0
\end{pmatrix}, \quad F\rhd \begin{pmatrix}
f_1 & f_2
\end{pmatrix} = \begin{pmatrix}
f_1 & f_2
\end{pmatrix}\begin{pmatrix}
0 & 1 \\ 0 & 0
\end{pmatrix}. 
\end{split}
\end{align}

\begin{lemma}\label{exp of braided-plane} The quantum-braided planes $\cc_q^2$ and $(\cc_q^2)^*$ in Lemma~\ref{braided-plane c_q^2} are dually paired by $\langle e_1^r e_2^s,  f_1^{r'} f_2^{s'}\rangle = \delta_{r,r'}\delta_{s,s'}[r]_{q}![s]_{q}! $. \end{lemma}
\begin{proof}
It is not hard to see that $\langle e_i^r, f_i^{r'} \rangle = \delta_{r,r'}[r]_{q}! $ and this implies that
\begin{align*}
 \langle e_1^r e_2^s,  f_1^{r'} f_2^{s'}\rangle  =& \langle e_1^r \otimes e_2 ^s , \sum\limits_{r_1=0}^{r'} \sum\limits_{s_1=0}^{s'}\begin{bmatrix} 
r' \\ r_1 \end{bmatrix}_{q}\begin{bmatrix}
s' \\ s_1 \end{bmatrix}_{q} q^{-ms_1(r'-r_1)}f_1^{r_1} f_2^{s_1} \otimes f_1^{r'-r_1}f_2^{s'-s_1} \rangle\\
&=\langle e_1^r \otimes e_2 ^s, f_1^{r'} \otimes f_2^{s'} \rangle = \delta_{r,r'}\delta_{s,s'}[r]_{q}![s]_{q}!.
\end{align*}
\end{proof}

In the double bosonisation, we read the generators $e_1,e_2$ of the quantum-braided plane $B=\cc_q^2$ as $E_{12}$ and $E_2$ respectively. Similarly, the generators $f_1,f_2$ of its dual quantum-braided plane $(\cc_q^2)^*$ are read as $F_{12},F_2$ respectively. Also, we read the generators $E,F,K$ of $\cu_q(sl_2)$ as $E_1, F_1$ and $K_1$ so that
\[E_1^n=F_1^n=0, \quad K_1^n=1, \quad K_1E_1K_1^{-1}=q^{-1}E_1, \quad K_1F_1K_1^{-1}=qF_1, \quad [E_1,F_1]=K_1-K_1^{-1}.\]

\begin{lemma}\label{dbos u_q(sl_3)} Suppose the setting of Lemma~\ref{braided-plane c_q^2} with $n$ odd and $\beta^2=3$ solved mod $n$.
\begin{enumerate}
\item The double bosonisation of $\cc_q^2$, which we denote $\cu_q(sl_3)$, is generated by $E_i,F_i,K_1,g$ for $i=1,2$, with $E_1,F_1,K_1$ generating $\cu_q(sl_2)$ as a sub-Hopf algebra, and has 
\[E_2K_1=q^{m} K_1E_2, \quad E_2g=q^{m\beta}gE_2,\quad K_1F_2=q^{m}F_2 K_1, \quad gF_2=q^{m\beta}F_2g, \]
\[[E_1,F_2]=[E_2,F_1]=0, \quad [E_2,F_2]=K_1^{m}g^{m\beta}-K_1^{-m}g^{-m\beta}, \]
\[ \{E_i^2,E_j\}=(q^m+q^{-m})E_iE_jE_i,\quad \{F_i^2,F_j\}=(q^m+q^{-m})F_iF_jF_i;\quad i\ne j,\]
\[\Delta E_2 = 1\otimes E_2 + E_2\otimes K_1^{m}g^{m\beta}, \quad \Delta F_2 = F_2 \otimes 1 + g^{-m\beta}K_1^{-m}\otimes F_2,\]
\[\CR_{\cu_q(sl_3)}=\dfrac{1}{n^2}\sum\limits \dfrac{(-1)^{r+v+w}q^{vw-st-ab}}{[r]_{q^{-1}}![v]_{q^{-1}}![w]_{q^{-1}}!} F_{12}^v F_2^w F_1^r K_1^s g^{a}\otimes E_{12}^v E_2^w E_1^r K_1^t g^{b},\]
where we sum over $r,s,a,t,b,v,w$ from $0$ to $n-1$ and 
\[ E_2E_1=q^m E_1E_2+ \lambda E_{12},\quad F_1F_2=q^{-m}F_2F_1 + F_{12};\quad \lambda=q^m-q^{-m}.\]
\item If $n>3$ and is not divisible by $3$ then $\cu_q(sl_3)$ is isomorphic to $u_{q^{-m}}(sl_3)$.
\end{enumerate}
\end{lemma}

\begin{proof}
(1) This is a direct computation using Theorem \ref{Double Bosonisation}. First, we have that $E_2h=h\o(E_2\lhd h\t)$ and $hF_2=(h\o \rhd F_2)h\t$ for all $h\in \widetilde{\cu_q(sl_2)}$ and using the correct actions mentioned above. Those not involving $E_{12},F_{12}$ are as listed, while two more are regarded in the statement as definitions of $E_{12},F_{12}$ in terms of the other generators. In this case the remaining cross relations 
\[E_{12}K_1 = q^{-m} K_1E_{12},\quad K_1F_{12}=q^{-m}F_{12}K_1,\quad E_{12}g=q^{m\beta}gE_{12},\quad gF_{12}=q^{m\beta}F_{12}g\]
are all empty and can be dropped. Similarly, the first two of 
\[ [ E_{12},F_1]=K_1^{-1}E_2,\quad  [E_1,F_{12}]=\lambda F_2K_1,\quad E_{12}E_1=q^{-m}E_1E_{12},\quad F_1F_{12}=q^{m}F_{12}F_1\]
are empty and can be dropped. The remaining two and the original quantum-braided plane relations $E_{12}E_2=q^mE_2E_{12}, F_{12}F_2=q^mF_2F_{12}$ are the four $q$-Serre relations stated for $i\ne i$. 
We next look at the cross relations between the two quantum-braided planes. For example,
\[[E_2,F_2]=\CR^{\t}\langle F_2, E_2 \lhd \CR^{\o} \rangle - \CR^{\mo}\langle \CR^{\mt}\rhd F_2, E_2 \rangle \]
Putting in the form of $\CR$ and $\CR^{-1}$ gives the stated cross relation. One similarly has
\[ [E_{12}, F_2]=-E_1K_1^{m}g^{m\beta}, \quad [E_2, F_{12}]=\lambda g^{-m\beta}K_1^{-m}F_1,\quad [E_{12},F_{12}]= K_1^{-m}g^{m\beta}-K_1^{m} g^{-m\beta}\]
of which the first two are empty by a similar computation to the one above and the last is also empty by a more complicated calculation. In fact all these identities can be useful even though we do not include them in the defining relations. We also have
\begin{align*}
\Delta E_{2}&= 1\otimes E_{2} + \frac{1}{n^2}\sum\frac{(-1)^r}{[r]_{q^{-1}}!} q^{-st-ab}(E_{2} \lhd F_1^r K_1^s g^{a})\otimes E_1^r K_1^t g^{b}\\
&=1\otimes E_{2} + \frac{1}{n^2}\sum q^{-s(t-m)- a(b-m\beta)} E_{2}\otimes K_1^t g^{b} =1\otimes E_{2} + E_{2} \otimes K_1^{m}g^{m\beta}, 
\end{align*}
where we sum over $r,s,a,t,b$ from 0 to $n-1$. To compute $\Delta F_2$, we need
\[ \CR^{-1}=S\CR^{\o}\otimes \CR^{\t}=\frac{1}{n^2}\sum\frac{q^{-st- ab}}{[r]_{q^{-1}}!}g^{-a}K_1^{r-s}F_1^r \otimes E_1^r K_1^tg^{b}.\]
Only the first term contributes when acting on $F_2$, 
\begin{align*}
\Delta F_{2}&= F_{2}\otimes 1 + \frac{1}{n^2}\sum\frac{(-1)^r}{[r]_{q^{-1}}!}q^{-st-ab}g^{-a}K_1^{r-s}F_1^r\otimes E_1^rK_1^tg^{b}\rhd F_{2}\\
&= F_{2} \otimes 1 + \frac{1}{n^2}\sum q^{-t(s-m)-b(a-m\beta)}g^{-a}K_1^{-s}\otimes F_{2} =F_{2}\otimes 1 + g^{-m\beta}K_1^{-m}\otimes F_{2}
\end{align*}
and similarly for $\Delta E_2$. One also has
\[\Delta E_{12} = 1\otimes E_{12} + E_{12}\otimes K_1^{-m}g^{m\beta} - E_2\otimes E_1K_1^{m}g^{m\beta},\]
\[\Delta F_{12} = F_{12} \otimes 1 + g^{-m\beta}K_1^{m}\otimes F_{12} + \lambda g^{-m\beta}K_1^{-m}F\otimes F_2\]
which we did not state as $E_{12},F_{12}$ are not generators.  By Theorem \ref{Double Bosonisation} and Lemma \ref{exp of braided-plane}, the quasitriangular structure of $\cu_q(sl_3)$ is 
\begin{align}\label{quasitriangular uq(sl3)}
\CR_{\cu_q(sl_3)}=(\sum\limits^{n-1}_{v,w=0} \dfrac{F_{12}^v F_2^w}{[v]_q![w]_q!}\otimes \underline{S}(E_{12}^v E_2^w))\CR_{\cu_q(sl_2)}\CR_{g}, 
\end{align}
where $\CR_{\cu_q(sl_2)}\CR_{g}$ is explained in the proof of Lemma \ref{braided-plane c_q^2}. By (\ref{braidedS}) for the braided-antipode, we find
\[\underline{S}(E_{12}^v E_2^w)=(-1)^{v+w}q^{\frac{v(v-1)+w(w-1)}{2}+vw}E_{12}^w E_{2}^v,\]
so that (\ref{quasitriangular uq(sl3)}) becomes the expression stated.

(2) If $m>1$,  we define  $\varphi : u_{q^{-m}}(sl_3)\to \cu_q(sl_3)$ by 
\[ \varphi(E_i)=E_i,\quad \varphi(F_i)=F_i, \quad \varphi(K_1)=K_1,\quad \varphi(K_2) = K_1^{m}g^{m\beta}.\]
It is easy to see that $\varphi$ is an algebra and coalgebra map. In the other direction, when $m>1$,  $\beta$ is invertible mod $n$ iff $3$ is. We then define $\phi : \cu_q(sl_3) \to u_{q^{-m}}(sl_3)$ by
\[  \phi(E_i)=E_i,\quad \phi(F_i)=F_i, \quad \phi(K_1)=K_1,\quad \varphi(g)=(K_1^{-m}K_2)^{1\over m \beta},\]
which is clearly inverse to $\varphi$. 
\end{proof}

We again write $p=q^{-m}$ so that $\cu_{q}(sl_3)$ is isomorphic to $u_p(sl_3)$ under our assumptions, where $n=33$ and $\beta=6$ is the first case excluded. The double bosonisation construction also gives $\{F_{12}^{i_1}F_2^{i_2}F_1^{i_3}K_1^{i_4}g^{i_5}E_1^{i_6}E_{12}^{i_7}E_2^{i_8}\}$ as a basis of $\cu_{q}(sl_3)$. 

\begin{example}\label{dbos u_q(sl_3) at q^3=1}
As mentioned before, when $q$ is a primitive cubic root of unity i.e., when $\beta=0$, $\cc_q^2$ is already a braided-Hopf algebra in the category of $\cu_q(sl_2)$-modules without an extension needed. Then Theorem \ref{Double Bosonisation} gives us a quasitriangular Hopf algebra, which we denote $\cu_q'(sl_3)$, generated by $E_{i},F_{i},K_1$ with $i=1,2$ with the  relations and coproducts
\[E_iK_1=qK_1E_i, \quad K_1F_i=qF_iK_1,\quad  [E_i,F_j]=\delta_{i,j} (K_1-K_1^{-1}),\]
\[ \{E_i^2,E_j\}=(q+q^{-1})E_iE_jE_i,\quad \{F_i^2,F_j\}=(q+q^{-1})F_iF_jF_i;\quad i\ne j,\]
\[\Delta E_i=1\otimes E_i+E_i\otimes K_1 , \quad \Delta F_i=F_i\otimes 1+ K_1^{-1} \otimes F_i,\]
\[\CR_H=\dfrac{1}{9} \sum \dfrac{(-1)^{r+v+w} q^{vw-st}}{[r]_{q^{-1}}![v]_{q^{-1}}![w]_{q^{-1}}!} F_{12}^v F_2^w F_1^r K_1^s \otimes E_{12}^v E_2^w E_1^r K_1^t,\]
where the sum is over $r,s,t,v,w$ from 0 to 2. This $\cu_q'(sl_3)$ is not isomorphic to $u_{q^{-1}}(sl_3)$ since we do not have the generator $K_2$. However, the element $K_1^{-1}K_2$ is central and group-like in $u_{q^{-1}}(sl_3)$ and $\cu_q'(sl_3)\cong u_{q^{-1}}(sl_3)/\<K_1^{-1}K_2-1\>$. In addition, Lemma \ref{dbos u_q(sl_3)} still applies and $g$ is already group-like, and central when $\beta=0$. Therefore we have $\cu_{q}(sl_3)= \cu_q'(sl_3)\otimes \C_q[g]/(g^3-1)$ for $m=1$. 
\end{example}

\subsection{Construction of $\cc_q[SL_3]$ from $\cc_q[SL_2]$}\label{cpsl3} Recall, see e.g. \cite{Foundation}, that the coquasitriangular Hopf algebra $\C_{q}[SL_3]$ is generated by $\mathbf{t}=(t^i{}_j)$ for $i,j=1,2,3$, with matrix-form of  coproduct $\Delta \mathbf{t}=\mathbf{t}\otimes \mathbf{t}$, and for $i<k$, $j<l$,  the relations
\[[t^i{}_l, t^i{}_j]_q=0,\quad [t^k{}_j,t^i{}_j]_q=0, \quad [t^i{}_l, t^k{}_j]=0, \quad [t^k{}_l, t^i{}_l]=\lambda t^i{}_l t^k{}_j,\]
\[\mathrm{det}_q(\mathbf{t}):=t^1{}_1(t^2{}_2t^3{}_3-q^{-1}t^2{}_3t^3{}_2)-q^{-1}t^1{}_2(t^2{}_1t^3{}_3-q^{-1}t^2{}_3t^3{}_1)+q^{-2}t^1{}_3(t^2{}_1t^3{}_2-q^{-1}t^2{}_2t^3{}_1)=1,\]
where $[a,b]_q:=ba-qab$ and $\lambda = q-q^{-1}$. The reduced version is denoted by $c_q[SL_3]$ and has the additional relations
\[(t^i{}_j)^n=\delta_{ij}.\]

Throughout this section we limit ourselves to $q$ a primitive $n=2m+1$-th root of unity so that $\cc_q[SL_2]\cong c_{q^{-m}}[SL_2]$ according to Theorem~\ref{cdbthm}. Since we only consider this case, it will be convenient to use the isomorphism to define new generators $a,b,c,d$ of $\cc_q[SL_2]$ related to our previous ones by $X=bd^{-1}, t=d^{-2}$ and $Y=d^{-1}c/(q^m-q^{-m})$. Then  we can benefit from both the matrix form of coproduct on the new set and the dual basis feature of the original set.  We let $A =\widetilde{\cc_q[SL_2]}= \cc_q[SL_2]\otimes \C_{q}[\varsigma]/(\varsigma^n-1)$ be the central extension dual to $\widetilde{\cu_q(sl_2)}=\cu_q(sl_2)\otimes \C_{q}[g]/(g^n-1)$. Here  $\langle \varsigma, g \rangle=q$ and $\CR(\varsigma,\varsigma)=q$ is the coquasitriangular structure on the central extension factor. Let $B$ be a quantum-braided plane $\cc_q^2$ as in Lemma~\ref{braided-plane c_q^2} but viewed in the category of left comodules over $A$ with left coaction 
\begin{align}\label{braided-plane left coaction}
\Delta_L\begin{pmatrix}X_1 \\ X_2\end{pmatrix}=\begin{pmatrix}
\atild & \btild\\
\ctild & \dtild
\end{pmatrix}\otimes \begin{pmatrix}X_1 \\ X_2\end{pmatrix};\quad \begin{pmatrix}
\atild & \btild\\
\ctild & \dtild
\end{pmatrix}=\begin{pmatrix}
a& b \\ c& d
\end{pmatrix}\varsigma^{m\beta}, 
\end{align} 
where we now denote the generators $X_1,X_2$. In this case we will have
\[\Psi(X_i\otimes X_j)=q^{m^2\beta^2}R^j{}_k{}^i{}_l X_k \otimes X_l,\] 
where $R$ was given in  (\ref{RSL2}). We again require that $\beta^2=3 \mod n$ so that 
$q^{3m^2}R$ has the correct normalisation factor $q^{3m^2+m(m+1)}=q^{-m}$ in front of the 
matrix in (\ref{RSL2}), as needed to obtain a braided-Hopf algebra. One also has, c.f. \cite{Foundation}, 
\[\underline{\Delta}(X_1^r X_2^s) = \sum\limits_{r_1=0}^{r} \sum\limits_{s_1=0}^{s}\begin{bmatrix} 
r \\ r_1 \end{bmatrix}_{q}\begin{bmatrix}
s \\ s_1 \end{bmatrix}_{q} q^{-ms_1(r-r_1)}X_1^{r_1} X_2^{s_1} \otimes X_1^{r-r_1}X_2^{s-s_1}.\]
The dual $B^*$ was likewise explained in the previous section and is now taken with generators $Y_i$ and regarded in the category of right comodules over $A$ with
\begin{align}\label{braided-plane rigt coaction}
\Delta_R \begin{pmatrix}
Y_1 & Y_2
\end{pmatrix}=\begin{pmatrix}
Y_1 & Y_2
\end{pmatrix}\otimes\begin{pmatrix}
\atild & \btild\\
\ctild & \dtild\\
\end{pmatrix} .
\end{align}

\begin{theorem}\label{codbos cq[SL_3]} Let $n=2m+1$ such that $\beta^2=3$ is solved mod $n$. Let $A=\widetilde{\cc_q[SL_2]}=\cc_q[SL_2]\otimes \C_q[\varsigma]/(\varsigma^n-1)$ regarded with generators $\varsigma,\atild,\btild,\ctild,\dtild$. Let $B, B^*$ be quantum-braided planes with generators $X_i,Y_i$ for $i=1,2$ as above.
\begin{enumerate}
\item The co-double bosonisation, denoted $\cc_q[SL_3]$, has cross relations and coproducts
\[X_iY_j=Y_jX_i, \quad X_i \varsigma = q^{m\beta}\varsigma X_i, \quad Y_i\varsigma=q^{m\beta}\varsigma Y_i, \quad \begin{pmatrix}
\atild & \btild\\ \ctild & \dtild
\end{pmatrix} X_1 = \begin{pmatrix}
q^{-1} X_1 \atild & q^{-1}X_1 \btild \\ q^{m}X_1 \ctild & q^{m} X_1 \dtild
\end{pmatrix},\]
\[\begin{pmatrix}
\atild & \btild\\ \ctild & \dtild
\end{pmatrix} X_2 = \begin{pmatrix}
q^{m} X_2 \atild + (q^{-1}-1)X_1 \ctild & q^{m}X_2 \btild + (q^{-1}-1)X_1 \dtild \\ q^{-1}X_2 \ctild & q^{-1} X_2 \dtild
\end{pmatrix},\]
\[Y_1 \begin{pmatrix}
\atild & \btild\\ \ctild & \dtild
\end{pmatrix} = \begin{pmatrix}
q \atild Y_1 & q^{-m}\btild Y_1\\ q \ctild Y_1 & q^{-m}\dtild Y_1
\end{pmatrix}, \quad Y_2 \begin{pmatrix}
\atild & \btild\\ \ctild & \dtild
\end{pmatrix} = \begin{pmatrix}
q^{-m} \atild Y_2 + (q-1)\btild Y_1 & q\btild Y_2\\ q^{-m} \ctild Y_2+(q-1)\dtild Y_1 & q\dtild Y_2
\end{pmatrix},\]
\begin{align*}
\Delta X_1=&X_1\otimes 1 +\sum\limits_{r,s=0}^{n-1} (q-1)^{r+s} \begin{bmatrix} r+s\\ s \end{bmatrix}_{q} \big(\atild Y_1^r Y_2^s \otimes X_1^{r+1}X_2^s +q^{-mr} \btild Y_1^r Y_2^s \otimes X_1^{r}X_2^{s+1}  \big),\\
\Delta X_2 =&X_2\otimes 1+ \sum\limits_{r,s=0}^{n-1} (q-1)^{r+s} \begin{bmatrix} r+s\\ s \end{bmatrix}_{q} \big(\ctild Y_1^r Y_2^s \otimes X_1^{r+1}X_2^s +q^{-mr} \dtild Y_1^r Y_2^s \otimes X_1^{r}X_2^{s+1}  \big),\\
\Delta Y_1 =&1\otimes Y_1 +\sum\limits_{r,s=0}^{n-1} (q-1)^{r+s-1} q^{-r+ms+1} \begin{bmatrix} r+s-1\\ s \end{bmatrix}_{q}  Y_1^r Y_2^s \otimes X_1^{r-1}X_2^{s}\atild \\
&+\sum\limits_{r,s=0}^{n-1} (q-1)^{r+s-1} q^{-r-s+1} \begin{bmatrix} r+s-1\\ s-1 \end{bmatrix}_{q} Y_1^r Y_2^s \otimes X_1^{r}X_2^{s-1}\ctild, \\
\Delta Y_2 =&1\otimes Y_2 +\sum\limits_{r,s=0}^{n-1} (q-1)^{r+s-1} q^{-r+ms+1} \begin{bmatrix} r+s-1\\ s \end{bmatrix}_{q}  Y_1^r Y_2^s \otimes X_1^{r-1}X_2^{s}\btild \\
&+\sum\limits_{r,s=0}^{n-1} (q-1)^{r+s-1} q^{-r-s+1} \begin{bmatrix} r+s-1\\ s-1 \end{bmatrix}_{q} Y_1^r Y_2^s \otimes X_1^{r}X_2^{s-1}\dtild, 
\end{align*}
\[\Delta \varsigma=\sum\limits_{r,s=0}^{n-1}q^{-m\beta(r+s)}(q-1)^{r+s}\begin{bmatrix}
r+2m\beta-1\\r
\end{bmatrix}_q \begin{bmatrix}
r+s+2m\beta-1\\s
\end{bmatrix}_q \varsigma Y_1^r Y_2^s \otimes X_1^r X_2^s \varsigma, \]
\begin{align*}
\Delta \atild=& \sum\limits_{r,s=0}^{n-1} (q-1)^{r+s} q^{-r-s} \begin{bmatrix} r+s\\ s \end{bmatrix}_{q} \atild Y_1^r Y_2^s \otimes \big( q^{-ms}[r+1]_{q} X_1^r X_2^s\atild  + [s]_{q}  X_1^{r+1}X_2^{s-1}\ctild \big)\\
&+\sum\limits_{r,s=0}^{n-1} (q-1)^{r+s} q^{m(r+s)} \begin{bmatrix} r+s\\ s \end{bmatrix}_{q} \btild Y_1^r Y_2^s \otimes \big(q^{-m}[r]_{q}  X_1^{r-1}X_2^{s+1}\atild + \otimes q^s X_1^{r}X_2^{s}\ctild \big),\\
\Delta \btild=& \sum\limits_{r,s=0}^{n-1} (q-1)^{r+s} q^{-r-s} \begin{bmatrix} r+s\\ s \end{bmatrix}_{q} \atild Y_1^r Y_2^s \otimes \big( q^{-ms}[r+1]_{q} X_1^r X_2^s\btild  + [s]_{q}  X_1^{r+1}X_2^{s-1}\dtild \big)\\
&+\sum\limits_{r,s=0}^{n-1} (q-1)^{r+s} q^{m(r+s)} \begin{bmatrix} r+s\\ s \end{bmatrix}_{q} \btild Y_1^r Y_2^s \otimes \big(q^{-m}[r]_{q}  X_1^{r-1}X_2^{s+1}\btild + \otimes q^s X_1^{r}X_2^{s}\dtild \big),\\
\Delta \ctild=& \sum\limits_{r,s=0}^{n-1} (q-1)^{r+s} q^{-r-s} \begin{bmatrix} r+s\\ s \end{bmatrix}_{q} \ctild Y_1^r Y_2^s \otimes \big( q^{-ms}[r+1]_{q} X_1^r X_2^s\atild  + [s]_{q}  X_1^{r+1}X_2^{s-1}\ctild \big)\\
&+\sum\limits_{r,s=0}^{n-1} (q-1)^{r+s} q^{m(r+s)} \begin{bmatrix} r+s\\ s \end{bmatrix}_{q} \dtild Y_1^r Y_2^s \otimes \big(q^{-m}[r]_{q}  X_1^{r-1}X_2^{s+1}\atild + \otimes q^s X_1^{r}X_2^{s}\ctild \big),\\
\Delta \dtild=& \sum\limits_{r,s=0}^{n-1} (q-1)^{r+s} q^{-r-s} \begin{bmatrix} r+s\\ s \end{bmatrix}_{q} \ctild Y_1^r Y_2^s \otimes \big( q^{-ms}[r+1]_{q} X_1^r X_2^s\btild  + [s]_{q}  X_1^{r+1}X_2^{s-1}\dtild \big)\\
&+\sum\limits_{r,s=0}^{n-1} (q-1)^{r+s} q^{m(r+s)} \begin{bmatrix} r+s\\ s \end{bmatrix}_{q} \dtild Y_1^r Y_2^s \otimes \big(q^{-m}[r]_{q}  X_1^{r-1}X_2^{s+1}\btild + \otimes q^s X_1^{r}X_2^{s}\dtild \big).
\end{align*}
\item If $n>3$ and is not divisible by $3$ then $\cc_q[SL_3]$ is isomorphic to $c_{q^{-m}}[SL_3]$ with its standard coquasitriangular structure.
\end{enumerate}
\end{theorem}
\begin{proof}
(1) From $X_2 \dotop X_1 = q X_1 \dotop X_2$, we work inductively and find that
\[X_1^{r(\op)}\dotop X_2^{s(\op)} = q^{\frac{-(r+s)(r+s-1)}{2}}X_1^rX_2^s, \quad \bar{S}(X_1^{r(\op)}\dotop X_2^{s(\op)})=(-1)^{r+s}q^{\frac{-(r+s)(r+s-1)}{2}}X_1X_2\]
where $X_1^{r(\op)}$ means $X_1 \dotop X_1 \dotop \cdots $ $r$-times. We also need that
\[ \Delta_L(X_1^r) = \sum\limits_{r_1=0}^{r}\begin{bmatrix}
r\\ r_1
\end{bmatrix}_{q} \atild^{r_1}\btild^{r-r_1}\otimes X_1^{r_1}X_2^{r-r_1}, \quad  \Delta_L(X_2^s) = \sum\limits_{s_1=0}^{s}\begin{bmatrix}
s\\ s_1
\end{bmatrix}_{q} \ctild^{r_1}\dtild^{s-s_1}\otimes X_1^{s_1}X_2^{s-s_1}\]
and that $\zeta$ commutes with $\atild,\btild,\ctild,\dtild$. Then computation from Theorem~\ref{codbos} gives
\[X_iY_j=Y_jX_i, \quad X_i \varsigma = q^{m\beta}\varsigma X_i, \quad Y_i\varsigma=q^{m\beta}\varsigma Y_i, \quad \begin{pmatrix}
a & b\\ c & d
\end{pmatrix} X_1 = \begin{pmatrix}
 q^{-m^2} X_1 a & q^{-m^2}X_1 b \\ q^{m^2}X_1 c & q^{m^2} X_1 d
\end{pmatrix},\]
\[\begin{pmatrix}
a & b\\ c & d
\end{pmatrix} X_2 = \begin{pmatrix}
q^{m^2} (X_2 a + (q^{m}-q^{-m})X_1 c) & q^{m^2}(X_2 b + (q^m-q^{-m})X_1 d) \\ q^{-m^2}X_2 c & q^{-m^2} X_2 d
\end{pmatrix},\]
\[Y_1 \begin{pmatrix}
a & b\\ c & d
\end{pmatrix} = \begin{pmatrix}
q^{m^2} a Y_1 & q^{-m^2}b Y_1\\ q^{m^2} c Y_1 & q^{-m^2}d Y_1
\end{pmatrix}, \ Y_2 \begin{pmatrix}
a & b\\ c & d
\end{pmatrix} = \begin{pmatrix}
q^{-m^2}(a Y_2 - (q^m-q^{-m})b Y_1)  & q^{m^2}b Y_2\\ q^{-m^2} (c Y_2-(q^m-q^{-m})d Y_1) & q^{m^2}d Y_2
\end{pmatrix}\]
and hence the relations stated. The algebra generated by $X_i, Y_i, \varsigma, \atild, \btild, \ctild, \dtild$ is $n^8$ dimensional as required for these to be all the relations. For the coproduct, we use Lemma~\ref{exp of braided-plane} to provide a basis and dual basis of $\cc_q^2$ and $(\cc_q^2)^*$. Then
\begin{align*}
\Delta& X_1 =X_1\otimes 1\\
&+ \sum\limits_{r,s=0}^{n-1}\sum\limits_{r_1=0}^{r}\sum\limits_{s_1=0}^{s}\begin{bmatrix}
r\\r_1 \end{bmatrix}_{q}\begin{bmatrix} s\\s_1 \end{bmatrix}_{q} (-1)^{r+s-r_1-s_1}q^{ s_1(r-r_1)+\frac{(r_1+s_1)(r_1+s_1+1)}{2}+r_1+s_1} \dfrac{\atild Y_1^r Y_2^s}{[r]_{q}![s]_{q}!}\otimes X_1^{r+1} X_2^s\\
&+\sum\limits_{r,s=0}^{n-1}\sum\limits_{r_1=0}^{r}\sum\limits_{s_1=0}^{s}\begin{bmatrix}
r\\r_1 \end{bmatrix}_{q}\begin{bmatrix} s\\s_1 \end{bmatrix}_{q} (-1)^{r+s-r_1-s_1}q^{s_1(r-r_1)+\frac{(r_1+s_1)(r_1+s_1-1)}{2}+s_1-mr}(1+[r_1]_q(q-1))\\
&\kern11cm \dfrac{\btild Y_1^r Y_2^s}{[r]_{q}![s]_{q}!}\otimes X_1^{r} X_2^{s+1}
\end{align*}
and similarly for $\Delta X_2$. Likewise,
\begin{align*}
\Delta Y_1 =&1\otimes Y_1+ \sum\limits_{r,s=0}^{n-1}\sum\limits_{r_1=0}^{r}\sum\limits_{s_1=0}^{s}\begin{bmatrix}
r\\r_1 \end{bmatrix}_{q}\begin{bmatrix} s\\s_1 \end{bmatrix}_{q}[r_1]_q (-1)^{r+s-r_1-s_1}q^{s_1(r-r_1)+\frac{(r_1+s_1-1)(r_1+s_1-2)}{2}-r+ms+r_1+s_1}\\
&\kern10cm \dfrac{Y_1^r Y_2^s}{[r]_{q}![s]_{q}!}\otimes X_1^{r-1}X_2^s \atild\\
&+\sum\limits_{r,s=0}^{n-1}\sum\limits_{r_1=0}^{r}\sum\limits_{s_1=0}^{s}\begin{bmatrix}
r\\r_1 \end{bmatrix}_{q}\begin{bmatrix} s\\s_1 \end{bmatrix}_{q}q^{r_1+s_1-r-s+\frac{(r_1+s_1)(r_1+s_1-1)}{2}}(-1)^{r+s-r_1-s_1} ([s_1]_q+[r_1]_q [s-s_1]_q)\\
&\kern10cm\dfrac{Y_1^r Y_2^s}{[r]_{q}![s]_{q}!}\otimes X_1^{r}X_2^{s-1} \ctild
\end{align*}
and similarly for $\Delta Y_2$. Next, we have
\begin{align*}
\Delta \varsigma=& \sum\limits_{r,s=0}^{n-1}\sum\limits_{r_1=0}^{r}\sum\limits_{s_1=0}^{s}\begin{bmatrix}
r\\r_1\end{bmatrix}_q \begin{bmatrix}s\\s_1 \end{bmatrix}_q q^{s_1(r-r1)+m\beta(2r_1+2s_1-r-s)+\frac{(r_1+s_1)(r_1+s_1-1)}{2}}(-1)^{r+s-r_1-s_1}\dfrac{\varsigma Y_1^r Y_2^s}{[r]_q![s]_q!}\otimes X_1^r X_2^s \varsigma.
\end{align*}
\begin{align*}
\Delta \atild&=\sum\limits_{r,s=0}^{n-1}\sum\limits_{r_1=0}^{r}\sum\limits_{s_1=0}^{s}\begin{bmatrix}
r\\r_1 \end{bmatrix}_{q}\begin{bmatrix} s\\s_1 \end{bmatrix}_{q}(-1)^{r+s-r_1-s_1} q^{s_1(r-r_1) +2r_1+s_1-r+ms+\frac{(r_1+s1)(r_1+s_1-1)}{2}}\\
&\kern10cm\dfrac{\atild Y_1^r Y_2^s}{[r]_{q}![s]_{q}!}\otimes X_1^r X_2^s \atild\\
&+\sum\limits_{r,s=0}^{n-1}\sum\limits_{r_1=0}^{r}\sum\limits_{s_1=0}^{s}\begin{bmatrix}
r\\r_1 \end{bmatrix}_{q}\begin{bmatrix} s\\s_1 \end{bmatrix}_{q}[s-s_1]_{q}(-1)^{r+s-r_1-s_1}q^{s_1(r-r_1)+2r_1+2s_1-r-s+\frac{(r_1+s1)(r_1+s_1-1)}{2}}\\
&\kern9cm(1-q) \dfrac{\atild Y_1^r Y_2^s}{[r]_{q}![s]_{q}!}\otimes X_1^{r+1}X_2^{s-1}\ctild\\
&+\sum\limits_{r,s=0}^{n-1}\sum\limits_{r_1=0}^{r}\sum\limits_{s_1=0}^{s}\begin{bmatrix}
r\\r_1 \end{bmatrix}_{q}\begin{bmatrix} s\\s_1 \end{bmatrix}_{q}[r_1]_{q}(-1)^{r+s-r_1-s_1}q^{s_1(r-r_1)+r1+s1+mr+ms+\frac{(r_1+s1)(r_1+s_1-1)}{2}}\\
&\kern8cm(q^{-m}-q^{m})\dfrac{\btild Y_1^r Y_2^s}{[r]_{q}![s]_{q}!}\otimes X_1^{r-1}X_2^{s+1}\atild\\
&+\sum\limits_{r,s=0}^{n-1}\sum\limits_{r_1=0}^{r}\sum\limits_{s_1=0}^{s}\begin{bmatrix}
r\\r_1 \end{bmatrix}_{q}\begin{bmatrix} s\\s_1 \end{bmatrix}_{q}(-1)^{r+s-r_1-s_1}q^{s_1(r-r_1)+\frac{(r_1+s1)(r_1+s_1-1)}{2}+r_1+2s_1+mr-s}\\
&\kern6cm(1-[r_1]_{q}[s-s_1]_{q}(q-1)^2)\dfrac{\btild Y_1^r Y_2^s}{[r]_{q}![s]_{q}!}\otimes X_1^{r}X_2^{s}\ctild
\end{align*}
and similarly for $\Delta \btild, \Delta \ctild, \Delta \dtild$. The stated coproducts follow from the $q$-identity
\begin{align}\label{q-identity2}
\sum\limits_{r_1=0}^{r}\frac{(-1)^{r-r_1}q^{\frac{r_1(r_1+1)}{2}}q^{s r_1}}{[r_1]_{q}![r-r_1]_q!}=(q-1)^{r}\begin{bmatrix}
r+s\\r
\end{bmatrix}_q
\end{align}
for all $r,s$ (of which (\ref{q-identity1}) are specials cases) and further calculation. 

(2) If $n>3$ and is not divisible by $3$ then $\beta$ is invertible mod $n$. We define $\varphi:c_{q^{-m}}[SL_3]\to \cc_q[SL_3]$ by
\[\varphi(t^1{}_1)=a\varsigma^{\frac{m}{\beta}}+ \lambda X_1, \quad \varphi(t^1{}_2)=b\varsigma^{\frac{m}{\beta}}+ \lambda X_1, \quad  \varphi(t^1{}_3)=X_1\varsigma^{\frac{1}{\beta}},\]
\[\varphi(t^2{}_1)=c\varsigma^{\frac{m}{\beta}}+ \lambda X_2, \quad \varphi(t^2{}_2)=d\varsigma^{\frac{m}{\beta}}+ \lambda X_2 \varsigma^{\frac{1}{\beta}}Y_2, \quad \varphi(t^2{}_3)=X_2\varsigma^{\frac{1}{\beta}},\]
\[\varphi(t^3{}_1)=\lambda \varsigma^{\frac{1}{\beta}}Y_1, \quad \varphi(t^3{}_2)=(q^m-q^{-m})\varsigma^{\frac{1}{\beta}}Y_2, \quad  \varphi(t^3{}_3)= \varsigma^{\frac{1}{\beta}},
\]
where $\lambda = q^m-q^{-m}$. A tedious calculation shows that this extends as an algebra map and is a coalgebra map. In the other direction, we define $\phi : \cc_q[SL_3] \to c_{q^{-m}}[SL_3]$ by
\[\phi(\varsigma)=(t^3{}_3)^{\beta}, \quad \phi(X_1)=t^1{}_3 (t^3{}_3)^{-1}, \quad \phi(X_2)=t^2{}_3(t^3{}_3)^{-1},\]
\[\phi(Y_1)=(q^m-q^{-m})^{-1} (t^3{}_3)^{-1} t^3{}_1, \quad \phi(Y_1)=(q^m-q^{-m})^{-1} (t^3{}_3)^{-1} t^3{}_2,\]
\[\phi(a)=t^1{}_1(t^3{}_3)^{-m}-q^m t^1{}_3 t^3{}_1(t^3{}_3)^m, \quad \phi(b)=t^1{}_2(t^3{}_3)^{-m}-q^m t^1{}_3 t^3{}_2(t^3{}_3)^m,\]
\[\phi(c)=t^2{}_1(t^3{}_3)^{-m}-q^m t^2{}_3 t^3{}_1(t^3{}_3)^m, \quad \phi(d)=t^2{}_2(t^3{}_3)^{-m}-q^m t^2{}_3 t^3{}_2(t^3{}_3)^m\]
as inverse to $\varphi$. Although one can verify these matters directly, the map $\varphi$ was obtained as adjoint to the isomorphism $u_{q^{-m}}(sl_3)\to \cu_q(sl_3)$ in part (2) of Lemma~\ref{dbos u_q(sl_3)} as follows. The standard duality between $u_{q^{-m}}(sl_3)$ and  $c_{q^{-m}}[SL_3]$ is by
\[ \langle \mathbf{t},F_1\rangle = \mathrm{\mathbf{e}}_{12}, \quad \langle \mathbf{t},F_2\rangle = \mathrm{\mathbf{e}}_{23}, \quad \langle \mathbf{t},F_{12}\rangle = \mathrm{\mathbf{e}}_{13}, \quad \langle \mathbf{t},E_1\rangle = \lambda \mathrm{\mathbf{e}}_{21}, \quad \langle \mathbf{t},E_2\rangle = \lambda \mathrm{\mathbf{e}}_{32},\]\[  \langle \mathbf{t}, E_{12} \rangle= \lambda \mathrm{\mathbf{e}}_{31},\quad 
\langle \mathbf{t},K_1\rangle = \begin{pmatrix}
 q^{-m} & 0 & 0 \\
 0 & q^m & 0 \\
 0 & 0 & 1 
\end{pmatrix}, \quad \langle \mathbf{t},K_2\rangle = \begin{pmatrix}
 1 & 0 & 0 \\
 0 & q^{-m} & 0 \\
 0 & 0 & q^m 
\end{pmatrix},
\]
where $\mathrm{\mathbf{e}}_{ij}$ is an elementary matrix with entry 1 in $(i,j)$-position and 0 elsewhere. The duality between $\cu_q(sl_3)$ and $\cc_q[SL_3]$ is part of our construction with a natural basis of $\cc_q[SL_3]$ built from bases of $\cc_q^2, (\cc_q^2)^*$ and $\widetilde{\cc_q[SL_2]}=\cc_q[SL_2]\otimes \C_{q}[\varsigma]/(\varsigma^n-1)$. The first tensor factor here has a basis of monomials in $X,t,Y$ by Theorem~\ref{cdbthm}. Therefore we have $\{X_1^{i_1}X_2^{i_2}X^{i_3}t^{j_1}\varsigma^{j_2}Y^{k_1}Y_1^{k_2}Y_2^{k_3}\}$ as a basis of $\cc_q[SL_3]$  essentially dual to the PBW basis of $\cu_q(sl_3)$ in the sense
\begin{align*}
\langle X_1^{i_1}X_2^{i_2}&X^{i_3}t^{j_1}\varsigma^{j_2}Y^{k_1}Y_1^{k_2}Y_2^{k_3},F_{12}^{i'_1}F_2^{i'_2}F_1^{i'_3}K_1^{j'_1}g^{j'_2}E_1^{k'_1}E_{12}^{k'_2}E_2^{k'_3}\rangle\\
&=\delta_{i_1i'_1}\delta_{i_2i'_2}\delta_{i_3i'_3}\delta_{k_1k'_1}\delta_{k_2k'_2}\delta_{k_3k'_3}[i_1]_{q^{-1}}![i_2]_{q^{-1}}![i_3]_{q^{-1}}!q^{j_1j'_1+j_2j'_2}[k_1]_{q}![k_2]_{q}![k_3]_{q}!.
\end{align*}
This is the dual basis result for $\cu_q(sl_3)$ and $\cc_q[SL_3]$ analogous to Corollary \ref{dual basis} in the $sl_2$ case. Hence the coefficients of $\varphi(t^i{}_j)$ in this basis of $\cc_q[SL_3]$ will be picked out by evaluation against the dual basis $F_{12}^{i_1}F_2^{i_2}F_1^{i_3}\delta_{j_1}(K_1)\delta_{j_2}(g)E_1^{k_1}E_{12}^{k_2}E_2^{k_3}$, where $\delta_j(K_1),\delta_j(g)$ are defined as in Corollary~\ref{dual basis}. These values are given by the matrix representation as above except that we still need the matrix representation of $g$. From  Lemma \ref{dbos u_q(sl_3)} we recall that that $\cu_q(sl_3) \cong u_{q^{-m}}(sl_3)$ with $g\mapsto(K^{-m}K_2)^{\frac{1}{m\beta}}$, hence we have $\langle \mathbf{t},g \rangle= {\rm diag}(q^{\frac{ m}{\beta}}, q^{\frac{ m}{\beta}}, q^{\frac{ 1}{\beta}})$. This gives, for example, 
\begin{align*} \langle \varphi(t^1{}_1), F_{12}^{i_1}F_2^{i_2}F_1^{i_3}\delta_{j_1}(K)\delta_{j_2}(g)E_1^{k_1}&E_{12}^{k_2}E_2^{k_3} \rangle=\delta_{i_1,0}\delta_{i_2,0}\delta_{i_3,0}\delta_{j_1,-m}\delta_{j_2,\frac{m}{\beta}}\delta_{k_1,0}\delta_{k_2,0}\delta_{k_3,0}\\
&+\delta_{i_1,0}\delta_{i_2,0}\delta_{i_3,1}\delta_{j_1,m}\delta_{j_2,\frac{m}{\beta}}\delta_{k_1,1}\delta_{k_2,0}\delta_{k_3,0}(q^m-q^{-m})\\
&+\delta_{i_1,0}\delta_{i_2,1}\delta_{i_3,0}\delta_{j_1,0}\delta_{j_2,\frac{1}{\beta}}\delta_{k_1,0}\delta_{k_2,0}\delta_{k_3,1}(q^m-q^{-m}), 
\end{align*}
which by summing against the dual basis implies that 
\[\varphi(t^1{}_1)=t^{-m}\varsigma^{\frac{m}{\beta}}+(q^m-q^{-m})Xt^m\varsigma^{\frac{m}{\beta}}Y+(q^m-q^{-m})X_1\varsigma^{\frac{1}{\beta}}Y_1.\]
We then convert over to the $a,b,c,d$ generators as discussed. 

Finally, the coquasitriangular structure of $\cc_q[SL_3]$  computed using  Lemma~\ref{codouble coquasitriangularity} and pulled back to the $c_{q^{-m}}[SL_3]$ generators is $\CR(\varphi(t^i{}_j),\varphi(t^k{}_l))=R^I{}_J$, where $I=(i,k)$, $J=(j,l)$ are taken in lexicographic order $(1,1),(1,2),\cdots,(3,3)$ and
\[R^I{}_J= q^{\frac{m}{3}}\begin{pmatrix}
q^{-m} & 0 & 0 & 0 & 0 & 0 & 0 & 0 & 0 \\
0 & 1& 0 & q^{-m}-q^m  & 0 & 0 & 0 & 0 & 0\\
0 & 0 & 1 & 0 & 0 & 0 & q^{-m}-q^m &0 &0\\
0 & 0 & 0 & 1 & 0 & 0 & 0 & 0 & 0\\
0 & 0 & 0 & 0 & q^{-m} & 0 & 0 & 0 & 0\\
0 & 0 & 0 & 0 & 0 & 1 & 0 & q^{-m}-q^m &0\\
0 & 0 & 0 & 0 & 0 & 0 & 1 & 0 & 0\\
0 & 0 & 0 & 0 & 0 & 0 & 0 & 1 & 0\\
0 & 0 & 0 & 0 & 0 & 0 & 0 & 0 & q^{-m}
\end{pmatrix}, \]
which is the standard coquasitriangular structure on the generators of $c_{p}[SL_3]$ given in \cite{Primer} when specialised to the root of unity $p=q^{-m}$. \end{proof}

\begin{remark}\label{work with Cq[GL_2]}
In the case (2) of the theorem above, we can identify $\widetilde{\cc_q[SL_2]}\cong c_{q^{-m}}[GL_2]$ by sending the four matrix generators of the latter to $\atild=a\varsigma^{m\beta},\btild=b\varsigma^{m\beta},\ctild=c\varsigma^{m\beta},\dtild=d\varsigma^{m\beta}$. The $q$-determinant $D$ maps to $\varsigma^{2m\beta}$. The converse direction is clear since  $\beta$ is invertible mod $n$ when $n>3$ and not divisible by 3, so we can write $\varsigma =D^{\frac{1}{2m\beta}}$. \end{remark}

\begin{example}
At $q^3=1$, $\cc_q^2$ is already a braided-Hopf algebra in the category of $\cc_q[SL_2]$-comodules without a central extension. Therefore we can apply Theorem~\ref{codbos} and obtain a Hopf algebra, which we denote $\cc_q'[SL_3]$, generated by $X_i,Y_i,a,b,c,d$ with the additional cross relations and coproducts
\[\begin{pmatrix}
a & b\\ c & d
\end{pmatrix} X_1 = \begin{pmatrix}
q^{-1} X_1 a & q^{-1}X_1 b \\ qX_1 c & q X_1 d
\end{pmatrix}, \quad \begin{pmatrix}
a & b\\ c & d
\end{pmatrix} X_2 = \begin{pmatrix}
q X_2 a + (q^{-1}-1)X_1 c & qX_2 b + (q^{-1}-1)X_1 d \\ q^{-1}X_2 c & q^{-1} X_2 d
\end{pmatrix},\]
\[Y_1 \begin{pmatrix}
a & b\\ c & d
\end{pmatrix} = \begin{pmatrix}
q a Y_1 & q^{-1}b Y_1\\ q c Y_1 & q^{-1}d Y_1
\end{pmatrix}, \quad Y_2 \begin{pmatrix}
a & b\\ c & d
\end{pmatrix} = \begin{pmatrix}
q^{-1} a Y_2 + (q-1)b Y_1 & qb Y_2\\ q^{-1} c Y_2+(q-1)d Y_1 & qd Y_2
\end{pmatrix},\]
	\begin{align*}
	\Delta X_1&= X_1\otimes 1 + a\otimes X_1 + b\otimes X_2 + \lambda aY_1\otimes X_1^2 + \lambda bY_2 \otimes X_2^2 + \lambda aY_2\otimes X_1X_2 + \lambda q^2 bY_1\otimes X_1X_2\\
	&+\lambda^2 aY_2^2 \otimes X_1X_2^2 + \lambda^2 q bY_1^2\otimes X_1^2X_2 + \lambda^2 [2]_{q}aY_1Y_2\otimes X_1^2X_2 + \lambda^2 [2]_{q} q^2 bY_1Y_2 \otimes X_1X_2^2,\\
	\Delta Y_1&=1\otimes Y_1 + Y_1\otimes a + Y_2\otimes c + \lambda q^2 Y_1^2\otimes X_1a + \lambda q^2 Y_2^2\otimes X_2c + \lambda q Y_1Y_2\otimes X_2a+ \lambda q^2 Y_1Y_2\otimes X_1c\\
	&+\lambda^2 q^2 Y_1Y_2^2\otimes X_2^2a + \lambda^2 [2]_{q}q Y_1Y_2^2\otimes X_1X_2c + \lambda^2 [2]_{q}Y_1^2Y_2\otimes X_1X_2a + \lambda^2 q Y_1^2 Y_2 \otimes X_1^2c,\\
	\Delta a &=a\otimes a + b\otimes c +\lambda [2]_{q}q^2 aY_1\otimes X_1a + \lambda q aY_2\otimes X_2a + \lambda q^2 aY_2\otimes X_1c + \lambda [2]_{q} q^{-1} bY_2\otimes X_2c \\
	&+\lambda bY_1\otimes X_2a + \lambda q bY_1\otimes X_1c + \lambda^2 q^2 aY_2^2 \otimes X_2^2a + \lambda^2 q^2 bY_1^2 \otimes X_1^2c + \lambda^2 [2]_{q} q aY_2^2\otimes X_1X_2c,\\
	&+\lambda^2 [2]_{q} q bY_1^2\otimes X_1X_2a + \lambda^2 [2]_q^2 aY_1Y_2 \otimes X_1X_2a + \lambda^2[2]_q^2 bY_1Y_2\otimes X_1X_2c \\
	&+ \lambda^2 [2]_{q} q aY_1Y_2\otimes X_1^2c + \lambda^2 [2]_{q}q bY_1Y_2\otimes X_2^2a,
	\end{align*}
	and similarly for the remaining coproducts. Here $\lambda=q-1$. This $\cc_q'[SL_3]$ is dual to $\cu_q'(sl_3)$ in Example~\ref{dbos u_q(sl_3) at q^3=1} and it is not isomorphic to $c_{q^{-1}}[SL_3]$, but rather to a sub-Hopf algebra  by $\phi:\cc_q'[SL_3]\hookrightarrow c_{q^{-1}}[SL_3]$ with 
\[\phi(X_1)=t^1{}_3(t^3{}_3)^{-1}, \quad \phi(X_2)=t^2{}_3(t^3{}_3)^{-1}, \quad \phi(Y_1)=-\dfrac{t^3{}_1(t^3{}_3)^{-1}}{\lambda}, \quad \phi(Y_2)=-\dfrac{t^3{}_2(t^3{}_3)^{-1}}{\lambda},\]
\[\phi(a) = t^1{}_1(t^3{}_3)^{-1}-t^1{}_3(t^3{}_3)^{-1}t^3{}_1(t^3{}_3)^{-1}, \quad \phi(b) = t^1{}_2(t^3{}_3)^{-1}-t^1{}_3(t^3{}_3)^{-1}t^3{}_2(t^3{}_3)^{-1}, \]
\[\phi(c) = t^2{}_1(t^3{}_3)^{-1}-t^2{}_3(t^3{}_3)^{-1}t^3{}_1(t^3{}_3)^{-1}, \quad \phi(d) = t^2{}_2(t^3{}_3)^{-1}-t^2{}_3(t^3{}_3)^{-1}t^3{}_2(t^3{}_3)^{-1}. \]

Moreover, $\cc_q'[SL_3]$ is a coquasitriangular Hopf algebra by Lemma~\ref{codouble coquasitriangularity}. Writing $s^1{}_1=a, s^1{}_2=b, s^2{}_1=c,s^2{}_2=d$ for the matrix form of the generators of $\cc_q[SL_2]$, the coquasitriangular structure of $\cc_q'[SL_3]$ comes out as 
\[\CR(s^i{}_j,s^k{}_l)=R^i{}_j{}^k{}_l,\quad \CR(X_i,Y_j)=-\delta_{i,j}, \quad \CR(X_i,X_j)=\CR(Y_i,Y_j)=\CR(Y_i,X_j)=0,\]
\[\CR(X_i,s^j{}_k)=\CR(Y_i,s^j{}_k)=\CR(s^i{}_j,X_k)=\CR(s^i{}_j,Y_k)=0,\]
\[\CR(X_is^j{}_k,s^u{}_vY_w)=-\delta_{w_1,i}R^j{}_{j_1}{}^{w_1}{}_w R^{j_1}{}_k{}^u{}_v, \quad \CR(s^i{}_jY_k, X_u s^v{}_w)=0,\]
where $R$ is as in (\ref{RSL2}) with $m=1$. 
Theorem~\ref{codbos cq[SL_3]} (1) still applies  at $q^3=1$ with $\beta=0$ giving that $\varsigma$ is central and group-like in $\cc_q[SL_3]$ and  that $\cc_{q}[SL_3]\cong \cc_q'[SL_3]\otimes \C_q[\varsigma]/(\varsigma^3-1)$.
\end{example}

\subsection{Fermionic version of $\C_q[SL_3]$}\label{fermionic}
Here we similarly apply co-double bosonisation but this time to obtain a part-fermionic version of $\C_q[SL_3]$ by using the fermionic quantum-braided plane. We no longer work at roots of unity but rather with $q$ generic and also, in the spirit of Remark~\ref{work with Cq[GL_2]}, we take as our middle Hopf algebra  $A = \C_q[GL_2]$, the coquasitriangular Hopf algebra generated by $\atild, \btild, \ctild, \dtild$ with the same $q$-commutation relations and coalgebra structure as $\C_{q}[SL_2]$, but with $D=\atild\dtild-q^{-1}\btild\ctild=\dtild\atild-q\btild\ctild$ inverted. The antipode and coquasitriangular structure are given in matrix form  by
\[S\begin{pmatrix}
\atild & \btild \\
\ctild & \dtild
\end{pmatrix} = D^{-1}\begin{pmatrix}
\dtild & -q\btild\\
-q^{-1}\ctild & \atild
\end{pmatrix}, \quad R = -q^{-1} \begin{pmatrix}
q & 0 & 0 & 0\\ 0 & 1 & q-q^{-1} & 0\\ 0 & 0 & 1 & 0\\ 0& 0 & 0 & q
\end{pmatrix}.\]
In fact the normalisation of $R$ here can be chosen freely (there is a 1-parameter family of such quasitriangular structures on this Hopf algbra) which we have fixed so that we have $B=\C_q^{0|2} \in {}^{A}\CM$ as a fermionic quantum-braided plane generated by $e_1, e_2$  with the relations and coproduct and braiding
\[e_i^2=0, \quad e_2e_1+q^{-1}e_1e_2=0, \quad \underline{\Delta} e_i = e_i\otimes 1 + 1\otimes e_i,\quad \underline{\epsilon}e_i=0, \quad \underline{S}e_i=-e_i,\]
\[\Psi(e_i\otimes e_i)=-e_i\otimes e_i, \quad \Psi(e_1\otimes e_2)=-q^{-1}e_2\otimes e_1, \quad \Psi(e_2\otimes e_1)=-q^{-1}e_1\otimes e_2 - (1-q^{-2})e_2\otimes e_1.\]
This has a left $\C_{q}[GL_2]$-coaction as in (\ref{braided-plane left coaction}). Similarly, its dual $B^* = 
(\C_q^{0|2})^*$ lives in the category of right $\C_q[GL_2]$-comodules with coaction as in (\ref{braided-plane rigt coaction}). 

\begin{proposition}
Let $q\in \C^*$ not be a root of unity. The co-double bosonisation $B^{\op}\lbiprod A \rbiprod B^*$ with the above $B,A,B^{*}$ is a coquasitriangular Hopf algebra $\C_q^{fer}[SL_3]$ generated by $e_i, f_i$ for $i=1,2$ and $\atild,\btild, \ctild, \dtild,D,D^{-1}$, with  cross relations and coproducts
\[f_i e_j = e_j f_i, \quad \begin{pmatrix} \atild & \btild \\ \ctild & \dtild \end{pmatrix}e_1 = \begin{pmatrix} -e_1 \atild & -e_1 \btild \\ -qe_1\ctild & -qe_1\dtild \end{pmatrix},\]
\[  \begin{pmatrix} \atild & \btild \\ \ctild & \dtild\end{pmatrix}e_2 = \begin{pmatrix} -qe_2\atild-(1-q^2)e_1 \ctild & -q e_2\btild - (1-q^2)e_1\dtild \\ -e_2\ctild & -e_2\dtild  \end{pmatrix}, \]
\[ f_1 \begin{pmatrix} \atild & \btild \\ \ctild & \dtild\end{pmatrix} = \begin{pmatrix} -\atild f_1 & -q^{-1}\btild f_1 \\ -\ctild f_1 & -q^{-1}\dtild f_1 \end{pmatrix}, \quad f_2 \begin{pmatrix} \atild & \btild \\ \ctild & \dtild \end{pmatrix} = \begin{pmatrix} -q^{-1}\atild f_2 - (1-q^{-2})\btild f_1 & -\btild f_2\\ -q^{-1}\ctild f_2 - (1-q^{-2})\dtild f_1 & -\dtild f_2 \end{pmatrix}, \]
\[\Delta e_1 = e_1 \otimes 1 + \atild \otimes e_1 + \btild\otimes e_2 +(1-q^{-2})(q^{-1}\btild f_1 - \atild f_2)\otimes e_1 e_2,\]
\[\Delta e_2 = e_2 \otimes 1 + \ctild\otimes e_1 + \dtild \otimes e_2 +(1-q^{-2})(q^{-1}\dtild f_1 - \ctild f_2)\otimes e_1 e_2,\]
\[\Delta f_1 = 1 \otimes f_1 + f_1 \otimes \atild + f_2 \otimes \ctild +(1-q^2)f_1f_2\otimes (e_1\ctild-q^{-1}e_2\atild),\]
\[\Delta f_2 = 1 \otimes f_2 + f_1 \otimes \btild + f_2 \otimes \dtild +(1-q^2)f_1f_2\otimes (e_1\dtild-q^{-1}e_2\btild),\]
\[\Delta \atild = \atild\otimes \atild + \btild\otimes \ctild + (q-q^{-1})(\btild f_1-q\atild f_2)\otimes(e_1\ctild-q^{-1}e_2\atild),\]
\[\Delta \btild = \atild\otimes \btild + \btild \otimes \dtild + (q-q^{-1})(\btild f_1-q\atild f_2)\otimes(e_1\dtild-q^{-1}e_2\btild),\]
\[\Delta \ctild = \ctild\otimes \atild + \dtild\otimes \ctild + (q-q^{-1})(\dtild f_1-q\ctild f_2)\otimes(e_1\ctild -q^{-1}e_2\atild),\]
\[\Delta \dtild = \ctild \otimes \btild + \dtild\otimes \dtild + (q-q^{-1})(\dtild f_1-q\ctild f_2)\otimes(e_1\dtild-q^{-1}e_2\btild).\]
\end{proposition}
\begin{proof}
First note that
\[\CR(S\atild,\atild)= \CR(S\dtild,\dtild)=-1, \quad \CR(S\atild,\dtild)=\CR(S\dtild,\atild)=-q, \quad \CR(S\btild,\ctild)=-(1-q^{2})\]
and zero on other cases of this form. Then  the inverse braiding is
\[\Psi^{-1}(e_1 \otimes e_2) = \CR(Se_{1}^{\bar{\o}},e_2^{\bar{\o}})e_2^{\bar{\infi}}\otimes e_1^{\bar{\infi}} = -q e_2 \otimes e_1 - (1-q^{2})e_1 \otimes e_2,\]
\[\Psi^{-1}(e_2\otimes e_1)= \CR(Se_2^{\bar{\o}}, e_1^{\bar{\o}})e_1^{\bar{\infi}}\otimes e_2^{\bar{\infi}} = -qe_1 \otimes e_2,\]
with the result that $\bar{S}(e_1\dotop e_2)=q^{2} e_1\dotop e_2$ and  $e_2 \dotop e_1 + q^{-1} e_1 \dotop e_2 = 0$ in $B^{\op}$. We now apply the co-double bosonisation theorem. It is easy to see that $f_ie_j \equiv (1\otimes 1 \otimes f_i)(e_j \otimes 1 \otimes 1)=e_j\otimes 1 \otimes 1 \equiv e_jf_i$. Next, we compute that for any $s^i{}_j \in \C_q[GL_2]$, where $s^1{}_1=\atild, s^1{}_2=\btild, s^2{}_1=\ctild, s^2{}_2=\dtild$, 
\[s^i{}_j e_k = e_k^{\bar{\infi}}(s^i{}_j)\t \CR(S(s^i{}_j)\o ,e_k^{\bar{\o}}) = \sum\limits_{l=1}^{2}e_k^{\bar{\infi}}s^l{}_j \CR(Ss^i{}_l,e_k^{\bar{\o}}),\]
\[f_k s^i{}_j = (s^i{}_j)\o f_{k}^{\bar{\z}}\CR(f_{k}^{\bar{\o}}, (s^i{}_j)\t) = \sum\limits_{l=1}^{2} s^i{}_l f_{k}^{\bar{\z}} \CR(f_{k}^{\bar{\o}}, s^l{}_j),\]
which comes out as the stated cross relations. Now let  
\[\{ e_a\}=\{1,e_1,e_2,e_1e_2\}, \quad \{f^a \}= \{1,f_1,f_2, f_1f_2 \}\]
be a basis and dual basis of $B,B^*$ respectively. Then
\[\Delta e_i =  e_i \otimes 1 +\sum\limits_{a=1}^{2} e_i^{\bar{\o}}\o f^a\otimes (e_a\underline{\o}^{\bar{\infi}}\dotop e_i^{\bar{\infi}}\dotop \bar{S}e_a\underline{\t}) \mathcal{R}(e_a\underline{\o}^{\bar{\o}}, e_i^{\bar{\o}}\t),\]
\[\Delta f_i =  1 \otimes f_i +\sum\limits_{a=1}^{2} f^a\otimes(e_a\underline{\o}\dotop \bar{S}e_a\underline{\th}^{\bar{\infi}})f_{i}^{\bar{\o}}\t \CR(Sf_{i}^{\bar{\o}}\o, e_a\underline{\th}^{\bar{\infi}}) \langle f_i, e_a\underline{\t} \rangle,\]
\[\Delta s^i{}_j = \sum\limits_{a,k,l,r=1}^{2}s^i{}_k f^{a}\otimes (e_a\underline{\o}^{\bar{\infi}}\dotop \bar{S}e_a\underline{\t}^{\bar{\infi}})s^r{}_j \CR(e_a\underline{\o}^{\o}, s^k{}_l)\CR(Ss^l{}_r, e_a\underline{\t}^{\o}),\]
which comes out as stated for all $i,j\in\{1,2\}$. Finally, we let 
\[s_1 = (q-q^{-1})(\btild f_1-q\atild f_2), \quad s_2 = (q-q^{-1})(\dtild f_1-q\ctild f_2),\]
\[t_1 = e_1\ctild-q^{-1}e_2\atild, \quad t_2 = e_1 \dtild -q^{-1}e_2\btild, \]
and write $\C_q^{fer}[SL_3]$ as having a matrix of generators $t^i{}_j$, where now $i,j\in \{1,2,3\}$, by
\[\mathbf{t}=\begin{pmatrix}
t^1{}_1 & t^1{}_2 & t^1{}_3\\
t^2{}_1 & t^2{}_2 & t^2{}_3\\
t^3{}_1 & t^3{}_2 & t^3{}_3
\end{pmatrix} = \begin{pmatrix}
X & t_1 & t_2\\
s_1 & \atild & \btild\\
s_2 & \ctild & \dtild
\end{pmatrix}; \]
\begin{align}\label{elementX}
X=D + (t_1D^{-1}(\dtild s_1-q \btild s_2)-q^{-1} t_2D^{-1}(\ctild s_1-q \atild s_2)).
\end{align}
Here $D$  obeys $Dt_i=qt_iD$ and $Ds_i=qs_iD$ for $i=1,2$. The coproduct now has the standard matrix form  $\Delta \mathbf{t}=\mathbf{t}\otimes \mathbf{t}$ and in these terms the quadratic relations  are
\[(t^1{}_2)^2=(t^1{}_3)^2=(t^2{}_1)^2=(t^3{}_1)^2=0,\]
\[[t^1{}_2,t^1{}_1]_{q^{-1}}=[t^1{}_3,t^1{}_1]_{q^{-1}}=[t^2{}_1,t^1{}_1]_{q^{-1}}= [t^3{}_1,t^1{}_1]_{q^{-1}}=[t^2{}_3,t^2{}_2]_{q}=[t^3{}_2,t^2{}_2]_{q}=0,\]
\[[t^3{}_3,t^2{}_3]_{q}=[t^3{}_3,t^3{}_2]_{q}=[t^2{}_1,t^1{}_2]=[t^2{}_1,t^1{}_3]=[t^3{}_1,t^1{}_2]=[t^3{}_1,t^1{}_3]=[t^3{}_2, t^2{}_3]=0,\]
\[[t^2{}_2,t^1{}_1]=-\lambda t^1{}_2 t^2{}_1, \quad [t^2{}_3,t^1{}_1]=-\lambda t^1{}_3 t^2{}_1, \quad [t^3{}_2,t^1{}_1]=-\lambda t^1{}_2 t^3{}_1,\]
\[[t^3{}_3,t^1{}_1]=-\lambda t^1{}_3 t^3{}_1, \quad [t^3{}_3,t^2{}_2]= \lambda t^2{}_3 t^3{}_2, \quad \{t^1{}_3,t^1{}_2 \}_{q^{-1}} = \{t^3{}_1,t^2{}_1 \}_{q^{-1}}=0,\]
\[\{t^2{}_2,t^1{}_2\}_{q}=\{t^2{}_2,t^2{}_1\}_{q}=\{t^2{}_3,t^1{}_3\}_{q}=\{t^2{}_3, t^2{}_1\}_{q}= \{t^3{}_2,t^1{}_2\}_{q}=\{t^3{}_2,t^3{}_1\}_{q}=0,\]
\[\{t^3{}_3,t^1{}_3\}_{q}=\{t^3{}_3, t^3{}_1\}_{q}=\{t^2{}_2,t^1{}_3\}=\{t^3{}_1,t^2{}_2\}=\{t^3{}_1,t^2{}_3\}=\{t^3{}_2, t^1{}_3\}=0,\]
\[\{t^2{}_3,t^1{}_2\}=-\lambda t^1{}_3 t^2{}_2, \quad \{t^3{}_3,t^1{}_2\}=-\lambda t^1{}_3 t^3{}_2, \quad \{t^3{}_2,t^2{}_1\}= \lambda t^2{}_2 t^3{}_1,\quad \{t^3{}_3,t^2{}_1\}=\lambda t^2{}_3 t^3{}_1,\]
where $[\ ,\ ]_{q}$ is as before, similarly $\{a,b\}_{q}=a b+q b a$ for any $a,b$ is the $q$-anti-commutator, and $\lambda =q-q^{-1}$. Using Lemma \ref{codouble coquasitriangularity}, the values $\CR(t^i{}_j,t^k{}_l)$ of the coquasitriangular structure of $\C_q^{fer}[SL_3]$ come out, in the same conventions as in the proof of part (2) of Theorem~\ref{codbos cq[SL_3]}, as
\[ R^I{}_J=\begin{pmatrix}
q^{-2} & 0 & 0 & 0 & 0 & 0 & 0 & 0 & 0\\
0 & q^{-1} & 0 & -q^{-1}\lambda & 0 & 0 & 0 & 0 & 0\\
0 & 0 & q^{-1} & 0 & 0 & 0 & -q^{-1}\lambda & 0 & 0 \\
0 & 0 & 0 & q^{-1} & 0 & 0 & 0 & 0 & 0\\
0 & 0 & 0 & 0 & -1 & 0 & 0 & 0 & 0\\
0 & 0 & 0 & 0 & 0 & -q^{-1} & 0 & -q^{-1}\lambda &0\\
0 & 0 & 0 & 0 & 0 & 0 & q^{-1} & 0 & 0\\
0 & 0 & 0 & 0 & 0 & 0 & 0 & -q^{-1} & 0\\
0 & 0 & 0 & 0 & 0 & 0 & 0 & 0 & -1
\end{pmatrix}.\]

Note that since $t^1{}_1$ was defined in terms of the other generators including the $q$-sub-determinant $D=t^2{}_2 t^3{}_3-q^{-1}t^2{}_3 t^3{}_2$, there are in fact only $8$ algebra generators and 28 $q$-(anti)commutation relations other than the nilpotency ones and those involving $t^1{}_1$, putting this conceptually on a par with $\C_q[SL_3]$. Instead of a cubic $q$-determinant relation, we can regard (\ref{elementX}) as the cubic-quartic relation
\[Dt^1{}_1-q t^1{}_2(t^3{}_3t^2{}_1-qt^2{}_3t^3{}_1)+t^1{}_3(t^3{}_2t^2{}_1-qt^2{}_2t^3{}_1)=D^2.\]
Also note that (\ref{quacom}) in the `R-matrix' form  $R^i{}_m{}^k{}_n t^m{}_j t^n{}_l=t^k{}_n t^i{}_m R^m {}_j{}^n{}_l$ (sum over repeated indices) encodes exactly the quadratic relations above for $\C_q^{fer}[SL_3]$ including the nilpotent ones. 
\end{proof}


\begin{thebibliography}{1}
\bibitem{EJB : bar} E.J.Beggs and S. Majid, Bar categories and star operations, \textit{Alg. Repn. Theory} \textbf{12} (2009) 103--152.
\bibitem{CL}C. de Concini and V. Lyubashenko, Quantum function algebras at roots of 1, \textit{Adv. Math.} \textbf{108} (1994) 205--262.
\bibitem{Dri} V. G. Drinfeld, {Quantum Groups}; in (ed. A. Gleason) \textit{Proceeding of the ICM} (AMS, 1987), pp. 798--820.
\bibitem{Jan} J. C. Jantzen, Lectures on Quantum Groups, \textit{AMS} (1996).
\bibitem{Jimbo} M. Jimbo, A $q$-difference analog of $U(g)$ and the Yang-Baxter equation, \textit{Lett. Math. Phys.} \textbf{10} (1985), 63--69.
\bibitem{Len}S. Lentner, A Frobenius homomorphism for Lusztig's quantum groups for arbitrary roots of unity, \textit{Commun. Contemp. Math.,} \textbf{18}, 1550040 (2016) 42pp. 
\bibitem{Lus} G. Lusztig, Introduction to Quantum Groups, \textit{Birkhauser}, 1993.
\bibitem{Lyu}V. Lyubachenko and S. Majid, Quantum Fourier transform, \textit{J. Algebra} \textbf{166} (1994) 506--528.
\bibitem{Ma:dou}S. Majid, Doubles of quasitriangular Hopf algebras, Comm. Algebra. 19 (1991) 3061--3073.
\bibitem{db} S. Majid, Double bosonisation of braided groups and construction of $U_{q}(\mathfrak{g})$, \textit{Math. Proc. Camb. Phil. Soc.} \textbf{125} (1999) 151--192.
\bibitem{db2} S.Majid, New quantum groups by double bosonisation, \textit{Czech. J. Phys.} \textbf{47} (1997) 79--90. 
\bibitem{Ma:Rem} S. Majid, Some remarks on braided reconstruction and the braided double, \textit{Czech J. Phys.} \textbf{48} (1998) 1447--1456.
\bibitem{Ma:blie}S. Majid,  Braided-Lie bialgebras, Pac. J. Math. 192 (2000) 329--356.
\bibitem{Foundation} S. Majid, Foundations of Quantum Group Theory, \textit{Cambridge University Press} (2000) 640pp.
\bibitem{Primer} S. Majid, A Quantum Group Primer, \textit{L.M.S. Lect. Notes} \textbf{292} (2002) 179pp.
\bibitem{Ma:Hod} S. Majid, Hodge star as braided Fourier transform, \textit{Alg. Repn. Theory} \textbf{20} (2017) 695--733.
\bibitem{AlgBr} S. Majid, Algebras and Hopf algebras in braided categories, \textit{Lec. Notes Pure and Applied Maths} \textbf{58} (1994) 55--105.
\bibitem{Braid} S. Majid, Braided groups, J. Pure Applied Algebra 86 (1993) 187--221.
\bibitem{Bos}S. Majid, Cross products by braided groups and bosonisation, \textit{J. Algebra} \textbf{163} (1994) 165--190.
\bibitem{Ma:skl} S. Majid, Braided matrix structure of the Sklyanin algebra and of the quantum Lorentz group. \textit{Commun. Math. Phys.} \textbf{156} (1993) 607--638.
\bibitem{Rad:str} D. Radford, The structure of Hopf algebras with a projection, \textit{J. Algebra} \textbf{92} (1985) 322--347.
\bibitem{NumberTheory} J.K. Strayer, Elementary Number Theory, \textit{Waveland Press, Inc.}, (1994) 290pp.
\bibitem{Woronowicz89} W.L. Woronowicz, Differential calculus on compact matrix pseudogroups (quantum groups), \textit{Commun. Math. Phys.}, 122 (1989) 125--170.
\end{thebibliography}
\end{document}